\documentclass[11pt]{article}

\PassOptionsToPackage{dvipsnames}{xcolor}
\usepackage{amssymb,amsmath,bm}
\usepackage{amsmath}
\usepackage{mathrsfs}
\usepackage{slashed}
\usepackage{amssymb}
\usepackage{amsthm}
\usepackage{stmaryrd}
\usepackage{graphicx,overpic}
\usepackage{svg}
\usepackage{epstopdf}
\usepackage{enumerate}
\usepackage{comment}
\usepackage{indentfirst}
\usepackage{setspace}  
\usepackage{float}
\usepackage{tikz-cd}

\usepackage{textcomp}
\usepackage{enumerate}      
\usepackage{graphicx} 
\usepackage{caption}
\usepackage{subcaption}

\usepackage{tabu}

\usepackage[export]{adjustbox}

\usepackage{mathrsfs}

\usepackage{url}

\usepackage{hyperref}

\def\psl{\mathrm{PSL}(2,\mathbb{R})}
\def\pgl{\mathrm{PGL}(2,\mathbb{R})}

\def\SL{\mathrm{SL}(2,\mathbb{R})}
\def\GL{\mathrm{GL}(2,\mathbb{R})}
\def\sl{\mathfrak{sl}_2(\mathbb{R})}
\def\fund{\pi_1(\Sigma)}
\def\hom{\mathrm{Hom}\big(\fund, \psl\big)}

\def\univcover{\widetilde{\SL}}
\def\Hyp{\mathrm{Hyp}}
\def\Par{\mathrm{Par}}
\def\Parp{\mathrm{Par}^+}
\def\Parm{\mathrm{Par}^-}
\def\Ell{\mathrm{Ell}}

\def\R{\mathcal{R}(\Sigma)}

\def\HP{\mathrm{HP}}

\def\HPns{\mathrm{HP}_n^s}

\def\nonabel{\mathcal{NA}^s_n}

\def\ev{\mathrm{ev}}
\def\tr{Tr}

\def\parpp
{\begin{bmatrix}
    1 & 1 \\
    0 & 1
\end{bmatrix}}

\def\parmp
{\begin{bmatrix}
    1 & -1 \\
    0 & 1
\end{bmatrix}}

\def\part
{\begin{bmatrix}
    1 & t \\
    0 & 1
\end{bmatrix}}

\def\ell
{\begin{bmatrix}
    \cos\frac{\theta}{2} & \sin\frac{\theta}{2} \\
    -\sin\frac{\theta}{2} & \cos\frac{\theta}{2}
\end{bmatrix}}

\def\path{\{\phi_t\}\interval}
\def\interval{_{t\in [0,1]}}

\def\ELns{\mathcal{EL}^s_n}
\def\NHns{\mathcal{NH}^s_n}

\def\Mns{\mathcal{M}^s_n}
\def\N{\mathcal{N}^s_n}
\def\pgd{\mathrm{PGD}_\gamma}

\newtheorem{theorem}{Theorem}[section]
\newtheorem{proposition}[theorem]{Proposition}
\newtheorem{lemma}[theorem]{Lemma}

\newtheorem{remark}[theorem]{Remark}

\theoremstyle{definition}

\theoremstyle{definition} 

\begin{document}
	
\title{\textbf{From Bowditch's question to Goldman's conjecture for type-preserving representations}}
\author{Viraj Joshi and Inyoung Ryu}
\maketitle

\begin{abstract}
    For punctured surfaces $\Sigma_{g,p}$ of genus $g\geqslant 2$,
    we explore the dynamics of the mapping class group action on the relative $\psl$-character varieties of type-preserving representations. For such relative character varieties, Goldman's conjecture predicts that the mapping class group acts ergodically on their non-Teichm\"uller components. A related question of Bowditch asks whether every non-elementary type-preserving representation that is non-Fuchsian sends some non-peripheral simple closed curve to a non-hyperbolic element. We show that, on the components of the relative character variety indexed by fixed signs of images of peripheral elements and relative Euler classes non-extremal in the generalized Milnor-Wood inequality, an affirmative answer to Bowditch's question implies Goldman's conjecture.

\end{abstract}
\tableofcontents

\section{Introduction}\label{sec_intro}


For a connected, oriented and closed surface $\Sigma_g$ of genus $g\geqslant 2$ and a representation $\rho:\pi_1(\Sigma_g)\rightarrow \mathrm{PSL}(2,\mathbb R)$, the \emph{Euler class} $e(\rho)$ is the Euler class of the associated flat $\psl$-bundle over $\Sigma_g$, which is an integer under the identification $\mathrm{H}^2\big(\Sigma;\pi_1(\psl)\big)\cong \mathbb Z.$
It was proved in \cite{milnor, wood} that the Euler class satisfies the Milnor--Wood inequality $$2-2g\leqslant e(\rho)\leqslant  2g-2.$$  In \cite{goldman}, Goldman proved that the equality $|e(\rho)| = 2g-2$ holds if and only if $\rho$ is Fuchsian, i.e. discrete and faithful. Moreover, he proved in the same paper that the Euler class defines a one-to-one correspondence between the connected components of Hom$(\pi_1(\Sigma_g),\psl)$ and the integers $n$ with $|n|\leqslant 2g - 2$. Recall that the \emph{$\psl$-character variety} of $\Sigma_g$ is the GIT (geometric invariant theory) quotient $$\text{Hom}\big(\pi_1(\Sigma_g),\psl\big)\sslash \psl,$$ where $\psl$ acts by conjugation. Since the Euler class is invariant under $\psl$-conjugation, it descends to a well-defined invariant on the character variety, characterizing its connected components. We denote by \(\mathcal M(\Sigma_g)\) the space of conjugacy classes of non-elementary representations (representations with Zariski-dense image), which is the smooth locus of the character variety. 
\smallskip

The mapping class group $\operatorname{MCG}(\Sigma_g)$ is the group of isotopy classes of orientation-preserving homeomorphisms of $\Sigma_g$. By the Dehn--Nielsen--Baer theorem, it is isomorphic to the group of orientation-preserving outer automorphisms $\operatorname{Out}^+\big(\pi_1(\Sigma_g)\big)$, which acts on $\mathcal{M}(\Sigma_g)$ by precomposition and preserves the Euler class. For $k\in \mathbb{Z}$, we denote by
$\mathcal{M}_k(\Sigma_g)$ the space of conjugacy classes of non-elementary representations of Euler class $k$. By the results of Goldman \cite{goldman}, the components $\mathcal{M}_{2-2g}(\Sigma_g)$ and $\mathcal{M}_{2g-2}(\Sigma_g)$ are identified, respectively, with the Teichm\"uller space of $\Sigma_g$ and with that of $\Sigma_g$ endowed with the opposite orientation. It is well known (see e.g., \cite[Theorem 12.2]{farb_margalit}) that $\operatorname{MCG}(\Sigma_g)$ acts properly discontinuously on  these extremal components. 
On the non-extremal components $\mathcal{M}_k(\Sigma_g), |k|<2g-2$, 
Goldman conjectured in \cite{goldman_conjecture}
that $\operatorname{MCG}(\Sigma_g)$ acts ergodically on $\mathcal{M}_k(\Sigma_g)$ with respect to the measure induced by the Goldman symplectic form \cite{goldman_measure}. 
\smallskip

A related longstanding question of Bowditch \cite[Question C]{Bowditch} asks whether every non-elementary non-Fuchsian representation sends some simple closed curve to a non-hyperbolic element. 
In \cite{Marche-Wolff}, Marché and Wolff investigated the relation between these two problems. They proved that, apart from the component $\mathcal{M}_0(\Sigma_2)$, an affirmative answer to Bowditch's question implies that Goldman's conjecture is true. For the closed surface \(\Sigma_2\) of genus \(2\), they further showed that every class in $\mathcal M_{1}(\Sigma_2)$ and $\mathcal M_{-1}(\Sigma_2)$ has a representative sending some simple closed curve to a non-hyperbolic element. Together with the implication above, this proved the ergodicity of the \(\operatorname{MCG}(\Sigma_2)\)-action on these two components.\smallskip

The goal of this paper is to investigate this relation for punctured surfaces. Let $\Sigma = \Sigma_{g,p}$ be a punctured surface of genus $g\geqslant 2$ with $p\geqslant 1$ punctures, with Euler characteristic $\chi(\Sigma) < 0$. We study \emph{type-preserving} representations, i.e. representations $\rho: \pi_1(\Sigma)\to \psl$ sending all the \emph{peripheral elements} (elements represented by loops around a puncture) of $\pi_1(\Sigma)$ to parabolic elements of $\psl.$ Let $\mathcal{R}(\Sigma)$ be the space of all type-preserving representations.
For each $\rho\in \mathcal{R}(\Sigma)$, its \emph{relative Euler class} $e(\rho)$ is defined in \cite{goldman} as an integer-valued invariant, and its \emph{sign} $s(\rho)$ is defined in \cite{ryu} as an element in $\{\pm 1\}^p$. See Section \ref{sec_preliminary} for details. 
For $n\in \mathbb{Z}$ and $s\in \{\pm 1\}^p$, we denote by
    $\mathcal{R}^s_n$ the subspace of $\mathcal{R}(\Sigma)$ consisting of type-preserving representations with relative Euler class $n$ and sign $s$. Moreover, we let $p_+(s)$ be the number of $+1$'s and let $p_-(s)$ be the number of $-1$'s in the components of $s.$ Yang and the second author showed in \cite{ryu} that, each connected component of $\R$ is exactly $\mathcal{R}^s_n$ where the pair $(n,s)$ satisfies the \emph{generalized Milnor-Wood inequality} 
    $$\chi(\Sigma) + p_+(s)\leqslant n\leqslant -\chi(\Sigma) - p_-(s).$$
    Now let $\mathcal{M}(\Sigma)$ be the space of conjugacy classes of non-elementary type-preserving representations, and denote by $\mathcal{M}^s_n$ the subspace of $\mathcal{M}(\Sigma)$ consisting of elements with representatives in $\mathcal{R}^s_n$. Up to a measure zero set, this is the smooth locus of the component of type-preserving relative character variety of Euler class $n$.
    Then $\mathcal{M}^s_n$ consists of the classes of Fuchsian representations if and only if $|n| = -\chi(\Sigma)$ \cite{goldman}, yielding an identification of such two components in $\mathcal{M}(\Sigma)$ with the Teichm\"uller space of $\Sigma$. 
    \smallskip
    
    By the Dehn--Nielsen--Baer theorem, the pure mapping class
    group $\operatorname{MCG}(\Sigma)$ is isomorphic to the group $\operatorname{Out}_c^+(\pi_1(\Sigma))$ of orientation-preserving outer automorphisms fixing each conjugacy class of peripheral elements in $\pi_1(\Sigma)$. Consequently, $\operatorname{MCG}(\Sigma)$  acts naturally on $\mathcal{M}(\Sigma)$,
    preserving the relative Euler class and the sign; in particular, it acts on each $\mathcal{M}^s_n$. 
    The analogues of Goldman's conjecture on ergodicity and Bowditch's question can also be formulated in this relative setting. 
    In Theorem \ref{mainresult}, we prove that for a fixed sign $s$ and 
    a relative Euler class $n$ non-extremal in the generalized Milnor-Wood inequality, i.e., except $ n = -\chi(\Sigma) - p_-(s)$ and $ n = \chi(\Sigma) + p_+(s)$,
    an affirmative answer to Bowditch's question  implies that Goldman's conjecture is true on $\mathcal M_n^s$. To state the results, for $n\in \mathbb{Z}$ and $s\in \{\pm 1\}^p$,
    we denote by $\mathcal{NH}^s_n$ the set of $[\rho]\in \mathcal{M}^s_n$ that have a representative sending some non-peripheral simple closed curve to a non-hyperbolic element. Then Bowditch's question asks whether $\mathcal{NH}^s_n = \mathcal{M}^s_n$ for all $|n|<-\chi(\Sigma)$.

    \begin{theorem}\label{mainresult}
    Let $\Sigma = \Sigma_{g,p}$ with $g\geqslant 2$ and $p\geqslant 1$. For $s\in \{\pm 1\}^p$ and $n\in \big\{\chi(\Sigma) + p_+(s) + 1,\dots, -\chi(\Sigma) - p_-(s)- 1\big\}$,
    the mapping class group $\operatorname{MCG}(\Sigma)$ acts ergodically on $\NHns$.
\end{theorem}

Thus, when $n$ lies in the range stated in Theorem \ref{mainresult}, showing ergodicity of the $\operatorname{MCG}(\Sigma)$-action on $\mathcal{M}^s_n$ is equivalent to showing $\mathcal{NH}^s_n$ is full measure in $\mathcal{M}^s_n$, hence an affirmative answer to Bowditch's question implies that Goldman's conjecture is true. The cases when $n = -\chi(\Sigma) -p_-(s)$ or $n = \chi(\Sigma) + p_+(s)$ (except when $p_-(s) = 0$ and $n = -\chi(\Sigma)$ or when $p_+(s) = 0$ and $n = \chi(\Sigma)$ where the corresponding components $\Mns$ are the copies of the Teichm\"uller space) remain to be explored.
\smallskip

Goldman's conjecture is investigated for several cases of genus 0 and 1. For $\Sigma=\Sigma_{1,1}$, see Goldman \cite{goldman_torus} and for $\Sigma= \Sigma_{0,4}$, see Yang \cite{yang}. 
For punctured spheres \(\Sigma=\Sigma_{0,p}\), certain components of the relative character varieties with elliptic boundary
holonomy are studied; see Maret ~\cite{maret}. The remaining cases of genus 0 and 1 deserve a further study.
\smallskip

\begin{remark}
Since punctured surfaces with at least two punctures have separating simple closed curves that bound a genus-zero subsurface of negative Euler characteristic, the techniques of Marché and Wolff do not directly apply. In this paper, we develop a new approach that overcomes this obstruction.
\end{remark}

\begin{remark}
In \cite{ryu2}, the second author identified non-Fuchsian type-preserving representations that send every non-peripheral simple closed curve to a hyperbolic element, and showed that the set of their conjugacy classes has full measure in certain components $\mathcal{M}^s_n$. The techniques developed in the proof of Theorem \ref{mainresult} also give an alternative proof of this result for punctured surfaces of genus $g\geqslant 2$. We state the result as Theorem \ref{thm_tothyp} and include its proof in Appendix for the interested readers.
\end{remark}


This paper is organized as follows. 
    In Section \ref{sec_preliminary}, we recall the definitions of the relative Euler class and the sign of representations with parabolic or hyperbolic boundary conditions, and state some results used throughout the paper.
    In Section \ref{sec_outline}, we introduce the required notation, and state the outline of the proof of Theorem \ref{mainresult}, which mainly follows from 
    Propositions \ref{EfullNH}, \ref{connect_el} and \ref{U}.
    In Sections \ref{sec_conn} and \ref{sec_U}, 
    we will prove Propositions \ref{connect_el} and \ref{U}, respectively. Finally, in Section \ref{sec_EfullNH}, we will prove Proposition \ref{EfullNH}.

\noindent\textbf{Acknowledgments.} 
The authors would like to thank their supervisor Tian Yang for introducing us to the works of March\'e, Wolff and Goldman. They are deeply grateful for his invaluable insights, numerous discussions, consistent guidance and support. The authors discussed the computations involved in the proof of Lemma \ref{lem_fourpuncsphere} with ChatGPT (GPT-5.5, OpenAI).
\section{Preliminaries}\label{sec_preliminary}

Let $\Sigma = \Sigma_{g,p}$ be an oriented punctured surface with genus $g$ and $p$ punctures, and let $\fund$ be the fundamental group of $\Sigma$ which has the presentation
\begin{equation*}
    \pi_1(\Sigma)=\big\langle a_1,b_1,...,a_g,b_g,c_1,...,c_p\ |\ [a_1,b_1][a_2,b_2]...[a_g,b_g]c_1c_2...c_p\big\rangle. 
    \end{equation*}
    A \emph{peripheral element} of $\fund$ 
    is an element represented by a closed curve freely homotopic to a circle around a puncture of $\Sigma$. Notice that each puncture of $\Sigma$ determines a conjugacy class of $\pi_1(\Sigma)$ consisting of the peripheral elements around it. Throughout this paper, we fix 
    peripheral elements $c_1,\dots,c_p$, one for each puncture of $\Sigma$, each represented by a simple loop around the corresponding puncture and oriented according to the surface orientation. We call these the \emph{preferred peripheral elements}.  Denote by $\mathrm{HP}(\Sigma)$ the subspace of $\hom$ consisting of representations $\rho:\pi_1(\Sigma) \rightarrow \psl$ such that for every $i\in \{1,\dots,p\}$, $\rho(c_i)$ is hyperbolic or parabolic.  
    \smallskip
    
    In this section, we recall the definition and basic properties of relative Euler classes and signs of $\psl$-representations in $\HP(\Sigma)$. We also review several previous results that will be used in the proof of our main result.
    In this paper, we will use the notation $g$ or $\pm A$ for $\psl$-elements, $A$ for $\SL$-elements, and $\widetilde{g}$ or $\widetilde{A}$ for the elements in the universal cover $\univcover$.

   \subsection{The relative Euler classes}\label{relE}

    For a closed surface $\Sigma$, each representation 
    $\rho\in \hom$ 
    determines  and is the holonomy representation of a flat principal $\psl-$bundle over $\Sigma$. 
    The Euler class of this bundle is an obstruction class in  $\mathrm{H}^2\big(\Sigma; \pi_1\big(\psl\big)\big)$ that measures the non-triviality of the principal bundle.
    This obstruction class defines the Euler class of the representation $\rho$, which can be considered as an integer under the isomorphism $\mathrm{H}^2\big(\Sigma; \pi_1\big(\psl\big)\big)\cong \mathbb{Z}.$ 
    For a punctured surface $\Sigma$, as $\mathrm{H}^2\big(\Sigma; \pi_1\big(\psl\big)\big) = 0,$ all flat principal bundles over $\Sigma$ are trivial. Therefore, boundary conditions are needed to obtain a nontrivial invariant of $\psl-$representations of $\pi_1(\Sigma).$
    For each representation $\rho \in \HP(\Sigma)$, 
    the corresponding flat bundle admits a well-defined special trivialization (see \cite[Section 3]{goldman}). The obstruction to extend this special trivialization on $\partial \Sigma$ to a trivialization on $\Sigma$ defines the \emph{relative Euler class} of $\rho$, which is an obstruction class in 
    $\mathrm{H}^2\big(\Sigma, \partial \Sigma ; \pi_1\big(\psl\big)\big)$ and can be considered as an integer under the isomorphism $\mathrm{H}^2\big(\Sigma, \partial \Sigma ; \pi_1\big(\psl\big)\big)\cong \mathbb{Z}$.
    \\

    For more details, let us recall that elements of $\psl$ are classified as hyperbolic, parabolic and elliptic as follows: For a non-trivial $\pm A$ in $\psl,$ let $A$ be one of its lifts in $\SL.$ Then $\pm A$ is hyperbolic if $|\tr(A)|>2,$ parabolic if $|\tr(A)|=2,$ and elliptic if $|\tr(A)|<2.$ 
We respectively let $\Hyp,$ $\Par$ and $\Ell$ be the spaces of hyperbolic, parabolic and elliptic elements of $\psl$. Then $\Hyp$ and $\Ell$ are connected subspaces of $\psl$, and $\Par$ has two connected components $\Parp$ and $\Parm$ which are respectively the $\psl$-conjugacy classes of $\pm \parpp$ and $\pm \parmp$. We call a parabolic element \emph{positive} if it lies in $\Par^+$, and \emph{negative} if it lies in $\Par^-$.
    
  Similarly, non-central elements of the universal covering $\univcover$ of $\psl$ are also classified as \emph{hyperbolic}, \emph{parabolic}, or 
    \emph{elliptic} according to the type of their projections to $\psl$. A parabolic element of $\univcover$ is \emph{positive} or \emph{negative} if its projection to $\psl$ is respectively so. As a convention, we define the \emph{trace} of an element in $\univcover$ to be the trace of its projection to $\SL.$ Then a non-central element in $\univcover$ is hyperbolic, parabolic, or elliptic if its trace is greater than, equal to, or less than $2$ in the absolute value.

    As $\psl$ is homeomorphic to the unit tangent bundle of the hyperbolic plane $\mathbb H^2$ which is an open solid torus, its universal covering 
    group $\univcover$ is  homeomorphic to $\mathbb{R}^3$, with the group of
    deck transformations 
    $\pi_1(\psl)\cong \mathbb{Z}$. 
The pre-image of each of $\Hyp$, $\Par^\pm$ and $\Ell$ in $\univcover$ have infinitely many connected components, indexed by the integers $\mathbb{Z}$. We denote these components as follows: Recall that the center of  $\univcover$ is  infinite cyclic consisting of the lifts of $\pm\mathrm I.$  Let $\widetilde{\exp}: \sl\to \univcover$ be the exponential map of $\univcover$, and let $z := \widetilde{\exp}\begin{bmatrix}
        0 & \pi \\
       -\pi & 0 
    \end{bmatrix}$. Then $z$ is a generator of the center $Z$ of $\univcover$ and for any $n\in\mathbb Z,$ the lift $ \widetilde{\exp}
    \begin{bmatrix}
        0 & n\pi \\
       -n\pi & 0 
    \end{bmatrix} $ of $\pm\mathrm{I}$  in $\univcover$ is equal to $z^n.$ 
(In \cite{goldman}, the generator $z$  is chosen as
    $\widetilde{\exp}
    \begin{bmatrix}
        0 & -\pi \\
       \pi & 0 
    \end{bmatrix}$, which is the inverse of ours.
    As a result, the relative Euler class defined there  
    is the negative of ours.) 
    Let  $\Hyp_0$ be the space of hyperbolic elements of $\univcover$ lying in the image  $\widetilde{\exp}(\sl)$ of the exponential map; 
    and for any $n\in\mathbb Z$, let  $\Hyp_n := z^n\Hyp_0$. 
    Then $\{\Hyp_n\}_{n\in \mathbb Z}$ are the connected components of the space of hyperbolic elements in $\univcover$; and two hyperbolic elements in $\univcover$ are conjugate if and only if they lie in the same connected component $\Hyp_n$  for some $n\in\mathbb Z$ and have the same trace.
    Similarly, let $\Par^+_0$ and $\Par^-_0$ respectively be the spaces of positive and negative parabolic elements of $\univcover$ lying in the image  of the exponential map, and let  $\Par_0
    = \Par^+_0\cup \Par^-_0.$ For any $n\in\mathbb Z,$ let $\Par^+_n = z^n\Par^+_0$ and $\Par^-_n = z^n\Par^-_0,$ and let  $\Par_n := \Par^+_n\cup \Par^-_n=z^n\Par_0.$ Then $\{\Par^+_n\}_{n\in \mathbb Z}$ and $\{\Par^-_n\}_{n\in \mathbb Z}$ are respectively the connected components of the space of positive and negative parabolic elements in $\univcover$; and two parabolic elements in $\univcover$ are conjugate if and only if they lie in the same connected component $\Par^+_n$ or $\Par^-_n$ for some $n\in\mathbb Z.$  Finally, for $n\in \mathbb{Z}$, we denote by $\overline{\Hyp_n}$ the closure of $\Hyp_n$ in $\univcover$, which is a union of 
    $\Hyp_n,$ $\Par_n$ and $\{z^n\}$.

Unlike hyperbolic and parabolic elements, all elliptic elements of $\univcover$ lie in the image of the exponential map; and we need to index their connected components differently. For $n>0$, let  $\Ell_n$ be the subspace of  $\univcover$  consisting of elements conjugate to $\widetilde{\exp}\begin{bmatrix}
        0 & \theta \\
       -\theta & 0 
    \end{bmatrix}$ for some $\theta$ in $((n-1)\pi, n\pi),$ and for $n<0$, let  $\Ell_n$ be the subspace of  $\univcover$ consisting of elements conjugate to $\widetilde{\exp}\begin{bmatrix}
        0 & \theta \\
       -\theta & 0 
    \end{bmatrix}$ for some $\theta$ in $(n\pi, (n+1)\pi).$
    Then $\{\Ell_n\}_{n\in \mathbb Z\setminus\{0\}}$ are the connected components of the space of elliptic  elements in $\univcover;$ and two elliptic elements in $\univcover$ are conjugate if and only if both lie in the same connected component $\Ell_n$ and have the same trace. See Figure \ref{fig: Universal_cover_Ell_colored2}.

\begin{figure}[hbt!]
        \centering
        \begin{overpic}[width=1.0\textwidth]{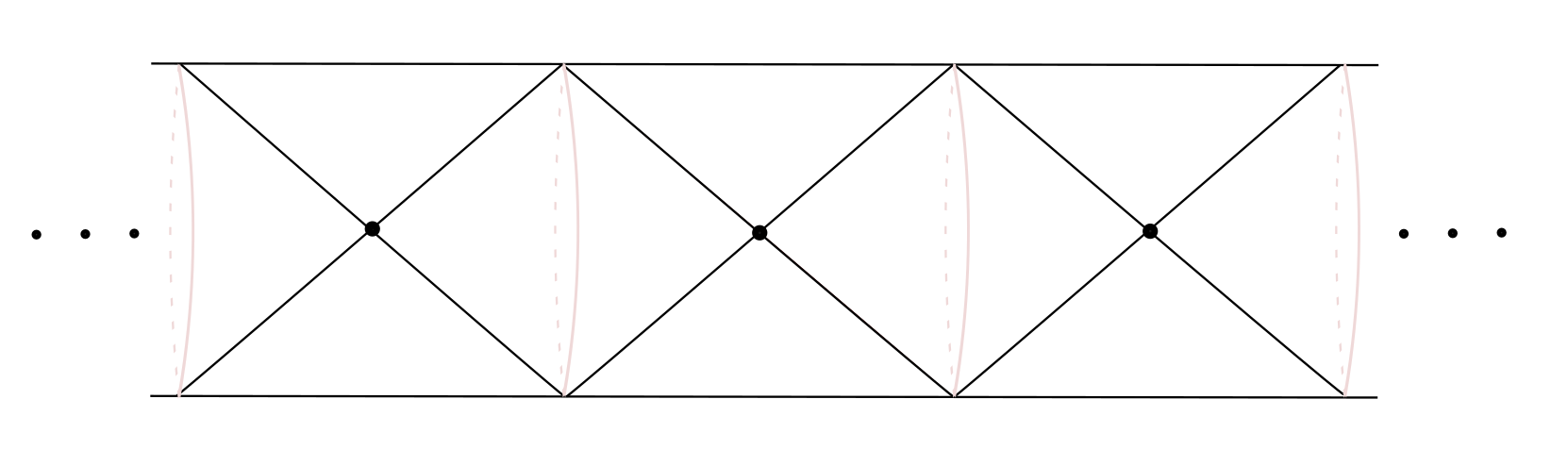}

            \put(22.0,17){\textcolor{black}{$z^{-1}$}}

            \put(48,17){\textcolor{black}{$\mathrm{I}$}}

            \put(72.5,17){\textcolor{black}{$z$}}

        \put(21,6){\textcolor{OliveGreen}{$\Hyp_{-1}$}}

            \put(46,6){\textcolor{OliveGreen}{$\Hyp_{0}$}}
            
            \put(71,6){\textcolor{OliveGreen}{$\Hyp_{1}$}}

            \put(14,23.5){\textcolor{blue}{$\Par^-_{-1}$}}

            \put(27,23.5){\textcolor{blue}{$\Par^+_{-1}$}}

            \put(39,23.5){\textcolor{blue}{$\Par^-_0$}}

            \put(53,23.5){\textcolor{blue}{$\Par^+_0$}}

            \put(64,23.5){\textcolor{blue}{$\Par^-_1$}}

            \put(78,23.5){\textcolor{blue}{$\Par^+_1$}}

        \put(10,14){\textcolor{red}{$\Ell_{-2}$}}
        
        \put(33,14){\textcolor{red}{$\Ell_{-1}$}}

        \put(59,14){\textcolor{red}{$\Ell_1$}}

        \put(84,14){\textcolor{red}{$\Ell_2$}}
        \end{overpic}
        \caption{\label{fig: Universal_cover_Ell_colored2} The universal covering $\univcover$}
    \end{figure}

    The relative Euler class of a representation $\rho$ in $\HP(\Sigma)$ can be defined as follows.  
    Since for each $i\in\{1,\dots, p\},$ the image $\rho(c_i)$ of the peripheral element $c_i$ is either hyperbolic or parabolic, there exists a unique lift 
    $\widetilde{\rho(c_i)}$  of $\rho(c_i)$ in $\univcover$ that lies in $\overline{\Hyp_0}$.  
 For $j \in \{1,\dots, g\}$, choose arbitrarily lifts 
    $\widetilde{\rho(a_j)}$ and $\widetilde{\rho(b_j)}$ of $\rho(a_j)$ and $\rho(b_j)$ 
    in $\univcover$; and 
    note that the commutator 
    $[\widetilde{\rho(a_j)}, \widetilde{\rho(b_j)}]$ is independent of the choice of the lifts. 
    Since $$[\rho(a_1), \rho(b_1)] \dots [\rho(a_g), \rho(b_g)] \rho(c_1)\dots\rho(c_p)=\pm\mathrm{I},$$ 
    the product $[\widetilde{\rho(a_1)}, \widetilde{\rho(b_1)}] \dots [\widetilde{\rho(a_g)}, \widetilde{\rho(b_g)}] \widetilde{\rho(c_1)}\dots\widetilde{\rho(c_p)}$ in $\univcover$ projects to 
    $\pm\mathrm{I}$ in $\psl$ under the covering map, hence lies in the center of $\univcover,$ i.e., there exists an $n\in\mathbb Z$ such that $$[\widetilde{\rho(a_1)}, \widetilde{\rho(b_1)}] \dots [\widetilde{\rho(a_g)}, \widetilde{\rho(b_g)}] \widetilde{\rho(c_1)}\dots\widetilde{\rho(c_p)} = z^n.$$
   The \emph{relative Euler class} $e(\rho)$ of the representation $\rho\in\mathrm{HP}(\Sigma)$ is defined as
   $e(\rho)=n.$
    \medskip
    
   The relative Euler class satisfies the following additivity property.
    
    \begin{proposition}\label{prop_additivity}\cite[Proposition 3.7]{goldman}
        Let $\Sigma = \Sigma_{g,p},$ let $\gamma$ be a separating simple closed curve on $\Sigma,$ and let $\Sigma_1$ and $\Sigma_2$ respectively be the two components of the complement $\Sigma\setminus \gamma.$ 
     Suppose a representation $\rho\in \HP(\Sigma)$ maps $[\gamma]\in \pi_1(\Sigma)$ to a hyperbolic or parabolic element of $\psl.$ Then $\rho|_{\pi_1(\Sigma_1)}\in \mathrm{HP}(\Sigma_1)$, $\rho|_{\pi_1(\Sigma_2)}\in \mathrm{HP}(\Sigma_2)$, and 
        $$e(\rho) = e\big(\rho|_{\pi_1(\Sigma_1)}\big) + e(\rho|_{\pi_1(\Sigma_2)}).$$
    \end{proposition}

    Recall that a \emph{holonomy representation} is a representation $\rho\in \HP(\Sigma)$ that is Fuchsian, i.e., discrete and faithful, and that the convex core \textbf{core}$\big(\mathbb{H}^2/\rho\big(\pi_1(\Sigma)\big)\big)$ is homeomorphic to $\Sigma$. 
    The following proposition determines the relative Euler classes of holonomy representations. 
    In the original statement, $\rho$ is assumed to send all peripheral elements to hyperbolic elements of $\psl$; however, since type-preserving representations of $\pi_1(\Sigma)$ are holonomy if and only if they are Fuchsian (see \cite{yang}, Appendix A), the proof works verbatim here.
    \begin{proposition}\label{prop_holonomy}\cite[Theorem 3.4]{goldman}
        Let $\rho$ be a representation in $\HP(\Sigma)$. Then its relative Euler class $e(\rho)$ satisfies $|e(\rho)| = -\chi(\Sigma)$ if and only if $\rho$ is a holonomy representation.
    \end{proposition}

    Recall that a $\psl$-representation is \emph{reducible} if there exists a line in $\mathbb{R}^2$ that is invariant under the action of  $\rho\big(\pi_1(\Sigma)\big)$, or equivalently, $\rho$ is conjugate to a representation into the 
    Borel subgroup of $\psl$ consisting of the projection of the upper-triangular matrices. A representation is \emph{irreducible} if it is not reducible. 
    The following proposition determines the relative Euler class of reducible representations in $\HP(\Sigma)$, where in the original statement $\rho$ is assumed to send all peripheral elements to hyperbolic elements,
    but the proof works verbatim here.
    \begin{proposition}\cite[Proposition 3.6]{goldman}\label{prop_reducible}
        If $\rho\in \HP(\Sigma)$ is reducible, then $e(\rho)=0$. In particular, if $\rho$ is abelian, then $e(\rho)=0.$
    \end{proposition}

   \subsection{The signs}\label{Sign}

    Let $\Sigma = \Sigma_{g,p}$ with the preferred peripheral elements $c_1, \dots, c_p$ of $\fund$. The \emph{sign} of a representation $\rho\in\HP(\Sigma)$ is a $p$-tuple $s(\rho)=(s_1,\dots,s_p)$ in $\{-1,0,+1\}^p$, where for each $i\in\{1,\dots,p\},$ 
    \begin{equation*}
   s_i = \left\{
    \begin{array}{rl}
       +1,  &  \text{if } \rho(c_i) \text{ is positive parabolic}\\
        0, & \text{if } \rho(c_i) \text{ is hyperbolic}\\ 
        -1, &  \text{if } \rho(c_i) \text{ is negative parabolic}.
    \end{array}\right.
    \end{equation*}

    Since the type of an element of $\psl$ is invariant under $\psl-$conjugation, the sign of $\rho$ is independent of the choice of the preferred peripheral elements $c_i$'s. Denote by $\HPns(\Sigma)$ the subspace of $\HP(\Sigma)$ consisting of representations with relative Euler class $n$ and sign $s.$ In the rest of this paper, for an $s\in \{-1,0,+1\}^p$, we respectively let  $p_+(s),$ $p_0(s)$ and $p_-(s)$ be the number of $1$'s, $0$,s and $-1$'s in the components of  $s$.
    \smallskip

    The following Theorem \ref{ryuhp} gives a criterion for the non-emptiness and connectedness of $\HPns(\Sigma)$, which also characterizes the connected components of $\mathcal{R}(\Sigma)$.

 \begin{theorem}\label{ryuhp}\cite[Theorem 1.1, Theorem 1.8]{ryu} 
        Let $\Sigma = \Sigma_{g, p}$ with $g\geqslant 1$ and $p\geqslant 1$. Let $n\in \mathbb{Z}$ and $s\in \{+1,0, -1\}^p$. Then the space $\HP_n^s(\Sigma)$ is non-empty if and only if the pair $(n,s)$ satisfies the generalized Milnor-Wood inequality: 
        $$\chi(\Sigma) + p_+(s)\leqslant n\leqslant -\chi(\Sigma) - p_-(s).$$
        Each non-empty  $\HP_n^s(\Sigma)$ above is connected.
    \end{theorem}

    The following proposition describes the behavior of the relative Euler class and the sign under the $\pgl\setminus\psl$-conjugations.
    \begin{proposition}\label{prop_pglpsl} 
    Let $h$ be an element of $\pgl\setminus \psl$; and for a representation $\rho\in\HP(\Sigma),$ let $h \rho h^{-1}:\pi_1(\Sigma)\to\psl$ be  defined by 
    $h \rho h^{-1}(c)=h \rho(c)h ^{-1}$
 for each $c\in\pi_1(\Sigma).$  Then $h \rho h^{-1}\in\HP(\Sigma)$ with the relative Euler class $e(h \rho h^{-1})=-e(\rho)$ and the sign $s(h \rho h^{-1})=-s(\rho).$    
    \end{proposition}
    Notice that 
    the relative Euler class and the sign are invariant under $\psl$-conjugations. Thus, we may denote by $\mathcal{M}^s_n$ the set of conjugacy classes of non-elementary type-preserving representations of sign $s\in \{\pm 1\}^p$ and relative Euler class $n$. 



\subsection{The path-lifting property and images of certain maps}\label{subsec_prelim}
We first review some results on the path-lifting property which will be used later in Section \ref{sec_conn}. 
 A smooth map $f: X\to Y$ between two smooth manifolds $X$ and $Y$ is said to satisfy the
    \emph{path-lifting property} if, 
    for every $x\in X$ and path $\{y_t\}_{t\in [0,1]}$ with $f(x) = y_0$, there exists a non-decreasing surjective map (a re-parametrization of the path) $\tau: [0,1]\to [0,1]$ and a path 
    $\{x_s\}_{s\in [0,1]}$ starting at $x_0 = x$ such that $f(x_s) = y_{\tau(s)}$ for $0\leqslant s\leqslant 1$. 
    We say a map $f:X\rightarrow Y$ satisfies the \textit{strong path-lifting property} if:
    \begin{enumerate}[(1)]
        \item For every path $\{y_t\}\interval \subset Y$ and every $x_0\in f^{-1}(y_0)$, we have a path $\{x_t\}\interval$ in $X$ starting at $x_0$ such that $f(x_t)=y_t$ for all $t\in [0,1]$, and
        \item 
        Every fiber of $f$ is path-connected.
    \end{enumerate}
    
    The following lemma provides a sufficient condition for a map to satisfy the strong path-lifting property.
    
    \begin{lemma}\cite[Lemma 1.4]{goldman}\label{lem_plp}
        Let $X, Y$ be smooth manifolds, and let $f: X\to Y$ be a smooth map. Suppose that 
        (a) $f$ is a submersion, and 
        (b) $f^{-1}(y)$ is path-connected for every $y\in Y$.
        Then $f$ satisfies the strong path-lifting property. Furthermore, 
        if $Y$ is path-connected, then $X$ is path-connected.
    \end{lemma}
    
  One  example of a map that we need to satisfy the path-lifting property is the following  \emph{character map}
    $\chi: \SL\times\SL\to \mathbb{R}^3$  defined  by 
    $$
    \chi(A,B) := 
    \big(
    Tr(A), 
    Tr(B), 
    Tr(AB)\big).
    $$

    \begin{proposition}
    \cite[Proposition 4.1, Theorem 4.3]{goldman}
    \label{prop_char}
        Let $\chi: \SL\times\SL\to \mathbb{R}^3$ be the character map, and let $\kappa: \mathbb{R}^3 \to \mathbb{R}$ 
        be the polynomial map defined by
        $$
        \kappa(x,y,z) := x^2 + y^2 + z^2 - xyz - 2 .
        $$
        Then 
        \begin{enumerate}[(1)]
            \item 
        For a pair $(A,B)$ in $\SL\times\SL$, $\kappa\big(\chi(A,B)\big) = \tr[A,B]$. Furthermore, if $A$ and $B$ are non-elliptic, 
        then $\kappa\big(\chi(A,B)\big) = 2$ if and only if the representation $\phi:\pi_1(\Sigma_{0,3})\to \SL$ determined by the pair $(A,B)$ is reducible.
        \item  If 
        $\kappa(x,y,z)\neq 2$
        and $\chi^{-1}(x,y,z)$ is non-empty, 
        $\chi^{-1}(x,y,z)$ in $\SL\times \SL$ 
        consists of one orbit of the $\GL$-conjugation, which is the union of two orbits of the $\SL$-conjugation.
        \end{enumerate} 
    \end{proposition}
    
    \begin{lemma}\cite[Theorem 4.3, Corollary 4.5 (b)]{goldman}
    \label{lem_plpchar}
        Let 
        $$\Omega 
        :=
        \big\{(A,B)\in \SL\times\SL\ |\ 
        [A,B]\neq \mathrm I
        \big\}.$$ Let  $\Gamma = [-2,2]^3\cap \kappa^{-1}([-2,2])$. Then the character map
        $$\chi: \Omega\to \mathbb{R}^3
        \setminus \Gamma
        $$ is surjective, and satisfies the path-lifting property.
         
    \end{lemma}
    
Another example of the maps satisfying the path-lifting property is the following \emph{lifted commutator} $\widetilde{R}: \psl \times \psl \to \univcover$
    defined for any pair of elements  $(\pm A,\pm B)$ of $\psl$ by 
    $$\widetilde{R}(\pm A, \pm B) = \widetilde{A} \widetilde{B} \widetilde{A}^{-1} \widetilde{B}^{-1},$$ where 
    $\widetilde{A}$ and $\widetilde{B}$ are respectively arbitrary lifts of $\pm A$ and $\pm B$ in $\univcover$.
    
        \begin{theorem}\cite[Theorem 7.1]{goldman}\label{thm_liftcommu}
            The image of the lifted commutator 
            $\widetilde{R}: \psl\times \psl \to \univcover$ equals 
            $$\mathcal{I} = 
            \{\mathrm I\}
             \cup
            \Hyp_0
            \cup 
            \mathrm{Par}_0
            \cup 
            \mathrm{Ell}_{-1}
            \cup
            \mathrm{Ell}_{1}
            \cup 
            \mathrm{Par}_{-1}^{+}
             \cup 
            \mathrm{Par}_{1}^{-}
            \cup
            \Hyp_{- 1}
            \cup
            \Hyp_{1}.
            $$
            Moreover, for 
            each $\widetilde{C}$ in the image $\mathcal{I}$, the fiber 
            $\widetilde{R}^{-1}(\widetilde{C})$ is connected.
        \end{theorem}

\begin{figure}[hbt!]
        \centering
        \begin{overpic}[width=0.8\textwidth]{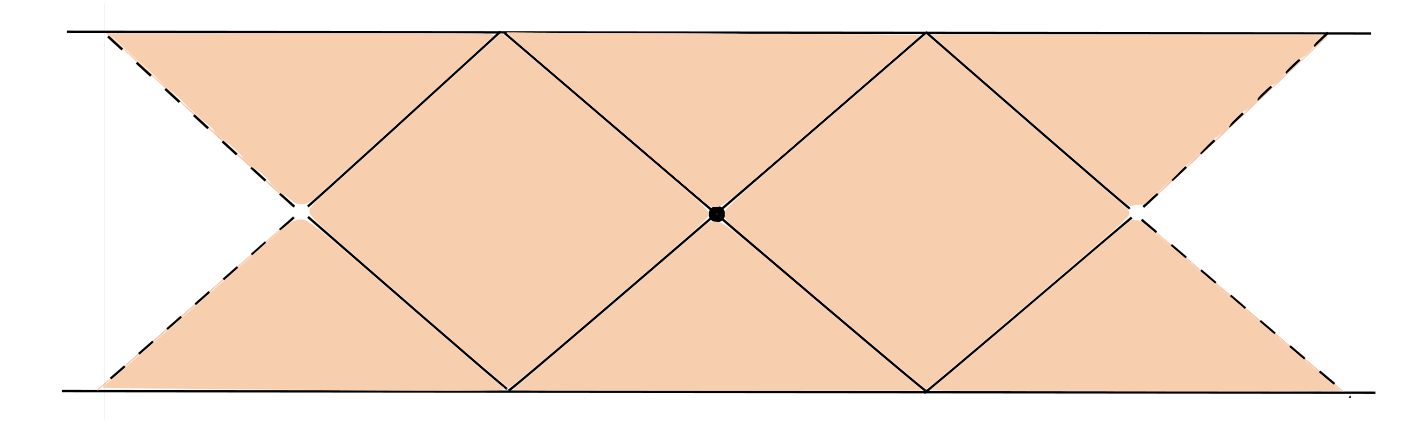}


            \put(49.8,17){\textcolor{black}{$\mathrm{I}$}}


        \put(19,4){\textcolor{OliveGreen}{$\Hyp_{-1}$}}

            \put(48,4){\textcolor{OliveGreen}{$\Hyp_{0}$}}
            
            \put(77,4){\textcolor{OliveGreen}{$\Hyp_{1}$}}


            \put(24,25){\textcolor{blue}{$\Par^+_{-1}$}}

            \put(39,25){\textcolor{blue}{$\Par^-_0$}}

            \put(56.5,25){\textcolor{blue}{$\Par^+_0$}}

            \put(69,25){\textcolor{blue}{$\Par^-_1$}}


        
        \put(33,14){\textcolor{red}{$\Ell_{-1}$}}

        \put(63,14){\textcolor{red}{$\Ell_1$}}

        \end{overpic}
        \caption{\label{fig: evimage} The image of the lifted commutator  and the lifted product maps}
    \end{figure}
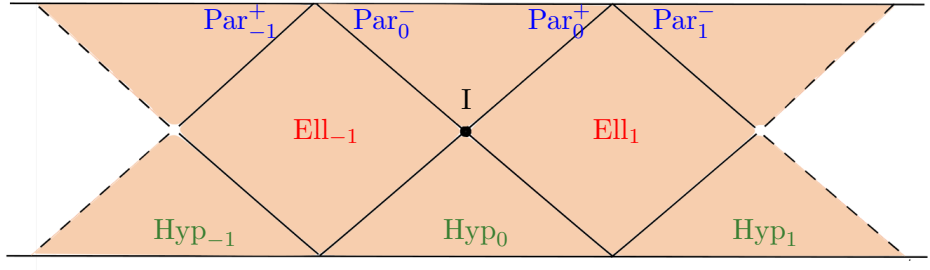

\begin{lemma}\label{lem_evTsubm}\cite[Lemma 2.9]{ryu}
    Let 
    $R: \SL\times \SL
    \to \SL$ defined by $R(A,B) = ABA^{-1}B^{-1}$. Then its differential $dR$ is surjective at the point $(A,B)\in \SL\times \SL$ if and only if $A$ and $B$ do not commute. As a consequence, the restriction of the lifted commutator $$\widetilde{R}: \widetilde{R}^{-1}(\mathcal{I} \setminus \{\mathrm I\}) \to \mathcal{I} \setminus 
            \{\mathrm I\}$$ 
    is a submersion.
\end{lemma}

As a consequence of Lemma \ref{lem_plp}, Theorem \ref{thm_liftcommu} and Lemma \ref{lem_evTsubm}, we have the following Proposition \ref{prop_plpliftcommu}.

        \begin{proposition}\label{prop_plpliftcommu}\cite[Proposition 2.10]{ryu}
            The lifted commutator 
            $\widetilde{R}: \widetilde{R}^{-1}(\mathcal{I} \setminus \{\mathrm I\}) \to \mathcal{I} \setminus 
            \{\mathrm I\}$ satisfies the strong path-lifting property.
        \end{proposition}

    Finally, we define the lifted product map as follows.
    Let $\overline{\Hyp}=\Hyp\cup\Par\cup\{\pm\mathrm{I}\}\subset \psl,$ and for $n\in\mathbb Z$, let $\overline{\Hyp_n}=\Hyp_n\cup\Par_n\cup\{z^n\}\subset\univcover$.
    The \emph{lifted product map}  
    $$\ev: \overline{\Hyp}\times \overline{\Hyp}\to \widetilde{\SL}$$
   maps the pair $(g_1,g_2)$ to the product $\widetilde{g_1}
    \widetilde{g_2}$ of the unique lifts $\widetilde{g_1}$ and $
    \widetilde{g_2}$ of  $g_1$ and $g_2$ in $\overline{\Hyp_0}$ which we call \emph{preferred lifts}. 
    Recall that the product $\widetilde{g_1}\widetilde{g_2}$ in the universal cover $\widetilde{\SL}$ is defined as follows. For $i=1,2,$ let $\{g_{i,t}\}\interval$ be a path connecting $g_{i,0} = \pm\mathrm{I}$ and $g_{i,1} = g_i$ in the one-parameter subgroup of $\psl$ generated by $g_i,$ and let $\{\widetilde{g_{1,t}g_{2,t}}\}\interval$ be the lift of the path $\{g_{1,t}g_{2,t}\}\interval$ in $\psl$ to  $\widetilde{\SL}$ starting from $\mathrm I.$ Then  $\widetilde{g_1}\widetilde{g_2}$ is the endpoint $\widetilde{g_{1,1}g_{2,1}}$ of the lifted path.
    \smallskip
    
    For a representation $\rho$ in $\HP(\Sigma_{0,3})$ with  $e(\rho)=n$, 
     let $\widetilde{\rho(c_i)},$  $i \in\{ 1,2,3\},$  be the preferred lift of $\rho(c_i)$. Then the image of the lifted product $\ev\big(\rho(c_1), \rho(c_2)\big)=\widetilde{\rho(c_1)}\widetilde{\rho(c_2)}$ equals $z^n\widetilde{\rho(c_3)}^{-1}$, hence lies in $\overline{\Hyp_n}$.

     \begin{proposition}\label{prop_evimage}\cite[Proposition 3.4]{ryu}
            The image of the lifted product map $\ev: \overline{\Hyp}\times \overline{\Hyp}\to \univcover$ equals 

            $$\mathcal{I} = 
            \Hyp_{- 1}
            \cup 
            \mathrm{Par}_{-1}^{+}
            \cup
            \mathrm{Ell}_{-1}
            \cup
            \mathrm{Par}_0
            \cup 
            \{\mathrm I\}
            \cup 
            \Hyp_0
            \cup
            \mathrm{Ell}_{1}
            \cup 
            \mathrm{Par}_{1}^{-}
            \cup
            \Hyp_{1}.
            $$
        Moreover, we  have:
        \begin{enumerate}[(1)]

            \item $\Hyp_0\cup \Hyp_{s}\subset \ev(\Hyp \times \Par^{sgn(s)})$ for $s\in\{\pm 1\}$.
            
            \item $\Hyp_{-1}\cup  \Hyp_0\cup\Hyp_1\cup \{\mathrm I\}\subset \ev(\Hyp \times \Hyp)$.

        \end{enumerate}
        \end{proposition}
Let us denote by $P:\univcover \times \univcover \rightarrow \univcover$ the map defined by group multiplication.
        \begin{proposition}
        \label{prop_prodimage}\cite[Theorem 2.6, Lemma 3.6]{ryu2}
        Let $P: \univcover\times \univcover\to \univcover$ be the product map sending $(\widetilde{A}, \widetilde{B})$ to $\widetilde{A}\widetilde{B}$, and let $s, s_1, s_2\in \{\pm 1\}$.
        Then the following holds:
        \begin{enumerate}[(1)] 



            \item $P\big(\Par_0\times\Ell_1 \big)\cap \big(\bigcup_{k\in \mathbb{Z}\setminus \{0\}}\Ell_k\big) = \Ell_1$.
            \item $P\big(\Par_0\times \Ell_{-1}\big)\cap \big(\bigcup_{k\in \mathbb{Z}\setminus \{0\}}\Ell_k\big)=\Ell_{-1} $
            \item $P\big(\Par^-_0\times\Par^-_0 \big)\cap \big(\bigcup_{k\in \mathbb{Z}\setminus \{0\}}\Ell_k\big) = \Ell_{-1}$.

            \item 
            $P\big(\Hyp_0 \times\Hyp_0\big)\cap \big(\bigcup_{k\in \mathbb{Z}\setminus \{0\}}\Ell_k\big) \subset \Ell_{-1}\cup \Ell_1$.

            \item $P\big(\Ell_{-1} \times\Ell_1\big)\cap \big(\bigcup_{k\in \mathbb{Z}\setminus \{0\}}\Ell_k\big) \subset \Ell_{-1}\cup \Ell_1$.
        \end{enumerate}
   \end{proposition}  

 \begin{proposition}\label{gol4.6}
  Consider the product map $$P:(\Hyp_0 \times \Hyp_0) \cap P^{-1}(\cup_{i = -1}^1 \Hyp_i) \to \cup_{i = -1}^1 \Hyp_i.$$ Then for every $\widetilde{H} \in \cup_{i = -1}^1 \Hyp_i$, the fiber $P^{-1}(\widetilde{H})$ is path-connected. 
 \end{proposition}
 \noindent For the proof of Proposition \ref{gol4.6}, see \cite[Proposition 4.6]{goldman}.

    \begin{lemma}\label{lem_evPsubm}\cite[Lemma 2.11]{ryu}
    Let $m: \SL\times \SL
    \to \SL$ be defined by $m(A,B) = AB$.
    Let $U$ be an open subset of $\SL$, and let 
    $\Par$ be the subset of $\SL$ consisting of the parabolic matrices.
    \begin{enumerate}[(1)]
        \item The restrictions of $m$ to 
        $U\times U$, $\Par\times U$, and $U\times \Par$ are submersions.
        \item At $(A,B)\in \Par\times \Par$, the differential $dm_{(A,B)}$ of the restriction   
    $m: \Par\times \Par
    \to \SL$ is surjective if and only if $A$ and $B$ do not commute. 
    \end{enumerate}
\end{lemma}

Lemma \ref{lem_plp} and Lemma \ref{lem_evPsubm} are used to show the path-lifting property of the lifted product map, which is stated in Proposition \ref{prop_plpliftcommu}.
The proof of Lemma \ref{parplp} is given in the Appendix.
    \begin{lemma}\label{parplp}
     The image of the lifted product map $\ev: \Par\times \Par\to \univcover$ is as follows:
     \[
        \begin{aligned}
        \ev(\Par^+\times \Par^+) &= \Par_0^+\cup \Ell_1 \cup \Par_1^-\cup \Hyp_1,\\
        \ev(\Par^+\times \Par^-) &= \Par_0\cup \Hyp_0 \cup \{\mathrm{I}\},\\
        \ev(\Par^-\times \Par^-) &= \Par_0^-\cup \Ell_{-1} \cup  \Par_{-1}^+\cup \Hyp_{-1}.
        \end{aligned}
    \]
 Two parabolic elements $\pm P_1$ and $\pm P_2$ commute if and only if $\ev(\pm P_1, \pm P_2)$ belongs to $\Par_0 \cup \{\mathrm{I}\}$.
 Moreover, for $s_1, s_2 \in \{\pm\}$, $\ev$ satisfies the strong path-lifting property on the set $$S_{s_1,s_2} = \big\{(\pm P_1,\pm P_2) \in \Par^{s_1}\times \Par^{s_2}| \pm P_1P_2 \neq \pm P_2P_1\big\}.$$
 
\end{lemma} 
        
\subsection{Auxiliary lemmas}\label{subsec_others}

Finally, we collect here several technical lemmas that will be used in the sequel.

    \begin{lemma}\label{lem_offdiag}\cite[Lemma 2.13]{ryu}
        For $n\in \mathbb{Z}$ and a parabolic element $\widetilde{A}\in \mathrm{Par}_n$ of $\univcover,$ let  $s(\widetilde{A})\in\{\pm\}$ be its sign, i.e., $s(\widetilde{A})=+$ if $\widetilde{A}\in \mathrm{Par}^+_n,$ and $s(\widetilde{A})=-$ if $\widetilde{A}\in \mathrm{Par}^-_n.$ Let $A$ be the projection of $\widetilde{A}$ to $\SL$, and for $i,j\in\{1,2\}$ let $a_{ij}$ be the $(i,j)$-entry of $A.$ Then either $a_{12}\neq0$ or $a_{21}\neq 0.$  Furthermore,
        
        \begin{enumerate}[(1)]
       \item If $n$ is even, then $s(\widetilde{A})=sgn(a_{12})$ if $a_{12}\neq 0,$ and $s(\widetilde{A})=-sgn(a_{21})$ if $a_{21}\neq 0.$
       
    \item If $n$ is odd, then $s(\widetilde{A})=-sgn(a_{12})$ if $a_{12}\neq 0,$  and $s(\widetilde{A})=sgn(a_{21})$ if $a_{21}\neq 0.$
    \end{enumerate}
    \end{lemma}

\begin{lemma}\label{lem_offdiag_Ell}\cite[Lemma 7.4]{ryu} 
        For $n\in \mathbb{Z}$ and an elliptic element $\widetilde{A}\in \mathrm{Ell}_n$ of $\univcover,$ let $A$ be its projection to $\SL$, and for $i,j\in\{1,2\}$ let $a_{ij}$ be the $(i,j)$-entry of $A.$ 
        Then $a_{12}\neq0$ and $a_{21}\neq 0.$  Moreover:
        \begin{enumerate}[(1)]
            \item 
            If $n$ is odd, then $sgn(n)=sgn(a_{12}) = -sgn(a_{21})$.
            
            \item 
            If $n$ is even, then $sgn(n)=-sgn(a_{12}) = sgn(a_{21})$.
        \end{enumerate}
    \end{lemma}

 \begin{lemma}\label{lem_conjugacypath_const}\cite[Lemma 2.19]{ryu}
        Let $\pm A$ be an element of $\psl\setminus \{\pm \mathrm I\}$,
        and let $\{\pm A_t\}\interval$ 
        be a continuous path in $\psl$ with $\pm A_0 = \pm A$. 
        Suppose that for each $t\in [0,1]$, the elements $\pm A_t$ and $\pm A$ are conjugate. 
        Then there exists a continuous path 
        $\{g_t\}_{t\in [0,1]}$ in $\psl$ with $g_0 = \pm \mathrm{I}$ such that $\pm A_t = \pm g_t Ag_t^{-1}$. 
    \end{lemma}

    \begin{lemma}\label{lem_torus}\cite[Lemma 3.4.5]{goldman_torus} 
        Let $A, B\in \SL$. The following conditions are equivalent:
        \begin{enumerate}[(1)]
       \item $\tr[A,B]<2$;
       \item $A, B$ are hyperbolic elements and their invariant axes cross.
       \end{enumerate}
    \end{lemma}
    \begin{lemma}\cite[Lemma 6.8]{Marche-Wolff}\label{BAn}
    Let $A,B \in P\SL$. Suppose that $A$ is elliptic of infinite order. Then there exists $n\in \mathbb{Z}$ such that $BA^n$ is elliptic, not of order 2.
    \end{lemma}

\section{Outline of the proof of Theorem \ref{mainresult}}\label{sec_outline}
We begin by fixing the notation used throughout the article. Let $\Sigma = \Sigma_{g,p},\  p\geqslant 1$. 
Let $\mathcal{S}$ denote the set of non-peripheral simple closed curves in $\Sigma$ which can be considered as a subset of conjugacy classes in $\pi_1(\Sigma)$. Recall that $\mathcal{M}^s_n$ is the set of the $\psl$-conjugacy classes of non-elementary type-preserving representations $\rho$ with $s(\rho)=s$ and $e(\rho)=n$.  Let \[\mathcal{NH}^s_n=\big\{[\rho]\in\mathcal{M}^s_n\big|\text{ } \exists [\gamma]\in \mathcal{S} \text{ such that } |Tr(\rho(\gamma))|\leqslant 2\big\}\]
be the set of the classes of non-elementary representations that map at least one non-peripheral simple closed curve to a non-hyperbolic element.
Let us define $\mathcal{S}^{ns}$ as the subset of $\mathcal{S}$ consisting of non-separating simple closed curves. We define 
\[\mathcal{E}^s_n= \big\{[\rho]\in \mathcal{M}^s_n\big| \text{ }\exists [\gamma]\in \mathcal{S}^{ns} \text{ such that } \rho(\gamma) \text{ is elliptic}\big\},\] which is a $\operatorname{MCG}(\Sigma)$-invariant open subset of $\mathcal{M}^s_n$. We also consider the set $\mathcal{EL}^s_n$ defined as follows
\[\mathcal{EL}^s_n = \big\{ [\rho] \in \mathcal{M}^s_n\big| \ \exists[\gamma]\in \mathcal{S}^{ns} \text{ such that } \rho(\gamma) \text{ is elliptic of infinite order}\big\}.\]
\smallskip

Theorem \ref{mainresult} follows from Theorem \ref{ergE} and Proposition \ref{EfullNH} below.
    
    \begin{theorem}\label{ergE}
     Let $\Sigma=\Sigma_{g,p}$ with $g \geqslant 2$ and $p\geqslant1$. For $s \in \{\pm 1\}^p$ and $n\in \big\{\chi(\Sigma) + p_+(s) + 1,\dots, -\chi(\Sigma) - p_-(s)- 1\big\}$, the space $\mathcal{EL}^s_n$ is non-empty and connected, and the action of $\operatorname{MCG}(\Sigma)$ on $\mathcal{E}^s_n$ is ergodic.  
    \end{theorem}
    \begin{proposition}\label{EfullNH}
    Let $\Sigma = \Sigma_{g,p}$ with $g\geqslant 2$ and $p\geqslant 1$. For $s\in \{\pm 1\}^p$ and $n\in \big\{\chi(\Sigma) + p_+(s) + 1,\dots, -\chi(\Sigma) - p_-(s)- 1\big\}$, $\mathcal{E}^s_n$ has full measure in $\NHns$. 
     \end{proposition}
    
    The proof of Proposition \ref{EfullNH} is given in Section \ref{sec_EfullNH}.

    In order to prove Theorem \ref{ergE}, we define the set $\mathcal{U}^s_n$ as follows. 
    Recall from Section \ref{sec_preliminary} that $c_1,\dots,c_p$ denote the preferred peripheral elements of
$\pi_1(\Sigma)$.
Let $\overline{\mathcal{M}}_{s,n}$ consist of conjugacy classes of non-elementary representations ${\phi}:\pi_1(\Sigma)\rightarrow \SL$ such that ${\phi}(c_i)$ is in the conjugacy class of $\begin{bmatrix}
        1 & s_i \\
        0 & 1
\end{bmatrix}$ for each $i\in \{1,\dots,p-1\}$, and ${\phi}(c_p)$ is in the conjugacy class of $(-1)^n\begin{bmatrix}
        1 & s_p\\
        0 & 1
\end{bmatrix}$. For every representation ${\phi}: \pi_1(\Sigma)\rightarrow \SL$, we denote by $\pi({\phi})$ the induced representation into $\psl$. For $[\rho]\in \mathcal{M}^s_n$, choose $[\overline{\rho}]\in \overline{\mathcal{M}}_{s,n}$ such that $[\pi(\overline{\rho})]=[\rho]$.
Then there exist neighbourhoods $W^{[\rho]}_{s,n} \subset \mathcal{M}^s_n$ of $[\rho]$ and $\overline{W}^{[\overline{\rho}]}_{s,n}\subset \overline{\mathcal{M}}_{s,n}$ of $[\overline{\rho}]$ with a diffeomorphism $\Pi:\overline{W}^{[\overline{\rho}]}_{s,n}\rightarrow  W^{[\rho]}_{s,n}$ defined by $\Pi([{\phi}])=[\pi({\phi})]$. 
     For every $\gamma \in \pi_1(\Sigma)$, define $F_{\gamma}:W^{[\rho]}_{s,n}\rightarrow \mathbb{R}$ by $F_{\gamma}([\phi]) = Tr\big(\Pi^{-1}([\phi])(\gamma)\big)$
     . Finally, recall that for $[\rho]\in \mathcal{M}^s_n$, the cotangent space $T^*_{[\rho]}\mathcal{M}^s_n$ has dimension $6g-6 +2p$. Let $\mathcal{U}^s_n$ be
 

    \[
      \mathcal{U}^s_n
    =
    \left\{
    [\rho]\in \mathcal{E}^s_n \;\middle|\;
    \exists\big ([\gamma_i]\big)_{i=1}^{6g-6+2p}\ \text{simple}:
    \ \langle dF_{\gamma_i}\rangle = T^*_{[\rho]}\mathcal{M}^s_n,\ 
    |\operatorname{Tr}(\rho(\gamma_i))|<2
    \right\}.
    \]

    To prove Theorem \ref{ergE}, we need the following Proposition \ref{connect_el}, Proposition \ref{U} and Lemma \ref{Nfull}.

\begin{proposition}\label{connect_el}
   Let $\Sigma=\Sigma_{g,p}$ with $g \geqslant 2$ and $p\geqslant 1$.  Let $n\in \mathbb{Z}$ and $s\in \{\pm 1\}^p.$ Then the space $\mathcal{EL}^s_n$ is non-empty if and only if the pair $(n,s)$ satisfies the following inequality: 
    $$\chi(\Sigma) + p_+(s)+1\leqslant n\leqslant -\chi(\Sigma) - p_-(s)-1.$$
    Moreover, each non-empty $\mathcal{EL}^s_n$ above is connected. 
\end{proposition}

    \begin{proposition}\label{U}
        The space $\mathcal{U}^s_n$ is an open subspace of $\mathcal{M}^s_n$ containing $\mathcal{EL}^s_n$.
    \end{proposition}

   Let $\N:=\big\{[\rho]\in \mathcal{M}^s_n\big| \text{ } \forall\gamma\in \mathcal{S}, \rho(\gamma) \text{ is not } \pm \mathrm{I}, \text{parabolic, nor elliptic of finite order}\big\}.$
    \begin{lemma}\label{Nfull}
     $\N$ has full measure in $\mathcal{M}^s_n$. In particular $\mathcal{EL}^s_n$ has full measure in $\mathcal{E}^s_n$.
    \end{lemma}
    \begin{proof}
    Let $[\rho] \in \mathcal{M}^s_n \setminus \N$. Then $\rho$ sends a simple closed curve $\gamma$ to either elliptic of finite order, parabolic, or the identity. Consider the map $f_\gamma: \mathcal{M}^s_n\rightarrow \mathbb{R}$ defined by $f_\gamma([\phi])= Tr(\phi(\gamma))^2$. Then $f_\gamma\big([\rho]\big)$ is contained in a countable set $4\cos^2(\mathbb{Q}\pi)$. This shows $\mathcal{M}^s_n \setminus  \N$ is a countable union of real-analytic subset of $\Mns$ which implies it has measure zero \cite{Mityagin}. Further, the complement of $\mathcal{EL}^s_n$ in $\mathcal{E}^s_n$ is a subset of $\mathcal{M}^s_n \setminus \N$ and is hence of measure zero which implies that $\mathcal{EL}^s_n$ has full measure in $\mathcal{E}^s_n$.
    \end{proof}
    
    The proof of Theorem \ref{ergE} below is adapted from that of \cite[Theorem 6.1]{Marche-Wolff}.
    {\begin{proof}[\textbf{Proof of Theorem \ref{ergE}}]Let $f:\mathcal{E}^s_n\to \mathbb{R}$ be a measurable function invariant under $\operatorname{MCG}(\Sigma)$ action. We show that $f$ is almost everywhere\ constant. By Proposition \ref{U},
     $\mathcal{EL}^s_n\subset \mathcal{U}^s_n\subset \mathcal{E}^s_n$, and by Lemma \ref{Nfull}, $\mathcal{EL}^s_n$ has full measure in $\mathcal{E}^s_n$. Therefore, the set $\mathcal{U}^s_n$ also has full measure in $\mathcal{E}^s_n$. Hence it suffices to prove that $f$ is almost everywhere\ constant on $\mathcal{U}^s_n$.
     
    We see $f$ is locally almost everywhere\ constant on $\mathcal{U}^s_n$, that is, for every $[\rho]\in \mathcal{U}^s_n$, there exists a neighbourhood $V_{[\rho]}$ such that $f$ is almost everywhere\ constant on $V_{[\rho]}$. It suffices to show $f\circ \Pi:\overline{W}^{[\overline{\rho}]}_{s,n}\rightarrow \mathbb{R}$ is locally almost everywhere constant. See \cite[Proposition~6.5] {Marche-Wolff} for the proof.  
    By Proposition \ref{connect_el}, $\mathcal{EL}^s_n$ is connected;  and since $\mathcal{EL}^s_n\subset \mathcal{U}^s_n$ has full measure in $\mathcal{U}^s_n$, $\mathcal{U}^s_n\subset \overline{\mathcal{EL}}^s_n$, and hence $\mathcal{U}^s_n$ is connected. 
    Define $\bar{f}:\mathcal{U}^s_n\rightarrow \mathbb{R}$ as follows: for every $[\rho] \in \mathcal{U}^s_n$, since $f$ is almost everywhere constant on $V_{[\rho]}$, we may define $\bar{f}\big([\rho]\big)$ as this constant. 
    Since $\bar{f}$ is locally constant and $\mathcal{U}^s_n$ is connected, $\bar{f}$ is constant; and since $f$ is equal to $\bar{f}$ almost everywhere, $f$ is almost everywhere constant on $\mathcal{U}^s_n$, as desired. 
    \end{proof}

\section{Non-emptiness criterion and connectedness of $\mathcal{EL}^s_n$}\label{sec_conn}

The main purpose of this section is to prove 
Proposition \ref{connect_el}. Let $\mathcal{S}$ be a set of non-peripheral simple closed curves in $\Sigma$, considered as a subset of the set of conjugacy classes of $\pi_1(\Sigma)$. Let $i:\mathcal{S}\times\mathcal{S} \rightarrow \mathbb{N}$ denote the minimal geometric intersection number, and let $\mathcal{C}$ denote the set of pairs $(a,b) \in \mathcal{S}\times\mathcal{S}$ with $i(a,b)=1$. Note that if $i(a,b)=1$, then $a$ and $b$ are non-separating, and a regular neighbourhood of $a\cup b$ is an embedded subsurface of $\Sigma$ homeomorphic to a one-holed torus.
\smallskip

For each $n\in \mathbb{Z}, s\in \{\pm 1\}^p$ and $(a,b)\in \mathcal{C}$, denote by $\mathcal{EL}^s_n(a,b)$ the set of $\psl$-conjugacy classes of representations $\rho$ of relative Euler class $n$ and sign $s$, such that $\rho(a)$ and $\rho(b)$ are elliptic, do not commute with each other, and at least one of them has infinite order. Notice that for each $[\rho]\in \ELns(a,b)$, $\rho\big([a,b]\big)$ is hyperbolic (See Lemma \ref{lem_torus}). 
\medskip


We first observe $\mathcal{EL}^s_n=\displaystyle\cup_{(a,b)\in \mathcal{C}}\mathcal{EL}^s_n(a,b)$. The reverse inclusion is by definition. Conversely, if $[\rho] \in \mathcal{EL}^s_n$, then $\rho$ sends some non-separating curve $a$ to an elliptic element of infinite order. Assume every curve $b$ with $(a,b)\in \mathcal{C}$ maps to an element that commutes with $\rho(a)$. Since the set of such simple closed curves $b$ along with $a$ generate $\pi_1(\Sigma)$, the peripheral elements map to elliptic elements, which contradicts that $\rho$ is type-preserving. Hence, there exists a curve $b$ satisfying $(a,b)\in \mathcal{C}$ such that $\rho(a)$ and $\rho(b)$ do not commute. Then by Lemma \ref{BAn}, there exists $k\in \mathbb{Z}$ such that $\rho \in \mathcal{EL}^s_n(a,ba^k)$. \\

 Define an equivalence relation $\sim$ on $\mathcal{C}$ generated by $(a,b)\sim(a',b')$ if $\mathcal{EL}^s_n(a,b)\cap \mathcal{EL}^s_n(a',b')\neq \emptyset$. \\

Since $\mathcal{EL}^s_n=\displaystyle\cup_{(a,b)\in \mathcal{C}}\mathcal{EL}^s_n(a,b)$, 
 Proposition \ref{connect_el} follows from Propositions \ref{unique} and \ref{connect_elab}.

 \begin{proposition}\cite[Proposition 6.7]{Marche-Wolff}\label{unique}
  The equivalence relation $\sim$ has a unique equivalence class in $\mathcal{C}$.   
 \end{proposition}

\begin{proposition}\label{connect_elab}
  Let $\Sigma=\Sigma_{g,p}$ 
  with genus 
  $g \geqslant 2$ and $p\geqslant 1$ punctures. 
  Let $n\in \mathbb{Z},\ s\in \{\pm 1\}^p,$ and $(a,b)\in \mathcal{C}$. The space $\mathcal{EL}^s_n(a,b)$ is non-empty if and only if the pair $(n,s)$ satisfies the following inequality: 
    $$\chi(\Sigma) + p_+(s)+1\leqslant n\leqslant -\chi(\Sigma) - p_-(s)-1.$$
    Moreover, each non-empty $\mathcal{EL}^s_n(a,b)$ above is connected.   
\end{proposition}

The proof of Proposition \ref{unique} follows verbatim that of \cite[Proposition 6.7]{Marche-Wolff}. We prove Proposition \ref{connect_elab} in the following subsection \ref{sec_pfof4.3}.

\subsection{Proof of Proposition \ref{connect_elab}}\label{sec_pfof4.3}

Throughout this section, we fix $(a,b)\in \mathcal{C}$, and denote by $T\subset \Sigma$ a regular neighbourhood of $a\cup b$ homeomorphic to a one-holed torus. 
By abuse of notation, we consider $a$ and $b$ as the generators of $\pi_1(T)\leqslant \pi_1(\Sigma)$.
Let $\gamma = [a,b]^{-1}$ represented by the boundary component of $T$.
Note that the complement $\Sigma\setminus T$ is a subsurface of $\Sigma$ of genus $g-1$, with $p$ punctures corresponding to the preferred peripheral elements $c_1,\dots, c_p\in \pi_1(\Sigma)$, and one boundary component representing $\gamma^{-1}$. 
\smallskip

 We fix a pants decomposition of $\Sigma = \Sigma_{g,p}$ as follows. Let $m = \frac{p}{2}$ if $p$ is even, and $m = \frac{p-1}{2}$ if $p$ is odd. For $i\in \{1,\cdots, m\}$, we denote by $P_i$ the pair of pants whose fundamental group $\pi_1(P_i)$ contains $c_{2i-1}$ and $c_{2i}$. 
Moreover, if $p$ is odd, then we denote by $P_\gamma$ the pair of pants such that $\pi_1(P_\gamma)$ contains $c_p = c_{2m+1}$ and $\gamma$. Let $T_j$ for $j\in \{1,\cdots, g-1\}$ denote the $g-1$ tori in $\Sigma\setminus T$. Next, decompose the complement of $\big(\bigcup_{j = 1}^{g-1} T_j\big)\cup \big(\bigcup_{i = 1}^m P_i\big)\cup P_\gamma$ (when $p$ is odd) or $\big(\bigcup_{j = 1}^{g-1} T_j\big)\cup \big(\bigcup_{i = 1}^m P_i\big)$ (when $p$ is even) by $g+m-2$ pairs of pants $\{P'_k\}_{k\in \{1,\dots, g+m-2\}}$, so that each of $T_j$ and $P_i$ is adjacent to some $P'_k$.
Together with $T$, we obtain a decomposition of $\Sigma$. See Figure \ref{even} and Figure \ref{odd} for an illustration of this pants decomposition.
\begin{figure}[h!]
    \centering

        \vspace{-1.7cm}
        \begin{overpic}[width=0.9\textwidth,height=13.5cm]{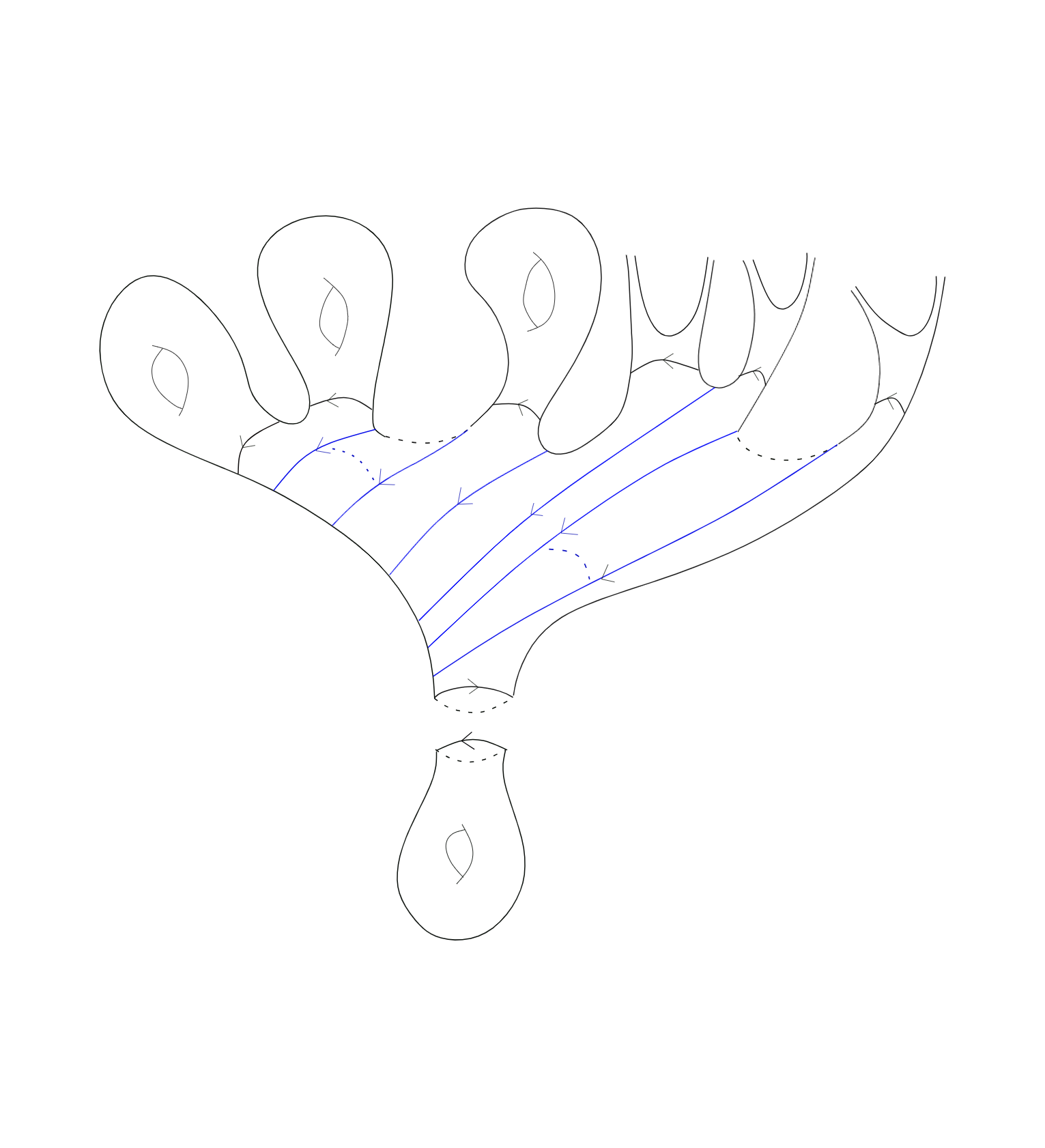}
            \put(35,20){\small $T$}
            \put(12,70){\small $T_1$}
            \put(25,74){\small $T_2$}
            \put(45,75){\small $T_{g-1}$}

            \put(26,55){\small $P'_1$}

            \put(36,50){\small $P'_{g-2}$}

            \put(59,72){\scriptsize $c_1$}
            \put(67,72){\scriptsize $c_2$}
            \put(70,71.5){\scriptsize $c_3$}
            \put(77,71.5){\scriptsize $c_4$}
            \put(80,69){\scriptsize $c_{p-1}$}
            \put(90,70){\scriptsize $c_p$}
            \put(22,57){\scriptsize $\alpha_1$}
            \put(31,60){\scriptsize $\alpha_2$}
            \put(42.5,59){\scriptsize $\alpha_{g-1}$}
            \put(57.5,61){\scriptsize $\alpha_{g}$}

            \put(63,67){\small{$P_1$}}
            \put(74,60){\scriptsize $\alpha_{g+1}$}
            
            \put(76.5,63){\small{$P_{2}$}}
            \put(87,57){\scriptsize $\alpha_{g+m-1}$}

            \put(90,62){
            \small{$P_{m}$}}
            \put(24,50){\scriptsize \textcolor{blue}{$\beta_1$}}
            \put(28,47){\scriptsize \textcolor{blue}{$\beta_{g-3}$}}
            \put(32.5,45){\scriptsize \textcolor{blue}{$\beta_{g-2}$}}
            \put(36,41){\scriptsize \textcolor{blue}{$\beta_{g-1}$}}
            \put(38.5,38.5){\scriptsize \textcolor{blue}{$\beta_{g}$}}
            \put(34.5,36){\scriptsize \textcolor{blue}{$\beta_{g+m-3}$}}
            \put(50,35){\scriptsize $\gamma^{-1}$}
            \put(49.5,31){\scriptsize $\gamma$}
        \end{overpic}
        \vspace{-1.2cm}
        \caption{The pants decomposition of $\Sigma$ when $p$ is even}
        \label{even}
\end{figure}
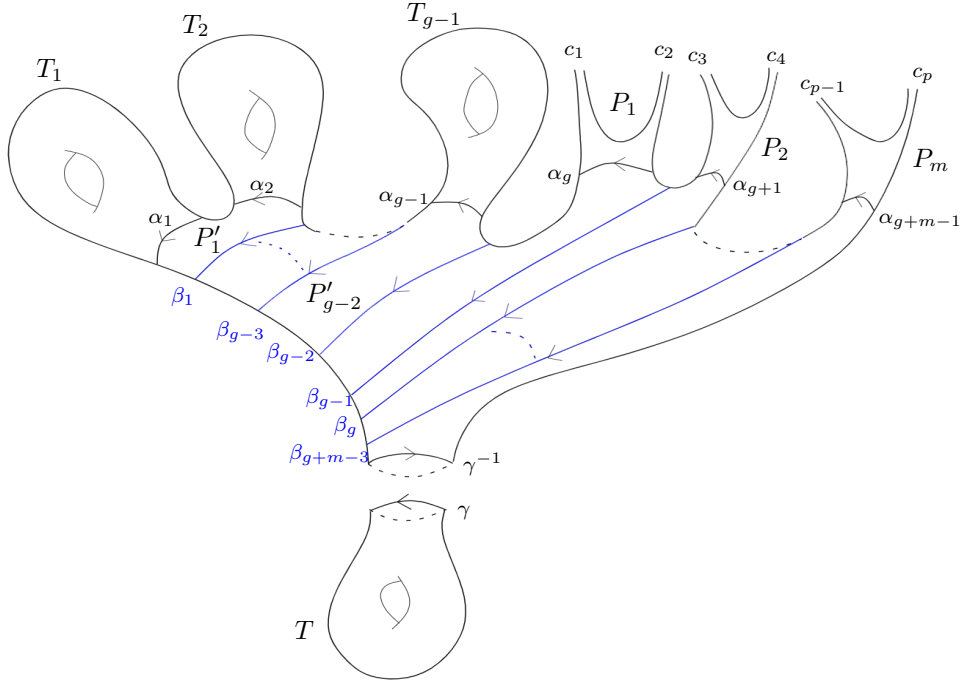

    For each $j\in\{1,\dots,g-1\}$, let $a_j, b_j$ be the generators of $\pi_1(T_j)$. The fundamental group of $\Sigma$ has the presentation
    $$\pi_1(\Sigma) = \big\langle 
     a,b, a_1, b_1,a_2, b_2,\cdots, a_{g-1}, b_{g-1},
     c_1, \cdots, c_p\ \big|\ [a,b][a_1, b_1][a_2, b_2]\cdots [a_{g-1}, b_{g-1}]\cdot c_1
     \cdots c_p  \big\rangle.$$ 
    For each $j\in\{1,\dots,g-1\}$,  let 
    $\alpha_j =  [a_j, b_j]$
    be the element of $\pi_1(\Sigma)$ represented by a boundary component $T_j$; and for $i\in\{1,\dots,m\}$, let $\alpha_{g+i-1} = c_{2i-1}c_{2i}$ represented by a boundary component of $P_i$. Let $\beta_1 = \alpha_1\alpha_2$, and
    for each $k\in \{2,\dots, g+m-3\}$, let $\beta_k = \beta_{k-1}\alpha_{k+1}$ be the element of $\pi_1(\Sigma)$ represented by the common boundary component of $P'_k$ and $P'_{k+1}$. Finally, if $p$ is odd, we let $\beta_{g+m-2} = \beta_{g+m-3}\alpha_{g+m-1}$ represented by the common boundary component of $P'_{g+m-2}$ and $P_\gamma$; and if $p$ is even, we let $\beta_{g+m-2} = \gamma = \beta_{g+m-3}\alpha_{g+m-1}$. 
    By abuse of notation, we let these $\alpha_l$ and $\beta_k$ be both the elements of $\pi_1(\Sigma)$ and their representatives, and call them the \emph{decomposition curves}.
\medskip

To prove Proposition \ref{connect_elab}, we introduce the subspace $\ELns(a,b)'\subset \ELns(a,b)$ consisting of classes of representations $[\rho]$ such that the restrictions $\rho|_{\pi_1(P_i)}$ for $i\in \{1,\dots, m\}$,  $\rho|_{\pi_1(T_j)}$ for $j\in \{1,\dots, g-1\}$ and $\rho|_{\pi_1(T)}$ are all non-abelian, and for each $k\in \{1,\dots, g+m-2\}$, none of the boundary components of $P'_k$ maps to the identity.

\begin{lemma}\label{density}
For $n\in \mathbb{Z}$ and $s\in \{\pm 1\}^p$, $\mathcal{EL}^s_n(a,b)'$ is dense in $\mathcal{EL}^s_n(a,b)$.
\end{lemma}
\begin{proof}
    For the sake of contradiction, suppose that there is a non-empty subset $V\subset \ELns(a,b)\setminus \ELns(a,b)'$ that is open in $\ELns(a,b)$. 
    Notice that each $[\rho]\in \ELns(a,b)\setminus \ELns(a,b)'$ sends some simple closed curve to either the identity or a parabolic element, hence is contained in $\Mns\setminus \N$. By Lemma \ref{Nfull}, 
    $\Mns\setminus \N$ has measure zero, hence its subset
    $\ELns(a,b)\setminus \ELns(a,b)'$ also has measure zero. We show $V$ has positive measure, which leads to a contradiction. Let $\mathcal{E}^s_n(a,b)$ be the subspace of $\Mns$ consisting of the conjugacy classes $[\rho]$ such that $\rho(a)$ and $\rho(b)$ are elliptic. Since $\ELns(a,b)\subset \mathcal{E}^s_n(a,b)$, $\mathcal{E}^s_n(a,b)$ is non-empty; and since $\mathcal{E}^s_n(a,b)$ is open in $\Mns$, it has positive measure. 
    Moreover, since the complement
    $\mathcal{E}^s_n(a,b)\setminus \ELns(a,b)$ is contained in $\Mns\setminus \N$, the set $\ELns(a,b)$ has full measure in $\mathcal{E}^s_n(a,b)$. 
    Let $U$ be an open subset of $\mathcal{E}^s_n(a,b)$, which is also open in $\Mns$, such that $V = U\cap \mathcal{EL}^s_n(a,b)$. As an open subset of $\Mns$, $U$ has positive measure. Since $\ELns(a,b)$ has full measure in $\mathcal{E}^s_n(a,b)$, 
    $V = U\cap \ELns(a,b)$ also has positive measure, contradicting that it is contained in a null set.
    \end{proof}

\begin{figure}[H]
    \centering

        \vspace{-1.7cm}
        \begin{overpic}[width=0.9\textwidth,height=13.5cm]{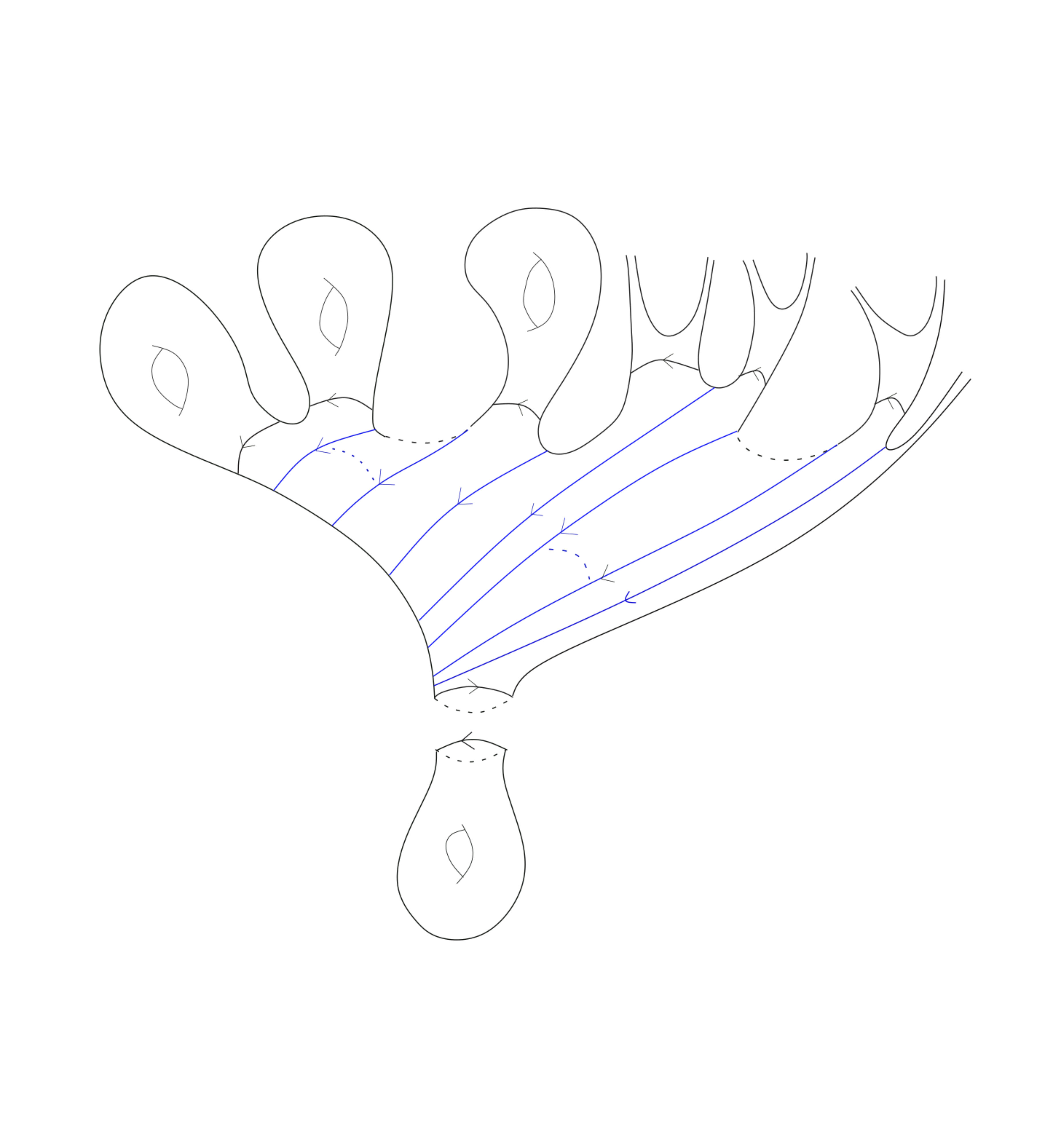}
            \put(35,20){\small $T$}
            \put(12,70){\small $T_1$}
            \put(25,74){\small $T_2$}
            \put(45,75){\small $T_{g-1}$}

            \put(26,55){\small $P'_1$}

            \put(36,50){\small $P'_{g-2}$}

            \put(55.5,38){\small $P_{\gamma}$}

            \put(59,72){\scriptsize $c_1$}
            \put(67,72){\scriptsize $c_2$}
            \put(70,71.5){\scriptsize $c_3$}
            \put(77,71.5){\scriptsize $c_4$}
            \put(80,69){\scriptsize $c_{p-2}$}
            \put(90,70){\scriptsize $c_{p-1}$}
            \put(93,62){\scriptsize $c_p$}
            \put(22,57){\tiny $\alpha_1$}
            \put(31,60){\tiny $\alpha_2$}
            \put(43.5,60){\tiny $\alpha_{g-1}$}
            \put(58,61){\tiny $\alpha_{g}$}

            \put(63,67){\small{$P_1$}}
            \put(75,61){\tiny $\alpha_{g+1}$}
            
            \put(76.5,63){\small{$P_{2}$}}
            \put(76.5,58){\tiny $\alpha_{g+m-1}$}
            \put(91,65){
            \small{$P_{m}$}}
            
            \put(24.5,51){\tiny \textcolor{blue}{$\beta_1$}}
            \put(28,48){\tiny \textcolor{blue}{$\beta_{g-3}$}}
            \put(33,44.5){\tiny \textcolor{blue}{$\beta_{g-2}$}}
            \put(36,41){\tiny \textcolor{blue}{$\beta_{g-1}$}}
            \put(38.8,38.8){\tiny \textcolor{blue}{$\beta_{g}$}}
            \put(34.8,37.5){\tiny \textcolor{blue}{$\beta_{g+m-3}$}}
            \put(34.8,36){\tiny \textcolor{blue}{$\beta_{g+m-2}$}}
            \put(49.5,35){\tiny $\gamma^{-1}$}
            \put(49,31){\tiny $\gamma$}
        \end{overpic}
        \vspace{-1.2cm}
        \caption{The pants decomposition of $\Sigma$ when $p$ is odd}
        \label{odd}

\end{figure}
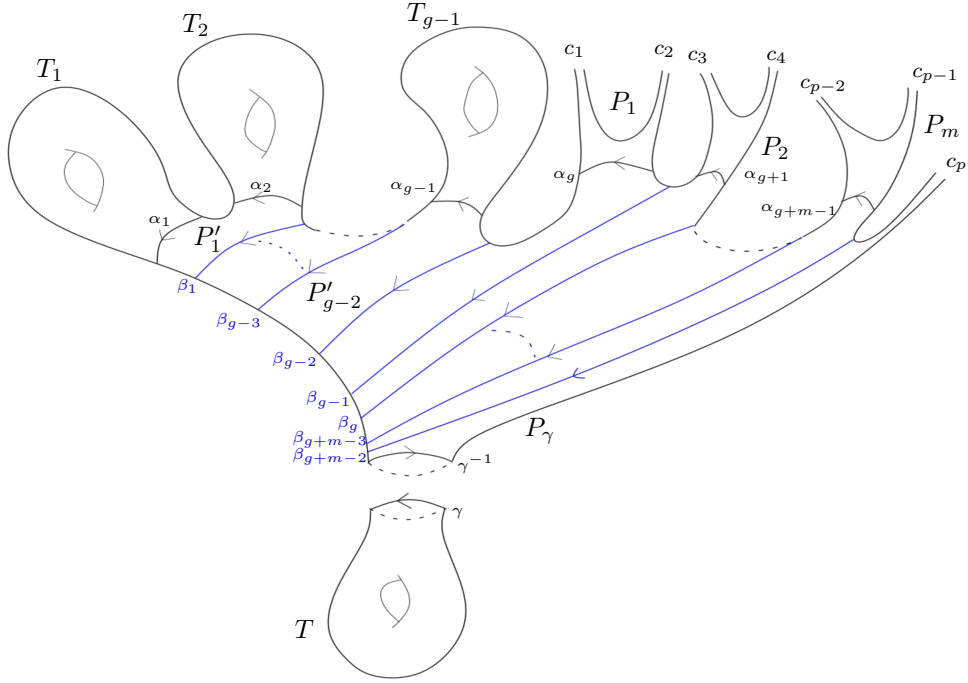
The following  Lemma \ref{ELnonempty} provides a necessary and sufficient condition for $\mathcal{EL}^s_n(a,b)'$ to be non-empty. 
\begin{lemma}\label{ELnonempty}
    Let $\Sigma=\Sigma_{g,p}$ 
      with genus 
      $g \geqslant 2$ and $p\geqslant 1$ punctures. 
      Let $n\in \mathbb{Z},\ s\in \{\pm 1\}^p,$ and $(a,b)\in \mathcal{C}$. The set $\mathcal{EL}^s_n(a,b)'$ is nonempty if and only if the pair $(n,s)$ satisfies the following inequality: 
    $$\chi(\Sigma) + p_+(s)+1\leqslant n\leqslant -\chi(\Sigma) - p_-(s)-1.$$
\end{lemma}

\begin{proof}
First, we assume that $\mathcal{EL}^s_n(a,b)'$ is non-empty. Let $\rho \in \mathcal{EL}^s_n(a,b)'$. Then $\rho$ maps the non-separating curves $a,b$ to elliptic elements that do not commute. 
Let $T\subset \Sigma$ be a regular neighbourhood of $a\cup b$ homeomorphic to a one-holed torus, whose boundary component represents the commutator 
$[a,b]$. By Lemma \ref{lem_torus}, its boundary image $\rho\big([a,b]\big)$ is hyperbolic, hence the relative Euler classes of the restrictions $\rho|_{\pi_1(T)}$ and $\rho|_{\pi_1(\Sigma\setminus T)}$ are well-defined.
By Theorem \ref{ryuhp}, we have 
$$-1\leqslant   e\big(\rho|_{\pi_1(T)}\big) \leqslant 1$$
and 
$$\chi(\Sigma) + p_+(s) + 1\leqslant e\big(\rho|_{\pi_1(\Sigma \backslash T)}\big) \leqslant -\chi(\Sigma) - p_-(s) - 1.$$
As $\rho(a)$ is elliptic, $\rho|_{\pi_1(T)}$ is not a holonomy representation; and by Proposition \ref{prop_holonomy}, we have $e\big(\rho|_{\pi_1(T)}\big) = 0$.
Finally, Proposition \ref{prop_additivity} gives the desired inequality.
\\
 Conversely, for $(n,s)$ satisfying $\chi(\Sigma) + p_+(s) + 1\leqslant n \leqslant -\chi(\Sigma) - p_-(s) - 1,$
we will construct a representation in $\mathcal{EL}^s_n(a,b)'$. 
Let $\rho$ be a representation on $\pi_1(T)$ such that $\rho(a)$ is elliptic of infinite order, and $\rho(b)$ is elliptic and does not commute with $\rho(a)$. 
Let $\gamma = [a,b]^{-1}$ be represented by the boundary of $T$. As in the previous paragraph, since $\rho(a)$ and $\rho(b)$ are non-commuting elliptic elements, $\rho(\gamma)$ is hyperbolic and $e(\rho)=0$. 
We will extend $\rho$ to $\pi_1(\Sigma)$ inductively along the pants decomposition, considering separately the cases when $p$ is even and when $p$ is odd.
\medskip

 We first assume that $p$ is even, in which case the pants decomposition is as in Figure \ref{even}. In this case, we
 let $m = \frac{p}{2}$.
 Let $n_{P_i}: = \frac{1}{2}(s_{2i-1}+s_{2i})$ for $i\in \{1,\cdots, m\}$. 
  Since we have
  $$3-2g - m + \sum_{i = 1}^m n_{P_i}\leqslant n   \leqslant 2g + m - 3+\sum_{i = 1}^m n_{P_i},$$
 for each $j\in \{1,\dots, g-1\}$ and $k\in \{1,\dots, g+m-2\}$, we may choose any numbers $n_{T_j}$ and $n_{P'_k}$ in $\{-1,0,1\}$ that together satisfy 
 $ \sum_{j = 1}^{g-1} n_{T_j} + \sum_{k = 1}^{g+m-2} n_{P'_k} + \sum_{i = 1}^m n_{P_i}
 = n $.
\\

First, we extend $\rho$ to $\pi_1(P'_{k})$ for all $k\in \{1,\dots, g+m-2\}$ by induction. We start with the case where $k = g+m-2$. In this case, $P'_k$ shares a boundary component $\gamma = \beta_{g+m-3}\alpha_{g+m-1}$ with $T$. Let $\widetilde{\rho(\gamma)}$ be the lift of $\rho(\gamma)$ in $\Hyp_{n_{P_k'}}$. Since $n_{P_k'}\in \{-1,0,1\}$, by Proposition \ref{prop_evimage}, there is a pair $(\pm A, \pm B)$ of hyperbolic elements such that the lifted product $\ev(\pm A, \pm B) = \widetilde{\rho(\gamma)}$. Then letting $\rho(\beta_{g+m-3}) = \pm A$ and $\rho(\alpha_{g+m-1}) = \pm B$ defines $\rho$ on $\pi_1(P'_k)$. 
Now let $k\in \{2,\dots, g+m-3\}$, and assume that $\rho$ is defined on $\pi_1(P'_{k+1})$. Then $P'_k$ shares a boundary component $\beta_k = \beta_{k-1}\alpha_{k+1}$ with $P'_{k+1}$. Let $\widetilde{\rho(\beta_k)}$ be the lift of $\rho(\beta_k)$ in $\Hyp_{n_{P'_k}}$. As in the case $k = g+m-2$, by Proposition \ref{prop_evimage}, we can find a pair of hyperbolic elements whose lifted product equals $\widetilde{\rho(\beta_k)}$, which defines $\rho(\beta_{k-1})$ and $\rho(\alpha_{k+1})$. This defines $\rho$ on $\pi_1(P'_k)$ for all $k\in \{2,\dots, g+m-2\}$. Finally, for $k = 1$, we have $\beta_1 = \alpha_1\alpha_2$. Letting $\widetilde{\rho(\beta_1)}$ be the lift of $\rho(\beta_1)$ in $\Hyp_{n_{P'_1}}$, we can define $\rho$ similarly using Proposition \ref{prop_evimage}.
This defines $\rho$ on $\pi_1(P'_k)$ for all $k\in \{1,\dots, g+m-2\}$, with the
relative Euler class $e\big(\rho|_{\pi_1(P'_k)}\big) = n_{P'_k}$.
\smallskip

Next, we extend $\rho$ to $\pi_1(T_j)$ and $\pi_1(P_i)$ for all $j\in \{1,\cdots, g-1\}$ and $i\in \{1,\dots, m\}$. For $l\in \{1,\dots, g+m-1\}$, the image $\rho(\alpha_l)$ is defined in the previous paragraph.
For $j\in \{1,\dots, g-1\}$, we have $\alpha_j = [a_j, b_j]$.
Let
$\widetilde{\rho(\alpha_j)}$ be the lift of $\rho(\alpha_j)$ in $\Hyp_{n_{T_{j}}}$. Since $n_{T_j}\in \{-1,0,1\}$, by Theorem \ref{thm_liftcommu}, there is a $\psl$-pair $(\pm A, \pm B)$ such that the lifted commutator $\widetilde{R}(\pm A,\pm B)=\widetilde{\rho(\alpha_j)}.$ Then letting $\rho(a_j) = \pm A$ and $\rho(b_j) = \pm B$ extends $\rho$ to $\pi_1(T_j)$, with the relative Euler class $e\big(\rho|_{\pi_1(T_j)}\big) = n_{T_j}$.
For $i\in \{1,\dots, m\}$, we have $\alpha_{g+i-1} = c_{2i-1}c_{2i}$.
let
$\widetilde{\rho(\alpha_{g+i-1})}$ be the lift of $\rho(\alpha_{g+i-1})$ in $\Hyp_{n_{P_i}}$. Since $n_{P_i} = \frac{1}{2}(s_{2i-1}+s_{2i})$, by Proposition \ref{parplp}, there is a pair $(\pm C_{2i-1}, \pm C_{2i})\in \Par^{sgn(s_{2i-1})}\times \Par^{sgn(s_{2i})}$ such that the lifted product $\ev(\pm A, \pm B) = \widetilde{\rho(\alpha_{g+i-1})}$. Then letting $\rho(c_{2i-1}) = \pm C_{2i-1}$ and $\rho(c_{2i}) = \pm C_{2i}$ defines $\rho$ on $\pi_1(P_i)$, 
with the relative Euler class  $e\big(\rho|_{\pi_1(P_i)}\big) = n_{P_i}$. This completes the construction of $\rho$ on $\pi_1(\Sigma)$ when $p$ is even.
\medskip

 We now assume that $p$ is odd, in which case the pants decomposition is as in Figure \ref{odd}. In this case, we
 let $m = \frac{p-1}{2}$.
 Let $n_{P_i}: = \frac{1}{2}(s_{2i-1}+s_{2i})$ for $i\in \{1,\cdots, m\}$.
In this case, we can choose $n_{P_\gamma}\in \{0, s_p\}$, $n_{T_j}\in \{-1,0,1\}$ for $j\in \{1,\dots, g-1\}$, and $n_{P'_k}\in \{-1,0,1\}$ for $k\in \{1,\dots, g+m-2\}$, which together satisfy $
 \sum_{j = 1}^{g-1} n_{T_j} + \sum_{k = 1}^{g+m-2} n_{P'_k} + \sum_{i = 1}^m n_{P_i} + n_{P_\gamma}
 = n  
 $. Note that $P_\gamma$ shares a boundary component $\gamma = \beta_{g+m-2}c_p$ with $T$. Let $\widetilde{\rho(\gamma)}$ be the lift of $\rho(\gamma)$ in $\Hyp_{n_{P_\gamma}}$. 
  Since $n_{P_{\gamma}}\in \{0,s_p\}$, by Proposition \ref{prop_evimage}, there is a pair $(\pm H,\pm C)\in \Hyp\times \Par^{sgn(s_p)}$ such that the lifted product $\ev(\pm H, \pm C)=\widetilde{\rho(\gamma)}$. Letting $\rho(\beta_{g+m-2}) = \pm H$ and $\rho(c_p) = \pm C$ defines $\rho$ on $\pi_1(P_\gamma)$ with the relative Euler class  $e\big(\rho|_{\pi_1(P_\gamma)}\big) = n_{P_\gamma}$. The extension of $\rho$ to $\pi_1(P'_k)$, $k\in \{1,\dots, g+m-2\}$, $\pi_1(T_j), j\in \{1,\cdots, g-1\}$, and $\pi_1(P_i), i\in \{1,\dots, m\}$, then follows verbatim from the case where $p$ is even.
\smallskip

Finally, we show that the constructed $\rho$ is in $\mathcal{EL}^s_n(a,b)'$, concluding that $\mathcal{EL}^s_n(a,b)'$ is non-empty. 
Since $\rho(c_i)\in \Par^{sgn(s_i)}$ for all $i\in \{1,\dots, p\}$, we have $s(\rho) = s$; and by Proposition \ref{prop_additivity}, we have 
$e(\rho) = \sum_{j = 1}^{g-1} n_{T_j} + \sum_{k = 1}^{g+m-2} n_{P'_k}
 + \sum_{i = 1}^m n_{P_i} + 0 = n$ when $p$ is even; and $e(\rho) = \sum_{j = 1}^{g-1} n_{T_j} + \sum_{k = 1}^{g+m-2} n_{P'_k}
 + \sum_{i = 1}^m n_{P_i} + n_{P_\gamma} + 0 = n$ when $p$ is odd. By the definition of $\rho(a)$ and $\rho(b)$, we have $\rho\in \ELns(a,b)$. Finally, since all decomposition curves $\alpha_l$ and $\beta_k$ have hyperbolic image, 
 each of $\rho|_{\pi_1(T_j)}$ and $\rho|_{\pi_1(P_i)}$ is non-abelian and each $\rho(\beta_k)$ is non-trivial, i.e.,
 $\rho$ is contained in $\ELns(a,b)'$.
\end{proof}

To show the connectedness of $\mathcal{EL}^s_n(a,b)'$, we first introduce the following notation. For $n\in \mathbb{Z}$ and $s \in \{\pm1\}^p$, denote by $HP^{s,0}_n(\Sigma\setminus T)$ the space of representations $\rho: \pi_1(\Sigma\setminus T)\to \psl$ of relative Euler class $n$ 
such that $\rho(\gamma^{-1})$ is hyperbolic, and for $i\in \{1,\dots, p\},\ \rho(c_i)$ is in $\Par^{sgn(s_i)}$. 
Let $HP^{s,0}_n(\Sigma\setminus T)'$ be the subspace of $ HP^{s,0}_n(\Sigma\setminus T)$ consisting of $\rho$ such that 
the restrictions $\rho|_{\pi_1(P_i)}$ for $i\in \{1,\dots, m\}$,  $\rho|_{\pi_1(T_j)}$ for $j\in \{1,\dots, g-1\}$ are non-abelian, and for each $k\in \{1,\dots, g+m-2\}$, none of the boundary components of $P'_k$ maps to the identity. Observe that for each $[\rho]\in \ELns(a,b)'$, the restriction $\rho|_{\pi_1(\Sigma\setminus T)}$ lies in $HP^{s,0}_n(\Sigma\setminus T)'$.

For 
each $j\in \{1,\dots, g-1\}$, let $\widetilde{\rho(a_j)}$ and $\widetilde{\rho(b_j)}$ be arbitrary lifts of $\rho(a_j)$ and $\rho(b_j)$ in $\univcover$; and for each $i\in \{1,\dots, p\}$, let $\widetilde{\rho(c_i)}$ be the lift of $\rho(c_i)$ in $\Par_0$.
For $n \in \{\chi(\Sigma\setminus T)+ p_+(s) ,\dots, -\chi(\Sigma\setminus T) - p_-(s)\}$, we define a map  $\widetilde{e}_{s,n}:HP^{s,0}_n(\Sigma\setminus T)'\rightarrow \Hyp_n$ by 
$$\widetilde{e}_{s,n}(\rho) = 
\bigg(
\prod_{i = 1}^{g-1}\widetilde{\rho(a_i)}\widetilde{\rho(b_i)}\widetilde{\rho(a_i)^{-1}}\widetilde{\rho(b_i)^{-1}}
\bigg)
\cdot\widetilde{\rho(c_1)}\cdots\widetilde{\rho(c_p)}.$$
Indeed, by the definition of the relative Euler class, $e(\rho) = n$ implies $\widetilde{e}_{s,n}(\rho) = z^n\widetilde{\rho(\gamma)}$, where $\widetilde{\rho(\gamma)}$ is the lift of $\rho(\gamma)$ in $\Hyp_0$; hence $\widetilde{e}_{s,n}$ lies in $\Hyp_n$.
\\

The following Lemma \ref{plp_HP'} gives a main tool to show the connectedness of $\mathcal{EL}^s_n(a,b)'$.
\begin{lemma}\label{plp_HP'}
    Let $\Sigma = \Sigma_{g,p}$ with $g\geqslant 2$ and $p\geqslant 1$, and $T$ be a subsurface of $\Sigma$ homeomorphic to a one-holed torus.
    Let $s\in \{\pm1\}^{p}$ and $n \in \{\chi(\Sigma\setminus T)+ p_+(s), \dots, -\chi(\Sigma\setminus T) - p_-(s)\}$. Then the map $\widetilde{e}_{s,n}:HP^{s,0}_n(\Sigma\setminus T)'\rightarrow \Hyp_n$ satisfies the strong path-lifting property. 
\end{lemma}

The proof of Lemma \ref{plp_HP'} can be found in Section \ref{plphp'}. Using Lemma \ref{plp_HP'}, we now prove Proposition \ref{connect_elab}.

\begin{proof}[Proof of Proposition \ref{connect_elab}]
As shown in the proof of Lemma \ref{density}, every non-empty $\mathcal{EL}^s_n(a,b)$ has positive measure, and $\ELns(a,b)'$ is a full-measure subset of $\ELns(a,b)$. Hence $\mathcal{EL}^s_n(a,b)$ is non-empty if and only if $\mathcal{EL}^s_n(a,b)'$ is non-empty. Therefore, the non-emptiness criterion directly follows from Lemma \ref{ELnonempty}.
\smallskip

By Lemma \ref{density}, it remains to show the connectedness of $\ELns(a,b)'$. Let $[\rho]$ and $[\rho']$ be arbitrary elements of $\ELns(a,b)'$. To construct a path in $\ELns(a,b)'$ connecting $[\rho]$ and $[\rho']$, we construct a path $\{\rho_t\}_{t\in [0,1]}$ of representations on $\pi_1(T)$ connecting $\rho_0 = \rho|_{\pi_1(T)}$ and $\rho_1 = \rho'|_{\pi_1(T)}$, then extend $\{\rho_t\}\interval$ to the entire $\pi_1(\Sigma)$. By the argument in the proof of
\cite[Lemma~6.6]{Marche-Wolff}, the restrictions
$\rho|_{\pi_1(T)}$ and $\rho'|_{\pi_1(T)}$
can be joined by a path $\{\rho_t\}_{t\in[0,1]}$ in the 
space of representations of \(\pi_1(T)\), where for all $t\in [0,1]$, $\rho_t(a)$ and $\rho_t(b)$ are non-commuting elliptic elements, and with at least one of them of infinite
order. By Lemma \ref{lem_torus}, the path $A_t=\rho_t(\gamma)$ lies in $\Hyp$. For each $t\in [0,1]$, let $\widetilde{A^n_t}$ be a lift of $A_t$ in $\Hyp_n$. Since the restrictions $\rho|_{\pi_1(\Sigma\setminus T)}, \rho'|_{\pi_1(\Sigma\setminus T)} \in HP^{s,0}_n(\Sigma\setminus T)'$, by Lemma \ref{plp_HP'}, we may lift $\{\widetilde{A^n_t}\}\interval$ to a path $\{\phi_t\}_{t\in[0,1]}\in HP^{s,0}_n(\Sigma\setminus T)'$ such that $\phi_0=\rho|_{\pi_1(\Sigma\setminus T)}$ and $\phi_1=\rho'|_{\pi_1(\Sigma\setminus T)}$. Since $\rho_t$ and $\phi_t$ agree on $\gamma$, by letting $\rho_t|_{\pi_1(\Sigma\setminus T)} = \phi_t$, we may extend $\rho_t$ to a path of representation on $\pi_1(\Sigma)$ which, by abuse of notation, we denote as $\rho_t$ again. Then $[\rho_t]\in \mathcal{EL}^s_n(a,b)'$ connecting $[\rho]$ and $[\rho']$.
\end{proof}

\subsection{Proof of Lemma \ref{plp_HP'}}\label{plphp'}

We first fix the terminology used in the proof. Let $P:\univcover\times \univcover \rightarrow\univcover$ be the product map on the universal cover $\univcover$ of $\psl$. Let $Z$ denote the center of $\univcover$. Denote by 
$$\mathscr{L} := \{\Hyp_k,\Ell_k,\Par_k^+,\Par_k^- \mid k\in\mathbb Z\}.$$
We call a non-empty open connected subset of $\univcover\setminus Z$ an \emph{interval in $\univcover$} if it is a union of elements of $\mathscr{L}$. The set $\mathcal{I}\setminus \{\mathrm{I}\}$ in Theorem \ref{thm_liftcommu} provides an example of an interval in $\univcover$.
\medskip

Recall from Section \ref{sec_preliminary} that the character map $\chi: \SL\times\SL\to \mathbb{R}^3$ is given by
    $$
    \chi(A,B) =
    \big(
    Tr(A), 
    Tr(B), 
    Tr(AB)\big).
    $$
By Proposition \ref{prop_char}, the composition $\kappa\circ \chi$ with the polynomial map $\kappa(x,y,z) := x^2 + y^2 + z^2 - xyz - 2$ gives the trace of the commutator $\tr[A,B]$. By abuse of notation, we also consider $\chi$ as a map defined on $\univcover\times \univcover$ by letting
$$
    \chi(\widetilde{A},\widetilde{B}) =
    \big(
    Tr(\widetilde{A}), 
    Tr(\widetilde{B}), 
    Tr(\widetilde{A}\widetilde{B})\big).
    $$
    
 \begin{lemma}\label{Prod_open}   
     Let $I_1$ and $I_2$ be intervals in $\univcover$, and let $D := (I_1\times I_2)\setminus P^{-1}(Z)$. Then the image $P(D)$ is also an interval in $\univcover$.
\end{lemma}
\begin{proof}
    As the restriction of the product map to the open set $D$, the map $P:D\rightarrow P(D)$ is a submersion. 
    Thus, the map $P$ is open, hence the image $P(D) = P(I_1\times I_2)\setminus Z$ is open. 
    Moreover, since both $I_1$ and $I_2$ are open and connected, the image $P(I_1\times I_2)$ is a connected open subset of the $3$-manifold $\univcover$. Then, since $Z$ is discrete in $\univcover$, the complement $P(D) = P(I_1\times I_2)\setminus Z$ is connected.
    It remains to show that $P(D)$ is a union of elements of $\mathscr{L}$. To this end, for any $S\in \mathscr{L}$, we show that the intersection $S\cap P(D)$ is either an empty set or $S$ itself. For given $k\in \mathbb{Z}$, we will consider separately the cases $S = \Hyp_k$ or $\Ell_k$  and $S = \Par^+_k$ or $\Par^-_k$.
    \medskip
    
    We first assume that $\Hyp_k\cap P(D)$ is non-empty. Note that if $\Hyp_k\cap P(D)$ contains an element $\widetilde{H}$, then for any $\widetilde{g}\in \univcover$, the conjugate $\widetilde{g}\widetilde{H}\widetilde{g}^{-1}$ of $\widetilde{H}$ is also contained in $\Hyp_k\cap P(D)$. 
    Indeed, if $(\widetilde{A}, \widetilde{B})\in I_1\times I_2$ maps to $\widetilde{H}$ under $P$, then since every interval is invariant under conjugation,
    we have $\widetilde{g}\widetilde{A}\widetilde{g}^{-1}\in I_1$ and $\widetilde{g}\widetilde{B}\widetilde{g}^{-1}\in I_2$; and we have $P(\widetilde{g}\widetilde{A}\widetilde{g}^{-1}, 
    \widetilde{g}\widetilde{B}\widetilde{g}^{-1}) = \widetilde{g}\widetilde{H}\widetilde{g}^{-1}$. Therefore, it suffices to show that every conjugacy class in $\Hyp_k$ has a
    representative lying in $P(D)$.
    Since the conjugacy classes in $\Hyp_k$ are exactly the fibers of the trace map $Tr: \Hyp_k \to \mathbb{R}$, we show that for each $z'\in Tr(\Hyp_k)$, there exists an $\widetilde{H}'\in \Hyp_k\cap P(D)$ such that $\tr(\widetilde{H}')=z'$.
    \smallskip
    
    
    
    Let $\widetilde{H}\in \Hyp_k \cap P(D)$, and let $\widetilde{A}\in I_1,\widetilde{B}\in I_2$ such that $P(\widetilde{A},\widetilde{B})=\widetilde{H}$. 
    Let $A,B,H$ be the projection of $\widetilde{A},\widetilde{B},\widetilde{H}$, respectively, to $\SL$. Without loss of generality, we can assume $[A,B]\neq \mathrm{I}$. Indeed, if $[A,B] = \mathrm{I}$, then since $I_1$ and $I_2$ are open in $\univcover$, we can choose a path $\{\widetilde{B_t}\}\interval$ in $I_2$ starting from $\widetilde{B_0} = \widetilde{B}$ such that the path $\{\widetilde{H}\widetilde{B_t}^{-1}\}\interval$ is contained in $I_1$,  and the projection $B_1$ of $\widetilde{B_1}$ satisfies $[HB_1^{-1}, B_1]\neq \mathrm I$; and we can replace $(\widetilde{A}, \widetilde{B})$ with $(\widetilde{H}\widetilde{B_1}^{-1}, \widetilde{B_1})$. 
    Then, as the pair $(A,B)$ is contained in $\Omega$ in Lemma \ref{lem_plpchar}, we will use Lemma \ref{lem_plpchar} to construct a path $\{(\widetilde{A_t}, \widetilde{B_t})\}\interval$ in $D$ such that the product $P(\widetilde{A_1}, \widetilde{B_1})$ has trace $z'$. 
    \smallskip
    
    If $k$ is even, then we have $Tr(\Hyp_k) = (2,\infty)$. Let $z'\in (2,\infty)$, and $\chi(A,B)=(x,y,z)$ where $z = Tr(H) \in (2,\infty)$. By letting $\chi_t := (x,y,(1-t)z+ tz')$ for $t\in [0,1]$,
    we define a path $\{\chi_t\}\interval$ in $\mathbb{R}^3$ connecting $\chi_0 = (x,y,z)$ and $\chi_1 = (x,y,z')$.  Since $(1-t)z+ tz' \in (2,\infty)$ for all $t\in [0,1]$, the path $\{\chi_t\}\interval$ is contained in $ \mathbb{R}^3\setminus \big([-2,2]^3 \cap \kappa^{-1}([-2,2])\big)$. As $[A,B]\neq \mathrm I$, by Lemma \ref{lem_plpchar}, there is a path $\{(A_t,B_t)\}_{t\in[0,1]}$ in $\SL\times\SL$ starting at $(A_0,B_0) =(A,B)$ such that for all $t\in [0,1]$, $\chi(A_t,B_t) = \chi_t$. Let $\{(\widetilde{A_t},\widetilde{B_t})\}_{t\in[0,1]}$ be a lift of the path $\{(A_t,B_t)\}_{t\in[0,1]}$ to $\univcover\times \univcover$ starting at $(\widetilde{A_0},\widetilde{B_0}) =(\widetilde{A},\widetilde{B})$.
    Since $\{\widetilde{A_t}\}\interval$ is a path of non-central elements of the same trace $x$, both $\widetilde{A_1}$ and $\widetilde{A_0} = \widetilde{A}$ are contained in the same element in $\mathscr{L}$. Similarly, $\widetilde{B_1}$ and $\widetilde{B}$ are in the same element in $\mathscr{L}$.
    Therefore, the pair $(\widetilde{A_1},\widetilde{B_1})$ is also in $I_1\times I_2$. Moreover, since $Tr(\widetilde{A_t}\widetilde{B_t})>2$ for all $t\in [0,1]$, we have $\widetilde{A_1}\widetilde{B_1}\in \Hyp_k$. Since $Tr(\widetilde{A_1}\widetilde{B_1})=z'$,
    setting $\widetilde{H}' = \widetilde{A_1}\widetilde{B_1}$ gives us the desired result. 
    If $k$ is odd, then we have $Tr(\Hyp_k) = (-\infty, -2)$, and the proof follows verbatim from the case when $k$ is even.
    \medskip

    Next, we assume that $\Ell_k\cap P(D)$ contains an element $P(\widetilde{A},\widetilde{B})$ where $\widetilde{A}\in I_1,\widetilde{B}\in I_2$. 
    As in the case $S = \Hyp_k$, it suffices to show that every conjugacy class in $\Ell_k$ intersects $P(D)$, which we will show using Lemma \ref{lem_plpchar} assuming $[A,B] \neq \mathrm I$.
    Let $(x,y,z) = \big(Tr(\widetilde{A}), Tr(\widetilde{B}),  Tr(\widetilde{A}\widetilde{B})\big)$, and let $z'\in Tr(\Ell_k) = (-2,2)$. 
    If 
    $|x|>2$ or $|y|>2$, then we can define the path $\{\chi_t\}\interval$ in $ \mathbb{R}^3\setminus \big([-2,2]^3 \cap \kappa^{-1}([-2,2])\big)$ as in the previous paragraph, and the proof follows verbatim. If $|x|\leqslant 2$ and $|y|\leqslant 2$, then
    $\chi(A,B) = (x,y,z)$ lies in $[-2,2]^3 \setminus \kappa^{-1}([-2,2])$, 
    which has four connected components (see \cite[Theorem 4.3 and Figure 2]{goldman}).
    Note that for every $z' \in (-2,2)$, there exist $x', y'\in (-2,2)$ such that the point $(x',y',z')$ lies in the same connected component as $(x,y,z)$.
    Then there is a path $\{\chi_t\}_{t\in[0,1]}$ in  
    the connected component  connecting $\chi_0=(x,y,z)$ and $\chi_1=(x',y',z')$. 
    As in the previous paragraph, by Lemma \ref{lem_plpchar}, we can lift $\{\chi_t\}\interval
    $ to a path $\{(\widetilde{A_t}, \widetilde{B_t})\}\interval$ which stays in $I_1\times I_2$; and the rest of the proof follows verbatim.
     \medskip
     
     Finally, for $s\in \{\pm\}$, we assume that $\Par^s_k\cap P(D)$ contains an element $\widetilde{P} = P(\widetilde{A},\widetilde{B})$ where $\widetilde{A}\in I_1,\widetilde{B}\in I_2$. 
     Any other element $\widetilde{P'}\in \Par_k^s$ is conjugate to $\widetilde{P}$, i.e., there is a $\widetilde{g}\in \univcover$ such that $\widetilde{P}=\widetilde{g}\widetilde{P'}\widetilde{g}^{-1}$. 
     Then 
     $\widetilde{P'} = P\big(\widetilde{g}^{-1}\widetilde{A}\widetilde{g},\widetilde{g}^{-1}\widetilde{B}\widetilde{g})$ is also in $P(D)$.
     This completes the proof.
\end{proof}

    \begin{lemma}\label{Prod_confib}
    Let $P: \univcover\times\univcover\to\univcover$ be the product map.
    \begin{enumerate}[(a)]
        \item 
         Let $I_1$ and $I_2$ be intervals in $\univcover$, and let $D := (I_1\times I_2)\setminus P^{-1}(Z)$. Then the product map $P:D\rightarrow P(D)$  has path-connected fibers and satisfies the strong path-lifting property.
    \item 
        Let $I$ be an interval in
        $\univcover$. For $l\in \mathbb{Z}$ and $s\in \{\pm\}$, let 
        $D := (I\times \Par_0^s )\cap P^{-1}(\Hyp_l)$.
        Then the product map $P:D \rightarrow P(D)$ has path-connected fibers and satisfies the strong path-lifting property.
    \end{enumerate}
    \end{lemma}

For the proof of Lemma \ref{Prod_confib}, 
we refer to open subsets of $\univcover \times \univcover$ of the form $$\Hyp_m\times \Hyp_n,\Ell_m \times \Ell_n,\text{ or }\Hyp_m\times \Ell_n \text{ for }n, m\in \mathbb{Z}$$ as \emph{blocks}. For $\widetilde{C}\in P(D)$, we will first show that the fiber $P^{-1}(\widetilde{C})$ is path-connected in the blocks in $D$. Letting $l\in \mathbb{Z}$ and $s\in \{\pm\}$, the proof for the cases $\widetilde{C}\in \Hyp_l, \widetilde{C}\in \Ell_l,$ and $\widetilde{C}\in \Par^s_l$ will be given in 
Subsections \ref{hypfibsec}, \ref{ellfibsec} and \ref{parfibsec} respectively. In Subsection \ref{sec_glue}, we will construct a path in $P^{-1}(\widetilde{C})$ that connects two points in different blocks, which together will imply the connectedness of  $P^{-1}(\widetilde{C})$ in $D$.

\subsubsection{Fibers of Hyperbolic Elements}\label{hypfibsec}

The following Lemma \ref{Hypfiber} shows the path-connectedness of $P^{-1}(\widetilde{C})$ in the blocks for hyperbolic $\widetilde{C}$.
\begin{lemma}\label{Hypfiber}
Let $P: \univcover\times \univcover\to \univcover$ be a product map, and let $m, n, l\in \mathbb{Z}$.
Consider the following restrictions of $P$:
\begin{align*}
P_1:&\ (\Hyp_m\times \Hyp_n) \cap P^{-1}(\Hyp_l) \to \Hyp_l,
\\
P_2: &\ (\Hyp_m\times \Ell_n) \cap P^{-1}(\Hyp_l) \to \Hyp_l, \text{ and} 
\\
P_3: &\ (\Ell_m \times \Ell_n) \cap P^{-1}(\Hyp_l) \to \Hyp_l.
\end{align*}
For each $i\in \{1,2,3\}$, every non-empty fiber of $P_i$ is path-connected.
\end{lemma}
\begin{proof}
 We first prove the lemma for $P_1:(\Hyp_m \times \Hyp_n) \cap P^{-1}(\Hyp_l) \to \Hyp_l$.
 Let $\widetilde{C} \in \Hyp_l$ 
 such that $P_1^{-1}(\widetilde{C})$ is non-empty.
 Since $P^{-1}(\widetilde{C}) \cap (\Hyp_m\times \Hyp_n)$ is homeomorphic to $P^{-1}(z^{-m-n}\widetilde{C})\cap (\Hyp_0\times \Hyp_0)$, by Proposition \ref{gol4.6}, $P_1^{-1}(\widetilde{C})$ is path-connected.
 \smallskip

  We now show the lemma for $P_2$; and the proof for $P_3$ follows verbatim. Let $\widetilde{C}\in \Hyp_l$, and let $(\widetilde{A_1},\widetilde{B_1})$ and $(\widetilde{A_2},\widetilde{B_2})$ be in the fiber $P_2^{-1}(\widetilde{C})$. Let $(A_1,B_1)$ and $(A_2,B_2)$ respectively be their projections to $\SL$.

 \smallskip

 First, we construct a path connecting $(\widetilde{A_1}, \widetilde{B_1})$ to $(\widetilde{A'_2}, \widetilde{B'_2})$ in $\Hyp_m\times \Ell_n$ such that $\chi(\widetilde{A'_2}, \widetilde{B_2'})=\chi(\widetilde{A_2}, \widetilde{B_2})$, where $\chi$ is the character map. To this end, we use the path-lifting property of $\chi$ in Lemma \ref{lem_plpchar}.
 Let $\chi(A_1,B_1) = (x_1,y_1,z)$ and $\chi(A_2,B_2) = (x_2,y_2,z)$, where $z = Tr(\widetilde{C})$. For $t\in [0,1]$, let $x^t = x_2t+(1-t)x_1$, $y^t = y_2t + (1-t)y_1$, $z^t =z$, and let $\chi_t=(x^t,y^t,z)$. This defines a path $\{\chi_t\}_{t\in[0,1]}$ connecting $(x_1,y_1,z)$ and $(x_2,y_2,z)$. 
 Note that both $x_1$ and $x_2$ lie in the interval which is either $(-\infty,-2)$ or $(2,\infty)$; and we have $y_1, y_2 \in (-2,2)$ and $|z|>2$. Therefore, $(x_i,y_i,z) \in \mathbb{R}^3\backslash ([-2,2]^3 \cap \kappa^{-1}[-2,2])$ for $i\in \{1,2\}$, and the path $\{\chi_t\}\interval$ is contained in $ \mathbb{R}^3\backslash ([-2,2]^3 \cap \kappa^{-1}[-2,2])$. Notice that $A_1$ is hyperbolic and $B_1$ is elliptic, thus $[A_1,B_1]\neq \mathrm I$. By Lemma \ref{lem_plpchar}, we may lift $\{\chi_t\}\interval$ starting at $({A_1},{B_1})$, which we can further lift to a path $\{\widetilde{\chi_t}\}\interval$ in $ \univcover\times \univcover$ starting at $(\widetilde{A_1},\widetilde{B_1})$, where we denote by $(\widetilde{A_2'}, \widetilde{B_2'})$ as its endpoint. Since $x_t$ lies in $(-\infty,-2)$ or $(2,\infty)$ and $y_t \in (-2,2)$ for every $t\in [0,1]$, the path $\{\widetilde{\chi_t}\}\interval$ is contained in $\Hyp_m \times \Ell_n$.
 \smallskip
 
 Next, we construct a path connecting $(\widetilde{A_2'}, \widetilde{B_2'})$ to $(\widetilde{A_2}, \widetilde{B_2})$ in $\Hyp_m\times \Ell_n$.
 Note that $\kappa(x_2,y_2,z)\neq 2$, as otherwise $y_2\in (-2,2)$ implies $x_2\in [-2,2]$. Then by proposition \ref{prop_char},  $(A_2',B_2')$ is $\GL$-conjugate to $(A_2,B_2)$. 
 Let $B_2(1,2)$ and $B_2'(1,2)$ denote the (1,2)-entries of the matrices $B_2$ and $B'_2$, respectively. As (1,2)-entries of elliptic elements, $B_2(1,2)$ and $B_2'(1,2)$ are nonzero. 
 If $B_2$ and $B_2'$ are $\GL\backslash \SL$-conjugate, then $sgn(B_2(1,2)) = -sgn(B_2'(1,2))$. Then by Lemma \ref{lem_offdiag_Ell}, this contradicts that they are both projections of elements in $\Ell_n$. Hence, they must be $\SL$-conjugate. Let $g \in \SL$ such that $g(A_2',B_2')g^{-1} = (A_2,B_2)$ and let $\{g_t\}\interval$ be a path in $\SL$ from $\mathrm{I}$ to $g$. Let $\{\widetilde{g_t}\}\interval$ be the lift of $\{g_t\}\interval$ in $\univcover$ starting from $\widetilde{g_0} = \mathrm{I} \in \univcover$. Then $\big\{\widetilde{g_t}(\widetilde{A_2}',\widetilde{B_2}')\widetilde{g_t}^{-1}\big\}\interval$ is a path in $\Hyp_m \times \Ell_n$ from $(\widetilde{A_2}',\widetilde{B_2}')$ to $(\widetilde{A_2},\widetilde{B_2})$. Composing this path with $\widetilde{\chi_t}$, we get a path $\{(\widetilde{P_t}^1,\widetilde{P_t}^2)\}\interval$ from $(\widetilde{A_1},\widetilde{B_1})$ to $(\widetilde{A_2},\widetilde{B_2})$. 
 \smallskip

Finally, using the path $\big\{(\widetilde{P_t}^1,\widetilde{P_t}^2)\big\}\interval$, we construct a path in $P^{-1}(\widetilde{C})$ connecting $(\widetilde{A_1}, \widetilde{B_1})$ and $(\widetilde{A_2}, \widetilde{B_2})$.
Let $\widetilde{C_t} = \widetilde{P_t}^1\widetilde{P_t}^2$for $t\in [0,1]$ which defines a path starting from $\widetilde{C}$. From the construction above, $Tr(\widetilde{C}_t) =z$ for all $t\in [0,1]$, thus $\{\widetilde{C_t}\}\interval$ lies in the same conjugacy class of $\widetilde{C}\in \Hyp_l$. Then by Lemma \ref{lem_conjugacypath_const}, we have a path $\{h_t\}\interval \subset \psl$ such that $h_0 = \pm \mathrm{I}$, and $\pm C_t = \pm h_tCh_t^{-1}$ for all $t\in [0,1]$. 
 Letting $\{\widetilde{h}_t\}\interval$ be the lift of $\{h_t\}\interval$ starting at $\mathrm{I}\in \univcover$,  $\{\widetilde{h_t}^{-1}\widetilde{P}_t^1\widetilde{P}_t^2\widetilde{h}_t\}\interval$ is a path from $\widetilde{C}$ to $\widetilde{h_1}^{-1}\widetilde{C}\widetilde{h_1}$. Notice that this is a lift of the constant path $\pm C$ starting at $\widetilde{C}$ and is hence equal to $\widetilde{C}$ for all $t\in [0,1]$. 
 Then the path $\{\widetilde{h_t}^{-1}(\widetilde{P_t}^1,\widetilde{P_t}^2)\widetilde{h_t}\}\interval$ connects $(\widetilde{A_1},\widetilde{B_1})$ and $\widetilde{h_1}^{-1}(\widetilde{A_2},\widetilde{B_2})\widetilde{h_1}$ in $P_2^{-1}(\widetilde{C})$. Since $\widetilde{h_1}^{-1}\widetilde{A_2}\widetilde{B_2}\widetilde{h_1} = \widetilde{h_1}^{-1}\widetilde{C}\widetilde{h_1}=\widetilde{C}$, we have $\pm h_1^{-1}Ch_1 = \pm C$. Thus, $\pm h_1$ commutes with $\pm C$. 
 Let $\widetilde{h'_1}$ be the lift of $h_1$ in $\Hyp_0$, and let $\{\widetilde{h'_t}\}\interval$ be a path connecting $\mathrm{I}$ and $\widetilde{h'_1}$ in the one-parameter subgroup of $\univcover$ generated by $\widetilde{h'_1}$.
 Since $\widetilde{h'_1} = z^k\widetilde{h_1}$ for some $k\in \mathbb{Z}$, 
 $\widetilde{h'_1}^{-1}(\widetilde{A_2},\widetilde{B_2})\widetilde{h'_1} = \widetilde{h_1}^{-1}(\widetilde{A_2},\widetilde{B_2})\widetilde{h_1}$; and the path $\{\widetilde{h'_{1-t}}^{-1}(\widetilde{A_2},\widetilde{B_2})\widetilde{h'_{1-t}}\}\interval$ lies in $P_2^{-1}(\widetilde{C})$, connecting $\widetilde{h_1}^{-1}(\widetilde{A_2},\widetilde{B_2})\widetilde{h_1}$ to $(\widetilde{A_2},\widetilde{B_2})$. Composing this path with the path $\{\widetilde{h_t}^{-1}(\widetilde{P_t}^1,\widetilde{P_t}^2)\widetilde{h_t}\}\interval$ we get the desired path from $(\widetilde{A_1},\widetilde{B_1})$ to $(\widetilde{A_2},\widetilde{B_2})$ in  $P_2^{-1}(\widetilde{C})$. 
 \end{proof}

\subsubsection{Fibers of Elliptic Elements}\label{ellfibsec}
The following Lemma \ref{ellfiber} shows the path-connectedness of $P^{-1}(\widetilde{C})$ in the blocks for elliptic $\widetilde{C}$.
\begin{lemma}\label{ellfiber}
Let $P: \univcover\times \univcover\to \univcover$ be a product map, and let $m, n, l\in \mathbb{Z}$.
Consider the following restrictions of $P$:
\begin{align*}
P_1:\ & (\Hyp_m\times \Hyp_n) \cap P^{-1}(\Ell_l) \to \Ell_l,
\\
P_2:\ & (\Hyp_m\times \Ell_n) \cap P^{-1}(\Ell_l) \to \Ell_l, \text{ and}
\\
P_3:\ & (\Ell_m \times \Ell_n) \cap P^{-1}(\Ell_l) \to \Ell_l,
\end{align*}
For each $i\in \{1,2,3\}$, every non-empty fiber of $P_i$ is path-connected.
\end{lemma}
\begin{proof}

The proof of the lemma for $P_1$ and $P_2$ follows verbatim from Lemma \ref{Hypfiber}. 
Here we show the lemma for 
$P_3:\Ell_m \times \Ell_n \cap P^{-1}(\Ell_l) \rightarrow \Ell_l$.
\medskip

Since $(\Ell_m \times \Ell_n) \cap P^{-1}(\Ell_l)$ is homeomorphic to $(\Ell_1\times \Ell_1)\cap P^{-1}(\Ell_{l+k})$ for some $k\in \mathbb{Z}$, it suffices to show the lemma for $P_3: (\Ell_1 \times \Ell_1) \cap P^{-1}(\Ell_l)\to \Ell_l$ for all $l\in \mathbb{Z}$.
 Let $P_3:(\Ell_1\times \Ell_1)\cap P^{-1}(\Ell_l) \rightarrow \Ell_l$ and $\widetilde{C}\in \Ell_l$. Further let $(\widetilde{A}_1,\widetilde{B}_1)$ and $(\widetilde{A}_2,\widetilde{B}_2)$ be two points in the fiber $P_3^{-1}(\widetilde{C})$. 
 Let $\chi(\widetilde{A_i},\widetilde{B_i}) = (x_i,y_i,z)$, where $|x_i|<2,|y_i|<2$, and $z = Tr(\widetilde{C})$.
The complement of $[-2,2]^3\cap \kappa^{-1}([-2,2])$ within the slice $[-2,2]^2\times \{z\}$ has four connected components. See Figure \ref{ellfib}. To construct a path in $P_3^{-1}(\widetilde{C})$ connecting $(\widetilde{A}_1,\widetilde{B}_1)$ and $(\widetilde{A}_2,\widetilde{B}_2)$, we will first show that $(x_1,y_1,z)$ and $(x_2,y_2,z)$ lie in the same connected component, then lift the path connecting $(x_1,y_1,z)$ and $(x_2,y_2,z)$ using Lemma \ref{lem_plpchar}.
\medskip

First, we assume $\kappa(x_i,y_i,z) \neq 2$ for each $i\in \{1,2\}$. 
By Proposition \ref{prop_prodimage}, $(\Ell_1 \times \Ell_1) \cap P^{-1}(\Ell_l)$ is non-empty only if $l \in\{1,2\}$.
We first consider the case when $l=1$,
for two cases $0\leqslant z < 2$ and $-2 < z \leqslant 0$ separately.
\begin{figure}[h!]
    \centering
    \vspace{-1cm}
    \begin{overpic}[width=0.8\textwidth]{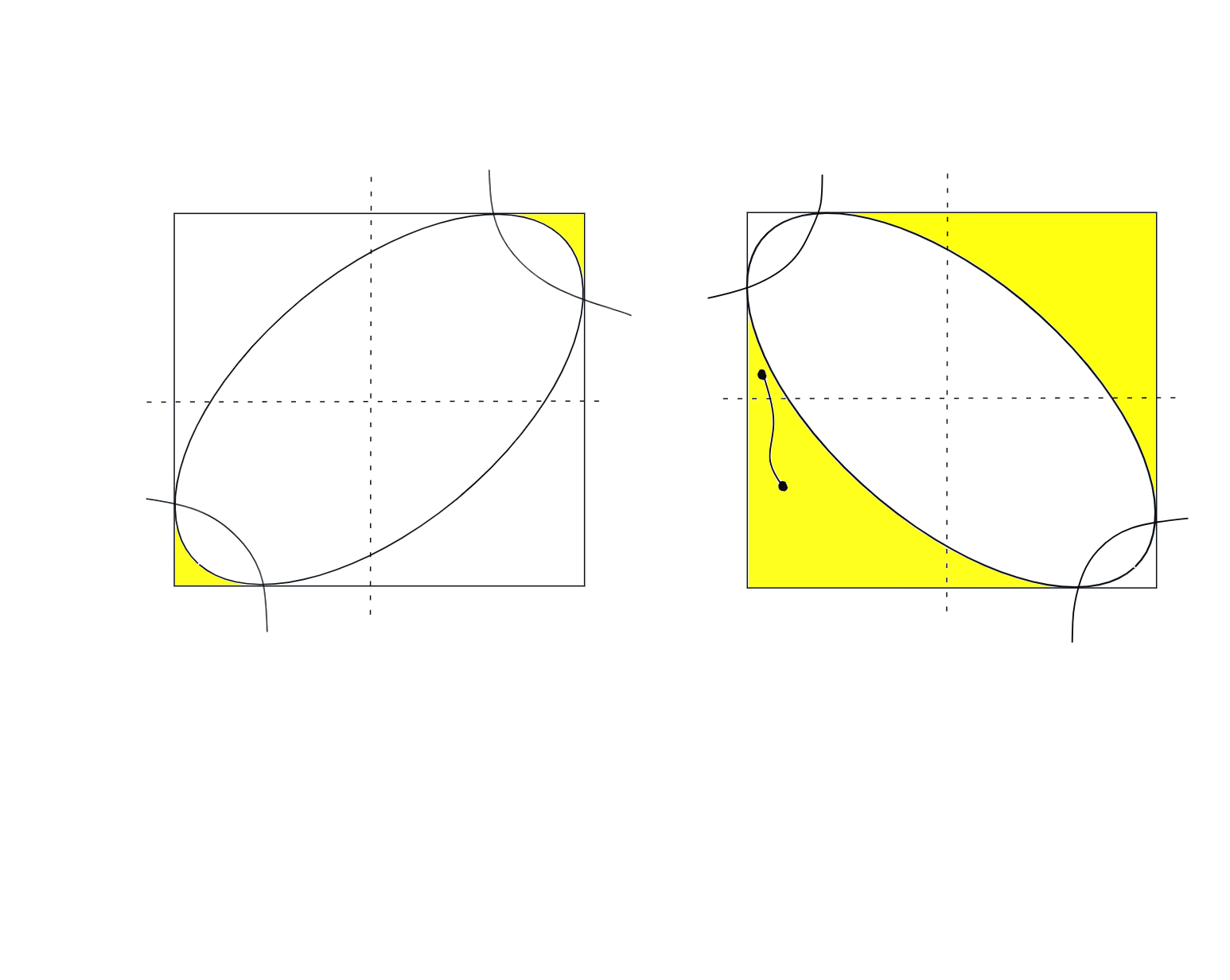}
        \put(26,24){\scalebox{0.7}{$\mathbf{(a)\ 0\leqslant z <2}$}}
        \put(8,63){\scalebox{0.55}{$(-2,2,z)$}}
        \put(47,63){\scalebox{0.55}{$(2,2,z)$}}
        \put(14.5,33){\scalebox{0.55}{$R_2$}}
        \put(7.5,30){\scalebox{0.55}{$(-2,-2,z)$}}
        \put(22,30){\scalebox{0.55}{$(-z,-2,z)$}}
        \put(6.3,40){\scalebox{0.55}{$(-2,-z,z)$}}
        \put(34.4,63){\scalebox{0.55}{$(z,2,z)$}}
        \put(48,53){\scalebox{0.55}{$(2,z,z)$}}
        \put(45.5,60.5){\scalebox{0.55}{$R_1$}}
        \put(47,30){\scalebox{0.55}{$(2,-2,z)$}}
        \put(31,48){\scalebox{0.55}{$(0,0,z)$}}
        \put(71,24){\scalebox{0.7}{$\mathbf{(b) -2< z\leqslant 0}$}}
        \put(53.5,30){\scalebox{0.55}{$(-2,-2,z)$}}
        \put(53,57){\scalebox{0.55}{$(-2,-z,z)$}}
        \put(63,35){\scalebox{0.55}{$R_2$}}
        \put(63,38){\scalebox{0.55}{$\textbf{x}_2=(x',y',z)$}}
        \put(62.7,50){\scalebox{0.55}{$\textbf{x}_1=(x,y,z)$}}
        \put(93.5,30){\scalebox{0.55}{$(2,-2,z)$}}
        \put(79,30){\scalebox{0.55}{$(-z,-2,z)$}}
        \put(94.5,36){\scalebox{0.55}{$(2,z,z)$}}
        \put(54,63){\scalebox{0.55}{$(-2,2,z)$}}
        \put(67,63){\scalebox{0.55}{$(z,2,z)$}}
        \put(93.5,63){\scalebox{0.55}{$(2,2,z)$}}
        \put(78,48){\scalebox{0.55}{$(0,0,z)$}}
        \put(85,56){\scalebox{0.55}{$R_1$}}

    \end{overpic}
    \vspace{-3cm}
    \caption{The slice $[-2,2]^2\times \{z\}$}
    \label{ellfib}
\end{figure}
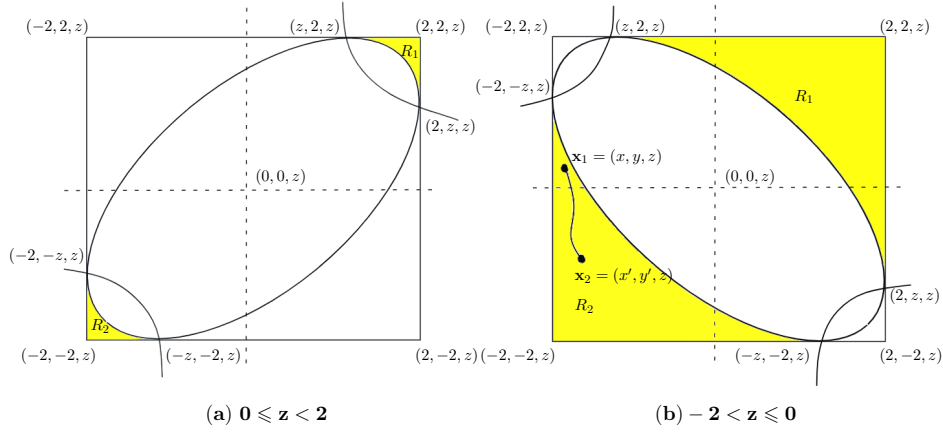
\[\text{Case 1}:0\leqslant z < 2\] 
For $i \in \{1,2\}$, 
let $A_i$ and $B_i$ be the projections of $\widetilde{A_i}$ and $\widetilde{B_i}$, respectively, to $\SL$.
Up to a simultaneous $\SL$-conjugation of $(A_1,B_1)$, we can assume that
        \[A_1 = \begin{bmatrix}
                \alpha_1 & \beta_1\\
                -\beta_1 & \alpha_1
                \end{bmatrix},
         B_1 = \begin{bmatrix}
                a_1 & b_1\\
                c_1 & d_1
         \end{bmatrix},     
         A_1B_1 = \begin{bmatrix}
                  a_1\alpha_1+c_1\beta_1 & b_1\alpha_1+d_1\beta_1\\
                  c_1\alpha_1 - a_1\beta_1 & d_1\alpha_1-b_1\beta_1 
         \end{bmatrix}\]
for some $\alpha_1, \beta_1, a_1, b_1, c_1, d_1\in \mathbb{R}$ such that $\alpha_1^2 + \beta_1^2 = a_1d_1 - b_1c_1 = 1$,
$2\alpha_1 = x_1$, $a_1+d_1 = y_1$, and $(a_1\alpha_1+c_1\beta_1) + (d_1\alpha_1-b_1\beta_1)= z$.
Since $\widetilde{A_1}, \widetilde{B_1}$ and $\widetilde{A_1}\widetilde{B_1}$ are in $\Ell_1$, by Lemma \ref{lem_offdiag_Ell}, we have $b_1>0,\ c_1<0$, $\beta_1>0$, and $b_1\alpha_1+d_1\beta_1>0$. 
Since $c_1-b_1<0$ and $\beta_1>0$, we have 
$$z = x_1y_1/2 + \beta_1(c_1-b_1) < x_1y_1/2.$$ Hence $(x_1,y_1,z)$ lies in the shaded region $R_1\cup R_2. $ See Figure \ref{ellfib}(a). 
Moreover, letting $(A_1B_1)_{ij}$ be the $(i,j)$-entry of $A_1B_1$, Lemma \ref{lem_offdiag_Ell} implies $(b_1-c_1)\alpha_1+(a_1+d_1)\beta_1 = (A_1B_1)_{12}-(A_1B_1)_{21}>0$. Since $b_1-c_1>0$ and $\beta_1>0$, we conclude $a_1+d_1$ and $\alpha_1$ cannot both be negative, hence $(x_1,y_1,z) = (2\alpha_1, a_1+d_1, z)$ must lie in $R_1$ in Figure \ref{ellfib}(a). Similarly, we can show that $(x_2,y_2,z)$ lies in $R_1$ which is path-connected. 
Therefore, there is a path $\{\chi_t\}\interval$ in $R_1$ connecting $(x_1,y_1,z)$ and $(x_2,y_2,z)$. For $i\in \{1,2\}$, by Proposition \ref{prop_char}, $\tr\big([A_i, B_i]\big) = \kappa(x_i, y_i,z) \neq 2$, which implies that $[A_i, B_i]\neq \mathrm{I}$. Therefore, as in 
the proof of Lemma \ref{Hypfiber}, we may use Lemma \ref{lem_plpchar} to lift $\{\chi_t\}\interval$
to construct
the path connecting $(\widetilde{A_1}, \widetilde{B_1})$ and $(\widetilde{A_2}, \widetilde{B_2})$ within the fiber.

\[\text{Case 2:} -2<z\leqslant 0\]

By the computation in Case 1, we can show that the point $(x_1,y_1,z)$ lies in $R_1 \cup R_2$. See Figure \ref{ellfib}(b). Assume $\textbf{x}_1 = (x_1,y_1,z)\in R_2$. Then, as shown in Case 1, $x_1$ and $y_1$ cannot be both negative, i.e., $\textbf{x}_1 = (x_1,y_1,z)$ lies in $R_2$ with $x_1\geqslant 0$ or $y_1\geqslant 0$. Then we can connect $\textbf{x}_1$ by a path $\{\chi^t\}\interval = \{(x^t,y^t,z)\}\interval$ in $R_2$ to a point $\textbf{x}_2= (x',y',z)\in R_2$ such that $x'<0$ and $y'<0$. See Figure \ref{ellfib}(b). Using Lemma \ref{lem_plpchar}, we can lift $\{\chi^t\}\interval$ to a path $\widetilde{\chi}^t=(\widetilde{A^t},\widetilde{B^t})$ starting from $(\widetilde{A_1},\widetilde{B_1})$;  and since $|x^t|<2,|y^t|<2$ and $|z|<2$ for all $t\in [0,1]$, the path $\{\widetilde{\chi}^t\}\interval$ lies in $\Ell_1\times \Ell_1$ and the path $\{\widetilde{A^t}\widetilde{B^t}\}\interval$ lies in $\Ell_1$. 
Let $\widetilde{\chi_1}=(\widetilde{A'},\widetilde{B'})$. Since $Tr(\widetilde{A}'\widetilde{B}')= Tr(\widetilde{C}) = z$ and both $(\widetilde{A'}\widetilde{B'})$ and $\widetilde{C}$ lie in $\Ell_1$, they are conjugate in $\univcover$ by an element $\tilde{g}$ i.e., $\tilde{g}\widetilde{A'}\widetilde{B'}\tilde{g}^{-1} = \widetilde{C}$. Thus, $\tilde{g}(\widetilde{A'},\widetilde{B'})\tilde{g}^{-1} \in \Ell_1\times \Ell_1$ lies in the fiber $P_3^{-1}(\widetilde{C})$ with $\chi({\widetilde{g}}(\widetilde{A'},\widetilde{B'})\widetilde{g}^{-1}) = (x',y',z)$. 
Again by the computation in Case 1, 
we can show that $x'\geqslant 0$ or $y'\geqslant 0$, which leads to a contradiction. Therefore, $\textbf{x}_1\in R_1$. Similarly, we can show that $(x_2,y_2,z)$ lies in $R_1$. 
Therefore, there is a path $\{\chi_t\}\interval$ in $R_1$ connecting $(x_1,y_1,z)$ and $(x_2,y_2,z)$. Using the path $\{\chi_t\}\interval$, the construction of the path connecting $(\widetilde{A_1}, \widetilde{B_1})$ and $(\widetilde{A_2}, \widetilde{B_2})$ within the fiber follows verbatim from the proof of Lemma \ref{Hypfiber}.
\medskip

When $l=2$, by Lemma \ref{lem_offdiag_Ell}, similar computations as in Case $l=1$ show that for $i\in \{1,2\}$, $\chi(\widetilde{A_i},\widetilde{B_i}) = (x_i,y_i,z)$ lies in the connected component $R_2$; and the proof follows verbatim from Case $l=1$.

Next, we assume $\kappa(x_i,y_i,z) = 2$ for some $i\in \{1,2\}$. Without loss of generality, we let $\kappa(x_1,y_1,z) = 2$. We first show that the level set $(\kappa\circ \chi)^{-1}(2)$ in $\Ell_1\times \Ell_1$ is the set of commuting pairs, which implies that $\widetilde{A_1}$ and $\widetilde{B_1}$ commute. If an elliptic $\SL$-pair $(A, B)$ lies in $(\kappa\circ \chi)^{-1}(2)$, then by Proposition \ref{prop_char}, both $A$ and $B$ preserve a common complex line $\mathbb{C}v\subset \mathbb{C}^2$. As real matrices, they also preserve the conjugate line $\mathbb{C}\bar{v}$. As $A$ and $B$ are elliptic, they don't have real eigenvectors, hence $v\neq\bar{v}$ and we have $\mathbb{C}^2 = \mathbb{C}v\oplus\mathbb{C}\bar{v}$. Therefore, both $A$ and $B$ are diagonal in the basis $(v,\overline{v})$, hence commute. Therefore, their lifts $\widetilde{A}, \widetilde{B}$ in $ \Ell_1$ also commute. Conversely, if $[\widetilde{A},\widetilde{B}]=\mathrm{I}$, then by Proposition \ref{prop_char}, we have $\kappa\big(\chi(\widetilde{A},\widetilde{B})\big)=2$.
\smallskip

We will find a path in $P_3^{-1}(\widetilde{C})$ connecting the commuting pair $(\widetilde{A_1}, \widetilde{B_1})$ to a non-commuting pair $(\widetilde{A'_1}, \widetilde{B'_1})$. Then, using $\kappa\big(\chi(\widetilde{A'_1}, \widetilde{B'_1})\big)\neq 2$, we will conclude from the previous case where $\kappa(x_i, y_i,z)\neq 2$ for each $i\in \{1,2\}$. 
 Since $\widetilde{A_1}$ and $\widetilde{B_1}$ are commuting elliptic elements, $\widetilde{A_1},\widetilde{B_1}$ and $\widetilde{C}$ all lie on the same one parameter subgroup generated by $\widetilde{C}$. Consider the restriction of the product map $P:\{\widetilde{C}\}\times \Ell_{-1}\rightarrow\univcover$. The image of this map is open, as it is homeomorphic to $\Ell_{-1}$. 
This allows us to choose a path $\{\widetilde{B^t}\}\interval \subset \Ell_1$ starting from $\widetilde{B_1}$ such that $[\widetilde{C},\widetilde{B^t}]\neq \mathrm{I}$ for all $t\in (0,1]$ and $\{\widetilde{C}\widetilde{B^t}^{-1}\}\interval \subset \Ell_1$. Let $(\widetilde{A_1'},\widetilde{B_1'}) = (\widetilde{C}\widetilde{B^1}^{-1},\widetilde{B^1}).$  Then $\big\{(\widetilde{C}\widetilde{B^t}^{-1},\widetilde{B^t})\big\}\interval$ is a path in $P_3^{-1}(\widetilde{C})$ whose endpoint $(\widetilde{A_1'},\widetilde{B_1'})$ lies in $\Ell_1\times \Ell_1$ and is non-commuting. Thus, for every commuting pair $(\widetilde{A_1},\widetilde{B_1)}$ in the fiber of $\widetilde{C}$, there is a path in the fiber connecting $(\widetilde{A_1},\widetilde{B_1})$ to some non-commuting pair $(\widetilde{A_1'},\widetilde{B_1'})$. Finally, since the space of non-commuting pairs (i.e. pairs $(\widetilde{A},\widetilde{B})$ such that $\kappa(\chi(\widetilde{A},\widetilde{B}))\neq 2$) in $P_3^{-1}(\widetilde{C})$ is path-connected, we conclude $P_3^{-1}(\widetilde{C})$ is path-connected.
\end{proof}

\subsubsection{Fibers of Parabolic Elements}\label{parfibsec}
The following Lemma \ref{parfiblemm} shows the path-connectedness of $P^{-1}(\widetilde{C})$ in the blocks for parabolic $\widetilde{C}$.

\begin{lemma}\label{parfiblemm}
Let $P: \univcover\times \univcover\to \univcover$ be a product map, and let $m, n, l\in \mathbb{Z}$, $s\in \{+,-\}$.
Consider the following restrictions of $P$:
\begin{align*}
P_1:&\ (\Hyp_m\times \Hyp_n) \cap P^{-1}(\Par^s_l) \to \Par^s_l,
\\
P_2: &\ (\Hyp_m\times \Ell_n) \cap P^{-1}(\Par^s_l) \to \Par^s_l, \text{ and} 
\\
P_3: &\  (\Ell_m \times \Ell_n) \cap P^{-1}(\Par^s_l) \to \Par^s_l.
\end{align*}
For each $i\in \{1,2,3\}$, every non-empty fiber of $P_i$ is path-connected.
\end{lemma}
\begin{proof}
 We first consider $P_1: (\Hyp_m\times \Hyp_n) \cap P^{-1}(\Par^s_l) \to \Par^s_l$. Let $\widetilde{C}\in \Par^s_l$, and for $i\in \{1,2\}$, let $(\widetilde{A_i},\widetilde{B_i}) \in P_1^{-1}(\widetilde{C})$ with its character $\chi(\widetilde{A_1},\widetilde{B_1}) = (x_i, y_i, c)$ where $|c| = 2$. If $\kappa(x_1,y_1,c)=2$, then since the set $(\kappa\circ \chi)^{-1}(2)$
 has codimension at least one in $\Hyp_m\times \Hyp_n$, we can choose a path $\{\widetilde{B}^t\}\interval \subset \Hyp_n$ starting at $\widetilde{B_1}$ such that $\big\{\widetilde{C}(\widetilde{B}^t)^{-1}\big\}\interval \subset \Hyp_m$ for all $t\in [0,1]$, and at $t=1$, $\kappa\big(\chi(\widetilde{C}(\widetilde{B^1})^{-1},\widetilde{B^1})\big)\neq 2$. As the path $\big\{(\widetilde{C}\widetilde{B_t}^{-1},\widetilde{B_t})\big\}\interval$ lies in the fiber of $\widetilde{C}$, we can replace $(\widetilde{A_1},\widetilde{B_1})$ with $\big(\widetilde{C}(\widetilde{B^1})^{-1},\widetilde{B^1}\big)$.
 Hence we can hereafter assume $\kappa\big(\chi(\widetilde{A_1},\widetilde{B_1})\big)\neq 2$.
Similarly, we assume $\kappa(\chi(\widetilde{A_2},\widetilde{B_2}))= \kappa((x_2,y_2,c))\neq 2$. For $t\in [0,1]$, let $\chi^t=(x^t,y^t,c)=\big((1-t)x_1+tx_2,(1-t)y_1+ty_2,c\big)$, which defines the path from $(x_1,y_1,c)$ to $(x_2,y_2,c)$. Since $|x_t|>2$ for all $t\in [0,1]$,
using Lemma \ref{lem_plpchar}, we may lift $\{\chi^t\}\interval$ to a path $\big\{(\widetilde{A^t},\widetilde{B^t})\big\}\interval$ of non-commuting pairs in $\Hyp_m\times \Hyp_n$ starting at $(\widetilde{A_1},\widetilde{B_1})$. 
Since $|Tr(\widetilde{A^t}\widetilde{B^t})| = |Tr(\widetilde{C})|= 2$ for all $t\in [0,1]$ and $\widetilde{A^0}\widetilde{B^0} = \widetilde{C}\in \Par^s_l$, the path $\{\widetilde{A^t}\widetilde{B^t}\}\interval$ must stay in $\Par_l \cup \{z^l\}$. Since $\widetilde{A^t}$ and $\widetilde{B^t}$ do not commute for all $t\in [0,1]$, the path $\{\widetilde{A^t}\widetilde{B^t}\}\interval$ never passes $z^l$, hence stays in $\Par^s_l$. Using the path $\big\{(\widetilde{A^t},\widetilde{B^t})\big\}\interval$, the construction of a path in $P_1^{-1}(\widetilde{C})$ connecting $(\widetilde{A_1}, \widetilde{B_1})$ and $(\widetilde{A_2}, \widetilde{B_2})$ follows verbatim as in Lemma \ref{Hypfiber}.
\medskip

In the case of $P_2$, let $(\widetilde{A_1},\widetilde{B_1})\in \Hyp_m\times \Ell_n$ be in the fiber of $\widetilde{C}$ and $\chi(\widetilde{A_1},\widetilde{B_1}) = (x_1,y_1,c)$ where $|c| = 2$. Since $|x_1|>2$ and $|y_1|<2$, we have $\kappa(x_1,y_1,c)\neq 2$. The rest of the construction of the path connecting $(\widetilde{A_1},\widetilde{B_1})$ to $(\widetilde{A_2},\widetilde{B_2})$ in the fiber of $P_2^{-1}(\widetilde{C})$ follows verbatim as $P_1$.
\medskip

     It remains to show the lemma for $P_3:(\Ell_m \times \Ell_n) \cap P^{-1}(\Par_l^s) \rightarrow \Par_l^s$.
     Since $ (\Ell_m \times \Ell_n) \cap P^{-1}(\Par_l^s)$ is homeomorphic to $ (\Ell_1\times \Ell_1)\cap P^{-1}(\Par_{l+k}^s)$ for some $k\in \mathbb{Z}$, it suffices to show the lemma for $P_3:  (\Ell_1 \times \Ell_1) \cap P^{-1}(\Par^s_l)\to \Par^s_l$ for all $l\in \mathbb{Z}$. Notice that $ (\Ell_1 \times \Ell_1) \cap P^{-1}(\Par^s_l)$ is non-empty only if $l = 1$. 
     Indeed, for $(\widetilde{A}, \widetilde{B})\in  (\Ell_1 \times \Ell_1) \cap P^{-1}(\Par^s_l)$ and $\widetilde{C} = P(\widetilde{A}, \widetilde{B})\in \Par^s_l$, we have $\widetilde{B}^{-1}\in \Ell_{-1}$, and 
     $\widetilde{A} = \widetilde{C}\widetilde{B}^{-1}$ lies in the set $P(\Par^s_l\times \Ell_{-1} )\cap \big(\bigcup_{k\in \mathbb{Z}\setminus \{0\}}\Ell_k\big)$, which equals 
    $z^{l-1}\cdot \Big(P(\Par^s_0\times \Ell_1 )\cap \big(\bigcup_{k\in \mathbb{Z}\setminus \{0\}}\Ell_k\big)\Big)$. By 
    Proposition \ref{prop_prodimage}, this set is contained in 
    $\Ell_l$ if $l > 0$; and contained in $\Ell_{l-1}$ if $l \leqslant 0$. Since $\widetilde{A_1}\in \Ell_1$, it follows that $l = 1$.
    Further notice that $\kappa\big(\chi(\widetilde{A}, \widetilde{B})\big)\neq 2$, as otherwise
    $\widetilde{A}$ and $\widetilde{B}$ commute as shown in the proof of Lemma \ref{ellfiber}, which contradicts that $\widetilde{A}\widetilde{B}$ is parabolic. 
    
    We show that the characters of all $(\widetilde{A},\widetilde{B})\in  P_3^{-1}(\widetilde{C})$
    are contained in the same path-connected component of $\big([-2,2]^2\times \{-2\}\big)\setminus\Gamma$. Then the path-connectedness of $P_3^{-1}(\widetilde{C})$ follows verbatim from the proof of Lemma \ref{ellfiber}. We will consider the two cases $s = +$ and $s = -$ separately.
 \begin{figure}[H]
    \centering
    \begin{overpic}[width=0.3\textwidth]{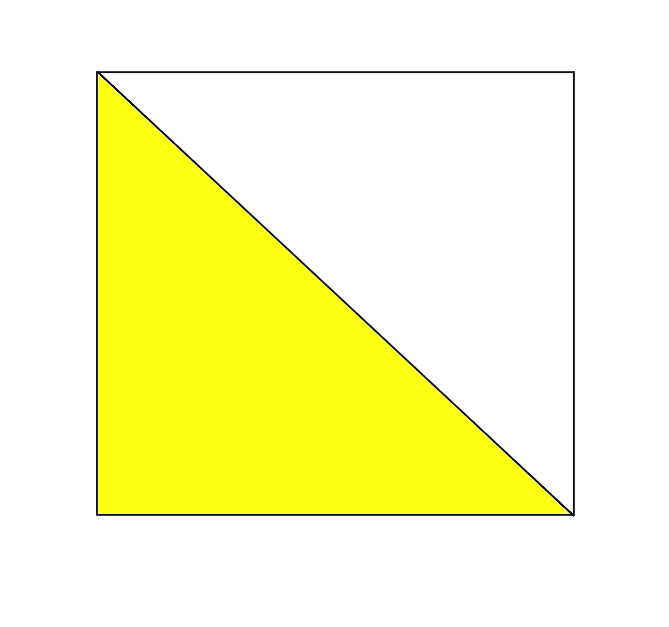}
     \put(35,35){\scalebox{1}{$R_1$}}
     \put(60,60){\scalebox{1}{$R_2$}}
     \end{overpic}
    \caption{Slices $z=-2$}
    \label{Parfib}
\end{figure}
We first consider $P_3:(\Ell_1\times \Ell_1)\cap P^{-1}(\Par^+_1) \rightarrow \Par_1^+$.
Let $(\widetilde{A},\widetilde{B})\in (\Ell_1\times \Ell_1)\cap P_3^{-1}(\widetilde{C})$ and $\chi(\widetilde{A},\widetilde{B})=(x,y,-2)$. Let the projections to $\SL$ be
 \[A = \begin{bmatrix}
          \alpha & \beta\\
          -\beta & \alpha
    \end{bmatrix},B = \begin{bmatrix}
                       a & b\\
                       c & d
    \end{bmatrix},AB = \begin{bmatrix}
        \alpha a + \beta c & b\alpha + d\beta\\
        c\alpha - a\beta & -b\beta + d\alpha
    \end{bmatrix}\]
    for some $\alpha,\beta,a,b,c,d \in \mathbb{R}$ such that $\alpha^2 +\beta^2=1, ad-bc=1$ and $|x|=|2\alpha| <2,|y|=|a+d|<2$.
     Since $\widetilde{A}\widetilde{B}=\widetilde{C}\in \Par_1^+$, by Lemma \ref{lem_offdiag}, we have $\alpha(c-b)-\beta(a+d) = (AB)_{21}-(AB)_{12}>0$. By Lemma \ref{lem_offdiag_Ell}, since $\beta>0$ and $c-b<0$, we conclude $x=2\alpha$ and $y=a+d$ cannot both be positive. If $(x,y,-2)$ lies in $R_2$, then as in Lemma \ref{ellfiber}, we can show that there exists $(\widetilde{A'},\widetilde{B}')\in (\Ell_1\times \Ell_1)\cap P_3^{-1}(\widetilde{C})$ such that $\chi(\widetilde{A'},\widetilde{B'})=(x',y',-2)$ for some $x'>0$ and $y'>0$, which gives a contradiction. Therefore, $(x,y,-2)$ lies in $R_1$.  
     \smallskip
     
    For $P_3:(\Ell_1\times \Ell_1)\cap P^{-1}(\Par^-_{1}) \rightarrow \Par^{-}_1$. In this case, by similar computations as in the case $s = +$, we can show that $\chi(\widetilde{A},\widetilde{B})$ lies in $R_2$ for all $(\widetilde{A},\widetilde{B})\in P_3^{-1}(\widetilde{C})$; and the proof follows verbatim.
\end{proof}

\subsubsection{Gluing the blocks}\label{sec_glue}

We first fix the terminology used in the proof.
Let $m\in \mathbb{Z}$. For $\mathbb{B}_m \in \{\Hyp_m,\ \Ell_m\}$, we define the spaces 
$\overline{\mathbb{B}}^L_m$ and 
$\overline{\mathbb{B}}^R_m$ as follows:
\begin{align*}
    \overline{\mathbb{B}}^L_m
    := \begin{cases}
    \Par_m^- \cup \Hyp_m & \text{ if } \mathbb{B}_m = \Hyp_m\\
    \Par_{m-1}^+ \cup \Ell_m & \text{ if } \mathbb{B}_m = \Ell_m,\ m> 0\\
    \Par_{m}^+ \cup \Ell_m & \text{ if } \mathbb{B}_m = \Ell_m,\ m < 0
    \end{cases}
\end{align*}
\begin{align*}
    \overline{\mathbb{B}}^R_m
    := \begin{cases}
    \Par_m^+ \cup \Hyp_m & \text{ if } \mathbb{B}_m = \Hyp_m\\
    \Par_{m}^- \cup \Ell_m & \text{ if } \mathbb{B}_m = \Ell_m,\ m> 0\\
    \Par_{m+1}^- \cup \Ell_m & \text{ if } \mathbb{B}_m = \Ell_m,\ m < 0
    \end{cases}
\end{align*}
 Let $\partial\overline{\mathbb{B}}_m^L = \overline{\mathbb{B}}_m^L \setminus \mathbb{B}_m$ and $\partial\overline{\mathbb{B}}_m^R = \overline{\mathbb{B}}_m^R \setminus  \mathbb{B}_m$.
 The \textit{edges} of a block $\mathbb{B}_m\times \mathbb{B}_n$ are defined as $$\partial\overline{\mathbb{B}}_m^L \times \mathbb{B}_n,\  \partial\overline{\mathbb{B}}_m^R \times \mathbb{B}_n,\ 
\mathbb{B}_m \times \partial\overline{\mathbb{B}}_n^R,\ \mathbb{B}_m \times \partial\overline{\mathbb{B}}_n^L,$$  
and its \textit{vertices} are defined as 
$$\partial \overline{\mathbb{B}}_m^R\times \partial\overline{\mathbb{B}}_n^R,\ 
\partial \overline{\mathbb{B}}_m^R\times \partial\overline{\mathbb{B}}_n^L,\ 
\partial \overline{\mathbb{B}}_m^L\times \partial\overline{\mathbb{B}}_n^R,\  \partial \overline{\mathbb{B}}_m^L\times \partial\overline{\mathbb{B}}_n^L.$$
In general, for some $s,s'\in \{+,-\}$ and $k,k'\in \mathbb{Z}$, we call the sets $\Par_k^s\times \Par^{s'}_{k'}$ as vertices.
We call two distinct blocks \textit{adjacent} if they share a common edge or a common vertex. Further, we call them \textit{edge-adjacent} if they share a common edge; and \textit{vertex-adjacent} if they share a unique common vertex. 
For a given vertex, we say an edge is \emph{incident to the vertex} if the vertex is contained in the closure of the edge in $(\univcover\setminus Z)\times (\univcover\setminus Z)$. Finally, if $V$ is a vertex of a block we say the block is adjacent to the vertex $V$.

\begin{lemma}\label{glue1}
  For $m_1, n_1, m_2, n_2 \in \mathbb{Z}$, let $\mathbb{B}_{m_1}\times \mathbb{B}_{n_1}$ and $\mathbb{B}_{m_2}\times \mathbb{B}_{n_2}$ be adjacent blocks. Let $\widetilde{C}\in \univcover$. Assume that the fiber $P^{-1}(\widetilde{C})$ intersects both the blocks. 
  \begin{enumerate}[(1)]
      \item If the blocks are edge-adjacent, then $P^{-1}(\widetilde{C})$ intersects their common edge.

      \item If the blocks are vertex-adjacent, then $P^{-1}(\widetilde{C})$ intersects either their common vertex or an edge incident to the common vertex.
  \end{enumerate} 
\end{lemma}

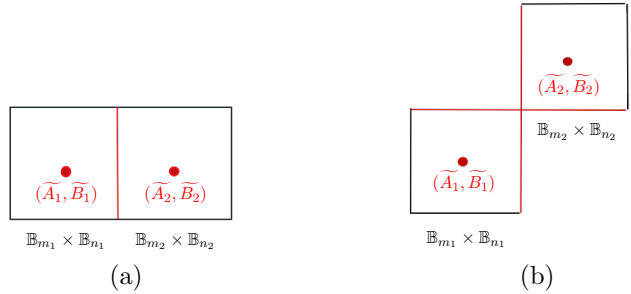
\begin{figure}[H]
    \centering
    \begin{overpic}[width=0.7\textwidth,trim = 0 500 0 200]{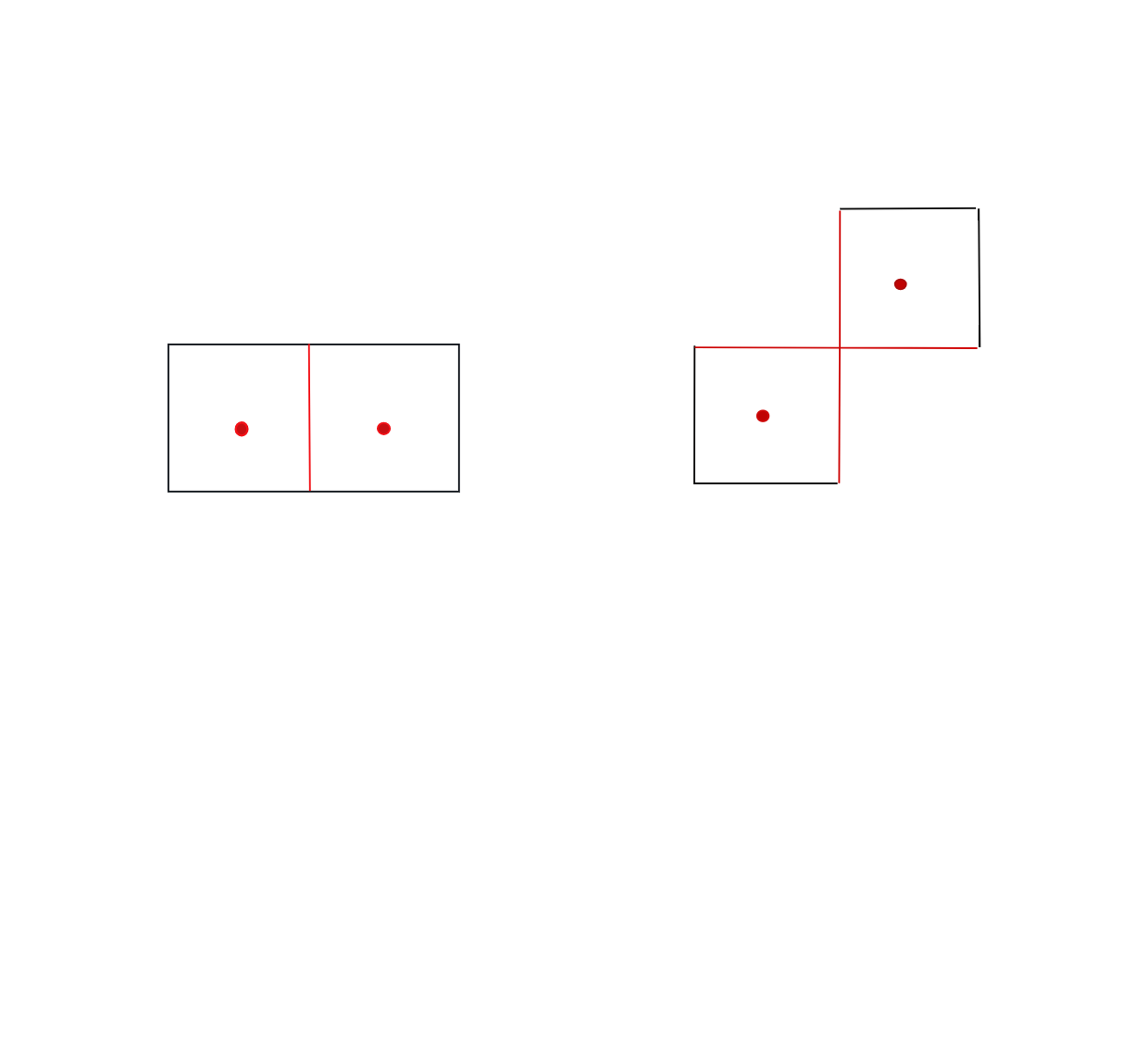}

    \put(26,0){\scalebox{0.9}{(a)}}

    \put(73,0){\scalebox{0.9}{(b)}}

    \put(17.5,9.7){\scalebox{0.6}{\textcolor{red}{$(\widetilde{A_1},\widetilde{B_1})$}}}
    \put(30,9.7){\scalebox{0.6}{\textcolor{red}{$(\widetilde{A_2},\widetilde{B_2})$}}}
    \put(16.5,4.5){\scalebox{0.6}{$\mathbb{B}_{m_1}\times \mathbb{B}_{n_1}$}}
    \put(29,4.5){\scalebox{0.6}{$\mathbb{B}_{m_2}\times \mathbb{B}_{n_2}$}}
    \put(63,11){\scalebox{0.6}{\textcolor{red}{$(\widetilde{A_1},\widetilde{B_1})$}}}
    \put(75,22){\scalebox{0.6}{\textcolor{red}{$(\widetilde{A_2},\widetilde{B_2})$}}}
    \put(62.3,4.7){\scalebox{0.6}{{$\mathbb{B}_{m_1}\times \mathbb{B}_{n_1}$}}}
    \put(75,17){\scalebox{0.6}{$\mathbb{B}_{m_2}\times \mathbb{B}_{n_2}$}}
        
    \end{overpic}
    \caption{Edge-adjacent and vertex-adjacent blocks}
    \label{blockadj}
\end{figure}
\begin{proof}
    For (1), without loss of generality, we assume $\partial \overline{\mathbb{B}}_{m_1}^R = \partial \overline{\mathbb{B}}_{m_2}^L$ and $\mathbb{B}_{n_1} = \mathbb{B}_{n_2}$, i.e., the two blocks share the common edge $ \partial \overline{\mathbb{B}}_{m_1}^R\times \mathbb{B}_{n_1} = \partial \overline{\mathbb{B}}_{m_2}^L\times \mathbb{B}_{n_2}$. 
    Let $(\widetilde{A_1},\widetilde{B_1})\in (\mathbb{B}_{m_1} \times\mathbb{B}_{n_1})\cap P^{-1}(\widetilde{C})$ and $(\widetilde{A_2},\widetilde{B_2})\in (\mathbb{B}_{m_2} \times\mathbb{B}_{n_2})\cap P^{-1}(\widetilde{C})$. See Figure \ref{blockadj}(a).
    As the image
    $P\big(\{\widetilde{C}\}\times \mathbb{B}_{-n_1}\big)$ is homeomorphic to $\mathbb{B}_{-n_1}$ in $\univcover$, $P\big(\{\widetilde{C}\}\times \mathbb{B}_{-n_1}\big)\setminus Z\subset \univcover\setminus Z$ is path-connected. Hence, $\widetilde{A_1} = \widetilde{C}\widetilde{B}_1^{-1}$ and $\widetilde{A_2} = \widetilde{C}\widetilde{B}_2^{-1}$ are connected by a path in $P\big(\{\widetilde{C}\}\times \mathbb{B}_{-n_1}\big)\setminus Z$. Since $\partial \overline{\mathbb{B}}_{m_1}^R$ separates $\univcover\setminus Z$ into two connected components with $\widetilde{A_1}\in \mathbb{B}_{m_1}$ lying in one and $\widetilde{A_2}\in \mathbb{B}_{m_2}$ in the other, the path intersects $\partial \overline{\mathbb{B}}_{m_1}^R$ at a point $\widetilde{P}$. Then the pair $(\widetilde{P}, \widetilde{P}^{-1}\widetilde{C})$ lies in the common edge $\partial \overline{\mathbb{B}}_{m_1}^R\times \mathbb{B}_{n_1}$, which shows that $(\partial \overline{\mathbb{B}}_{m_1}^R\times \mathbb{B}_{n_1})\cap P^{-1}(\widetilde{C})$ is non-empty.
    \smallskip
    
    For (2), without loss of generality, we assume $\partial \overline{\mathbb{B}}_{m_1}^R
    = \partial\overline{\mathbb{B}}_{m_2}^L$ and $\partial \overline{\mathbb{B}}_{n_1}^R = \partial \overline{\mathbb{B}}_{n_2}^L$, i.e., the two blocks share the common vertex $\partial \overline{\mathbb{B}}_{m_1}^R\times \partial\overline{\mathbb{B}}_{n_1}^R = \partial \overline{\mathbb{B}}_{m_2}^L\times \partial\overline{\mathbb{B}}_{n_2}^L$. 
    Let $(\widetilde{A_1},\widetilde{B_1})\in (\mathbb{B}_{m_1} \times\mathbb{B}_{n_1})\cap P^{-1}(\widetilde{C})$ and $(\widetilde{A_2},\widetilde{B_2})\in (\mathbb{B}_{m_2} \times\mathbb{B}_{n_2})\cap P^{-1}(\widetilde{C})$. See Figure \ref{blockadj}(b).
    Since the closures of $\mathbb{B}_{n_1}$ and $\mathbb{B}_{n_2}$ intersect at $\partial \overline{\mathbb{B}}_{n_1}^R$, the space 
    $\mathbb{B}_{n_1}\cup \partial\overline{\mathbb{B}}_{n_1}^R\cup \mathbb{B}_{n_2}$ is path-connected. Similarly, $\mathbb{B}_{-n_1}\cup \partial\overline{\mathbb{B}}_{-n_1}^L\cup \mathbb{B}_{-n_2}$ is path-connected, and the image 
    $P\big(\{\widetilde{C}\}\times \big(\mathbb{B}_{-n_1}\cup \partial\overline{\mathbb{B}}_{-n_1}^L\cup \mathbb{B}_{-n_2}\big)\big) \setminus Z$ is path-connected in $\univcover\setminus Z$.
    Hence, $\widetilde{A_1} = \widetilde{C}\widetilde{B}_1^{-1}$ and $\widetilde{A_2} = \widetilde{C}\widetilde{B}_2^{-1}$are connected by a path in $P\Big(\{\widetilde{C}\}\times \big(\mathbb{B}_{-n_1}\cup \partial\overline{\mathbb{B}}_{-n_1}^L\cup \mathbb{B}_{-n_2}\big)\Big)\setminus Z$.
    Since $\partial \overline{\mathbb{B}}_{m_1}^R$ separates $\univcover\setminus Z$ into two connected components with $\widetilde{A_1}\in \mathbb{B}_{m_1}$ lying in one and $\widetilde{A_2}\in \mathbb{B}_{m_2}$ in the other, the path intersects $\partial \overline{\mathbb{B}}_{m_1}^R$ at a point $\widetilde{P}$. Then the pair $(\widetilde{P}, \widetilde{P}^{-1}\widetilde{C})$ lies in the space $\partial \overline{\mathbb{B}}_{m_1}^R\times \big(\mathbb{B}_{n_1}\cup \partial\overline{\mathbb{B}}_{n_1}^R\cup \mathbb{B}_{n_2}\big)$, which is a union of the common vertex $\partial \overline{\mathbb{B}}_{m_1}^R\times \partial\overline{\mathbb{B}}_{n_1}^R$ and two incident edges $\partial \overline{\mathbb{B}}_{m_1}^R\times \mathbb{B}_{n_1}$ and $\partial \overline{\mathbb{B}}_{m_2}^L\times \mathbb{B}_{n_2}$.
\end{proof}

When the fiber $P^{-1}(\widetilde{C})$ intersects an edge $E$ at a point $(\widetilde{A},\widetilde{B})$, the following Lemma \ref{glue2} provides paths in the fiber which perturbs $(\widetilde{A},\widetilde{B})$ into the blocks having $E$ as their common edge. 

\begin{lemma}\label{glue2}
 Let $\mathbb{B}_{m_1}\times \mathbb{B}_{n_1}$ and $\mathbb{B}_{m_2}\times \mathbb{B}_{n_2}$ be edge-adjacent blocks. For $\widetilde{C}\in \univcover$, let $(\widetilde{A},\widetilde{B})\in P^{-1}(\widetilde{C})$ contained in their common edge. For each $i\in\{1,2\}$, there is a path $\{\delta_{i,t}\}\interval$ in $P^{-1}(\widetilde{C})$ starting from $(\widetilde{A},\widetilde{B})$  such that for all $t\in (0,1]$, $\delta_{i,t}$ lies in $\mathbb{B}_{m_i}\times \mathbb{B}_{n_i}$.
\end{lemma}
\begin{figure}[H]
    \centering
    \begin{overpic}[width=0.7\textwidth,trim = 0 500 0 350]{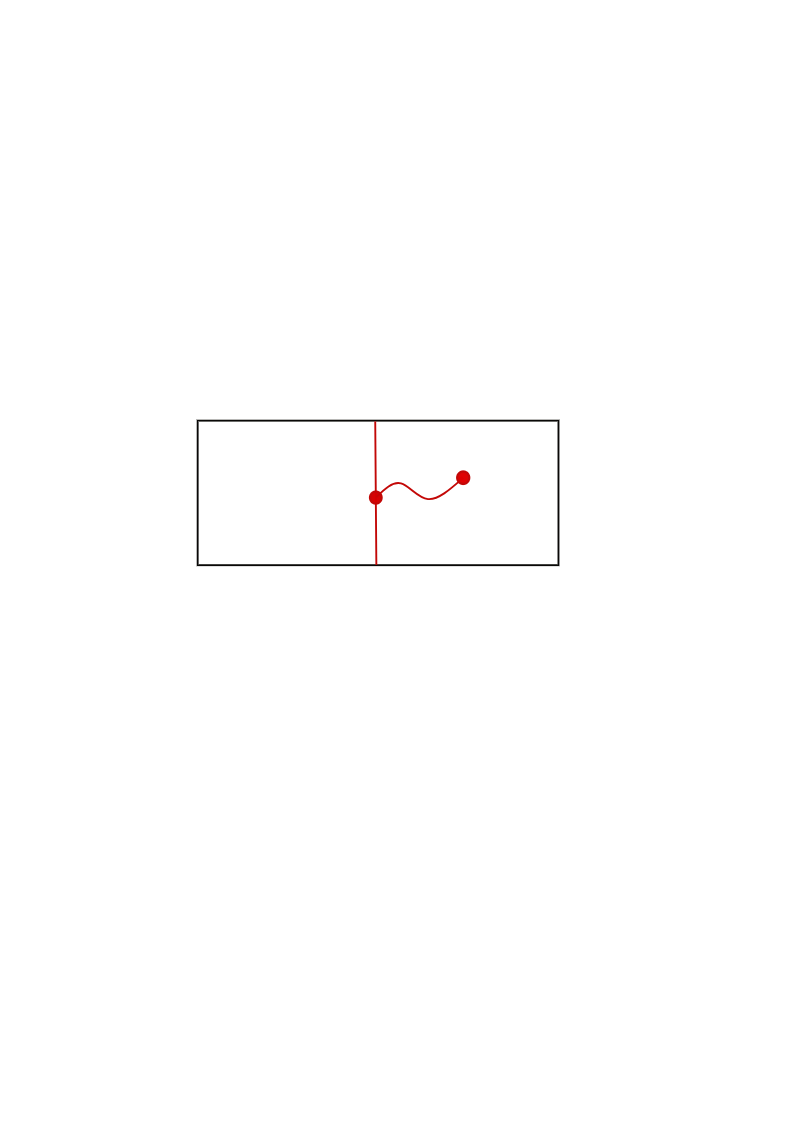}
    \put(40,15){\scalebox{0.6}{\textcolor{red}{$(\widetilde{A},\widetilde{B})$}}}
    \put(58,20){\scalebox{0.6}{\textcolor{red}{$(\widetilde{A'},\widetilde{B'})$}}}
    \put(32.5,4){\scalebox{0.6}{$\mathbb{B}_{m_1}\times \mathbb{B}_{n_1}$}}
    \put(54,4){\scalebox{0.6}{$\mathbb{B}_{m_2}\times \mathbb{B}_{n_2}$}}

    \end{overpic}
    \caption{Pushing into adjacent blocks}
    \label{pathadj}
\end{figure}
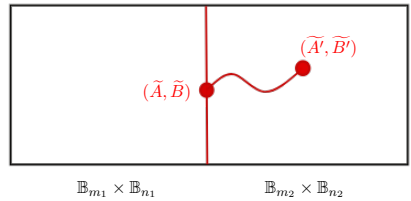
\begin{proof}
     Let $\mathbb{B}_{m_1}\times \mathbb{B}_{n_1}$ and $\mathbb{B}_{m_2}\times \mathbb{B}_{n_2}$ be edge-adjacent. Without loss of generality, we assume $\mathbb{B}_{n_1} = \mathbb{B}_{n_2}$ and $\partial \overline{\mathbb{B}}_{m_1}^R = \partial \overline{\mathbb{B}}_{m_2}^L$, i.e., the two blocks share the common edge $ \partial \overline{\mathbb{B}}_{m_1}^R\times \mathbb{B}_{n_1} = \partial \overline{\mathbb{B}}_{m_2}^L\times \mathbb{B}_{n_2}$. Let $(\widetilde{A},\widetilde{B})\in (\partial \overline{\mathbb{B}}_{m_1}^R\times \mathbb{B}_{n_1})\cap P^{-1}(\widetilde{C})$. See Figure \ref{pathadj}. Let $i\in \{1,2\}$. By the continuity of the map $P:\univcover \times\{\widetilde{C}\} \rightarrow\univcover$, we may choose a path $\{\widetilde{P_{i,t}}\}\interval$  
    starting from $\widetilde{P_{i,0}} = \widetilde{A}^{-1}$ such that $\widetilde{P_{i,t}}\in \mathbb{B}_{-m_i}$ for all $t\in (0,1]$, and 
    $\{\widetilde{P_{i,t}}\widetilde{C}\}\interval \subset \mathbb{B}_{n_i}$. Thus, letting $\delta_{i,t}=(\widetilde{P}_{i,t}^{-1},\widetilde{P}_{i,t}\widetilde{C})$ for all $t\in [0,1]$ gives us the desired path $\{\delta_{i,t}\}\interval$ in $P^{-1}(\widetilde{C})$ connecting $(\widetilde{A},\widetilde{B})$ to a point $(\widetilde{A}',\widetilde{B}')$ in $\mathbb{B}_{m_i}\times \mathbb{B}_{n_i}$.
\end{proof}

For $m, n\in \mathbb{Z}$ and $s_1,s_2 \in \{+, -\}$,
let $V =\Par_m^{s_1}\times \Par_n^{s_2}$. It is a common vertex of four adjacent blocks, which we denote by $B^1_L,B^2_L,B^1_R,B^2_R$. Moreover, four of their edges are incident to $V$.  Let $S_V$ denote the union of the four adjacent blocks, four incident edges and $V$. See Figure \ref{4box}.
\begin{figure}[H]
    \centering
    \begin{overpic}[width=0.3\textwidth]{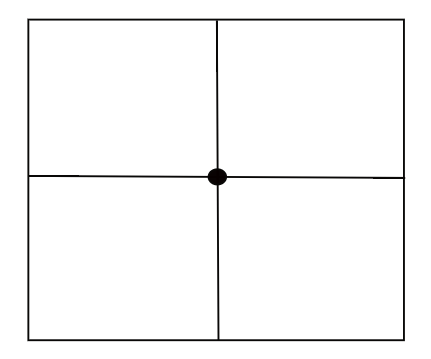}
 
    \put(24,58){\scalebox{0.8}{$B^1_L$}}
    \put(24,21){\scalebox{0.8}{$B^2_L$}}
    \put(67,58){\scalebox{0.8}{$B^1_R$}}
    \put(67,21){\scalebox{0.8}{$B^2_R$}}
    \put(53,44){\scalebox{0.8}{$V$}}
    \end{overpic}
    \caption{$S_V$}
    \label{4box}
\end{figure}

When the fiber $P^{-1}(\widetilde{C})$ intersects a vertex $V$ at a point $(\widetilde{P_1}, \widetilde{P_2})$, the following Lemma \ref{glue3} provides a path in the fiber which perturbs $(\widetilde{P_1}, \widetilde{P_2})$ into a block
having $V$ as its vertex. 
\begin{lemma}\label{glue3}
For $m, n\in \mathbb{Z}$ and $s_1,s_2 \in \{+, -\}$,
let $V =\Par_m^{s_1}\times \Par_n^{s_2}$. Let $\widetilde{C}\in \univcover$, and let $(\widetilde{P_1},\widetilde{P_2}) \in P^{-1}(\widetilde{C})\cap V$. For each $i\in \{1,2\}$, there is a path $\{\delta_{i,t}\}\interval$ in $P^{-1}(\widetilde{C})\cap S_V$ starting at $(\widetilde{P_1}, \widetilde{P_2})$ such that the endpoint $\delta_{i,1}$ lies in one of the two blocks $B^i_L \text{ or } B^i_R$. Likewise, for $\alpha\in \{R,L\}$, there is a path $\{\delta_{\alpha,t}\}\interval$ in $P^{-1}(\widetilde{C})\cap S_V$ starting at $(\widetilde{P_1}, \widetilde{P_2})$ such that the endpoint $\delta_{\alpha,1}$ lies in one of the two blocks $B^1_\alpha \text{ or } B^2_\alpha$.
\end{lemma}
\begin{proof}
     We prove the lemma for the case $s_1=s_2=+$ and $m = n = 0$; and for the other cases of $s_1, s_2\in \{+,-\}$ and $m, n\in \mathbb{Z}$, the proof follows verbatim. 
    In this case, the blocks adjacent to $V$ are
    
    $$B^1_L = \Hyp_0\times \Ell_1,\ B^1_R = \Ell_1\times \Ell_1,\  B^2_L=\Hyp_0\times \Hyp_0,\ B^2_R = \Ell_1\times \Hyp_0.$$
    By continuity of the product map, we may choose a path $\{\widetilde{P^t}\}\interval$ starting at $\widetilde{P^0} = \widetilde{P'_2}^{-1}$ such that $\{\widetilde{P^t}\}_{t\in(0,1]} \subset \Hyp_0$ and $\{\widetilde{C}\widetilde{P^t}\}\interval \subset \Hyp_0\cup \Par_0^+ \cup \Ell_1$. 
    Hence the path $\big\{(\widetilde{C}\widetilde{P^t},\widetilde{P^t}^{-1})\big\}_{t\in(0,1]}\subset \big((\Hyp_0 \cup \Par_0^+ \cup \Ell_1) \times \Hyp_0\big)\cap P^{-1}(\widetilde{C})$. 
    If $(\widetilde{C}\widetilde{P^1},\widetilde{P^1}^{-1})\in (\Ell_1\cup \Hyp_0)\times \Hyp_0$, then we let $\{\delta'_{2,t}\}\interval = \big\{(\widetilde{C}\widetilde{P^t},\widetilde{P^t}^{-1})\big\}\interval$. 
    If $(\widetilde{C}\widetilde{P^1},\widetilde{P^1}^{-1})\in \Par_0^+\times \Hyp_0$, then we can concatenate $\{(\widetilde{C}\widetilde{P^t},\widetilde{P^t}^{-1})\}_{t\in[0,1]}$ with a path $\{\delta''_t\}\interval \subset P^{-1}(\widetilde{C})$ obtained from Lemma \ref{glue2} to construct the desired path $\{\delta'_{2,t}\}\interval$. Similarly, we may find a path $\{\delta'_{1,t}\}\interval$ in $P^{-1}(\widetilde{C})$ starting at $(\widetilde{P_1},\widetilde{P_2})$ and ending at $B^1_L$ or $B^1_R$. 
    \smallskip
    
    The construction of $\{\delta_{R,t}\}\interval$ and $\{\delta_{L,t}\}\interval$ follows verbatim.
    \end{proof}
   
Let $I_1$ and $I_2$ denote intervals in $\univcover$.
Notice that $I_1\times I_2$ is a union of blocks glued along their edges and vertices contained in $I_1\times I_2$. Let $\widetilde{C}\in \univcover\setminus Z$. Let $\mathcal{B}$ denote the collection of blocks in $I_1\times I_2$ that intersect the fiber $P^{-1}(\widetilde{C})$. 
If $P^{-1}(\widetilde{C})\cap(I_1\times I_2)$ is non-empty, then by Lemmas \ref{glue2} and \ref{glue3},  $\mathcal{B}$ is non-empty. Let $G$ be a graph whose vertices correspond to elements of $\mathcal{B}$, where two vertices are connected by an edge if the corresponding blocks are adjacent. 

\begin{lemma}\label{gconn}
    The graph $G$ is connected.
\end{lemma}

\begin{proof}

For the sake of contradiction, assume $G$ has at least two connected components. We choose subintervals $I_1' \subset I_1$ and $I_2' \subset I_2$ such that $I_1'\times I_2'$ satisfies the following conditions; see Figure \ref{r1r2}.
\begin{enumerate}[(a)]
        \item 
    $I_1'\times I_2'$ intersects only those blocks in $\mathcal{B}$ whose corresponding vertices in $G$ belong to exactly two connected components $C_1$ and $C_2$ of $G$. Denote by $K_1$ and $K_2$ the sets of vertices in $C_1$ and $C_2$, respectively, whose corresponding blocks are contained in $I_1'\times I_2'$.
    \item
    For $i\in \{1,2\}$, the subgraph of $G$ induced by $K_i$ is connected.
\end{enumerate}
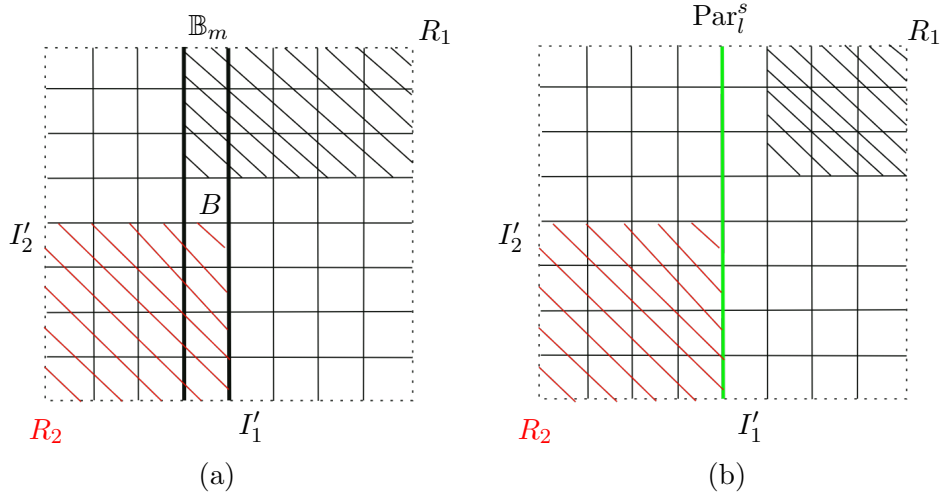
\begin{figure}[H]
    \centering
    \begin{overpic}[width=0.8\textwidth, trim = 0 225 0 175]{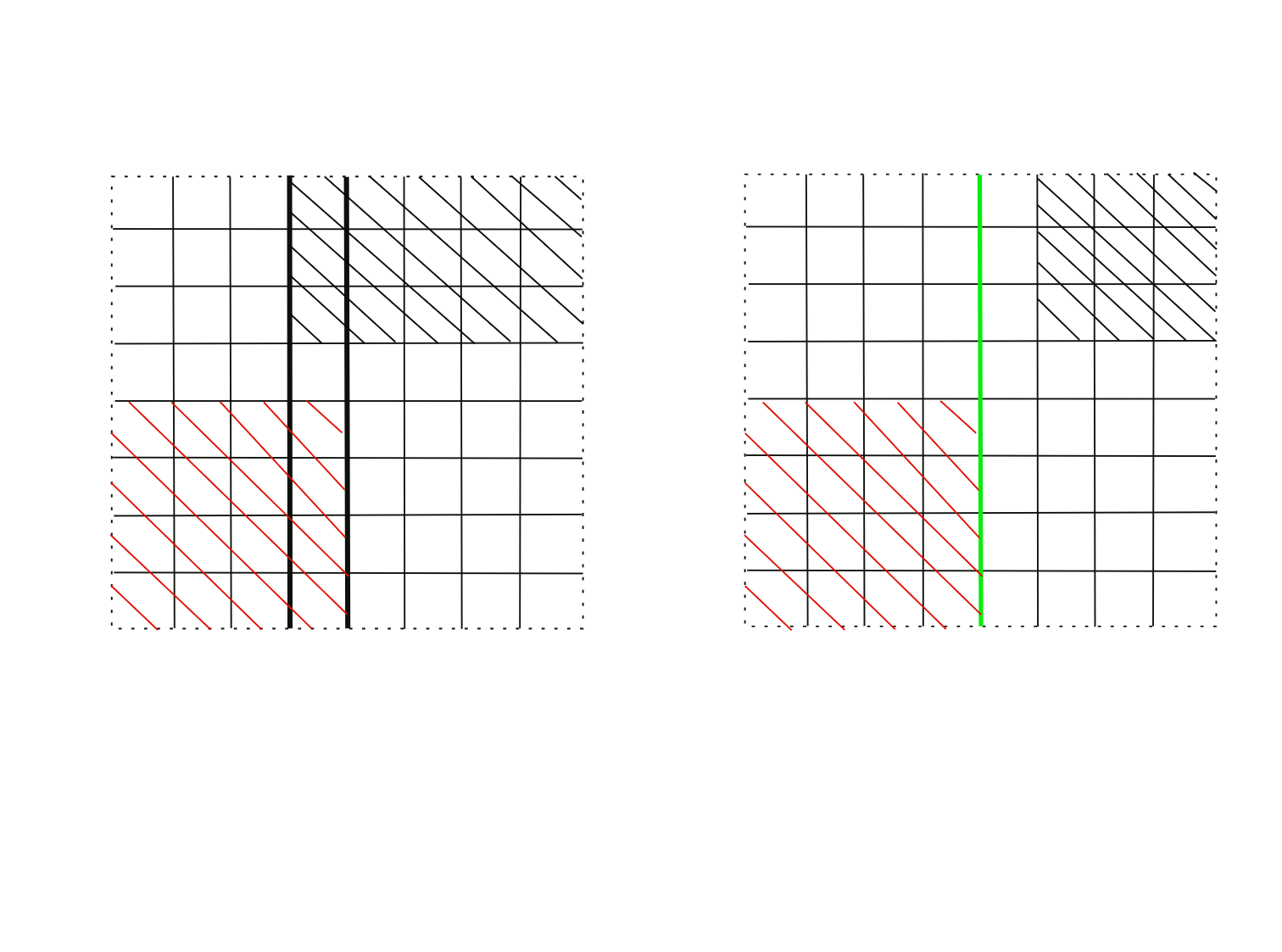} 
    \put(5,26){\scalebox{1}{$I_2'$}}
    \put(54,26){\scalebox{1}{$I_2'$}}
    \put(28,7){\scalebox{1}{$I_1'$}}
    \put(24,2){\scalebox{1}{(a)}}
    \put(23,47){\scalebox{1}{$\mathbb{B}_m$}}
    \put(24,29){\scalebox{1}{$B$}}
    \put(7,6.5){\scalebox{1}{\textcolor{red}{$R_2$}}}
    \put(46,46.5){\scalebox{1}{$R_1$}}
    \put(78,7){\scalebox{1}{$I_1'$}}
    \put(75,2){\scalebox{1}{(b)}}
    \put(73.5,48){\scalebox{1}{$\Par^s_l$}}
    \put(56,6.5){\scalebox{1}{\textcolor{red}{$R_2$}}}
    \put(95,46.5){\scalebox{1}{$R_1$}}
    \end{overpic}
    \caption{Regions $R_1$ and $R_2$}
    \label{r1r2}
\end{figure}

 Denote by $R_1$ and $R_2$ the union of blocks in $\mathcal{B}$ corresponding to vertices in $K_1$ and $K_2$, respectively.\smallskip
 
 We first show for every $\mathbb{B}_m\subset I_1',m\in \mathbb{Z}$, the set $\mathbb{B}_m\times I_2'$ intersects at most one of $R_1$ or $R_2$. Suppose not. Then there exists $\mathbb{B}_m\subset I'_1$ for some $m\in \mathbb{Z}$ such that $\mathbb{B}_m\times I'_2$ intersects both $R_1$ and $R_2$. See Figure \ref{r1r2}(a). 
Then there
 exists a block $B=\mathbb{B}_m\times \mathbb{B}_n \subset \mathbb{B}_m\times I_2'$ for some $n \in \mathbb{Z}$, which is disjoint from $P^{-1}(\widetilde{C})$, and such that the interior of $(\mathbb{B}_m\times I'_2)\setminus B$ has two connected components, say $\mathbb{B}_m\times J'$ and $\mathbb{B}_m\times J''$, which intersect $R_1$ and $R_2$, respectively. Indeed, if no such block existed, then the sequence of blocks in $\mathbb{B}_m\times I'_2$ would induce a subpath in $G$ connecting the subgraphs of $G$ induced by $K_1$ and $K_2$, contradicting that the two subgraphs are contained in two distinct connected components of $G$. 
 Since $\widetilde{C}\not\in P(\mathbb{B}_m\times \mathbb{B}_n)$, the image $P(\mathbb{B}_{-m}\times \{\widetilde{C}\})$ in $\univcover$ does not intersect $\mathbb{B}_n$. 
 Moreover, since $P^{-1}(\widetilde{C})$ intersects every block in $R_1$ and $R_2$, it intersects both $\mathbb{B}_m\times J'$ and $\mathbb{B}_m\times J''$; and hence $P(\mathbb{B}_{-m}\times \{\widetilde{C}\})$ intersects both intervals $J'$ and $J''$. Since $J'$ and $J''$ lie in distinct connected components of $(\univcover\setminus Z)\setminus \mathbb{B}_n$, this implies that $P(\mathbb{B}_{-m}\times \{\widetilde{C}\})\setminus Z \subset \univcover\setminus Z$ is not path-connected. However, since the space
$P\big(\mathbb{B}_{-m}\times\{\widetilde{C}\}\big)$ is homeomorphic to $\mathbb{B}_{-m}$, $P(\mathbb{B}_{-m}\times \{\widetilde{C}\})\setminus Z$ is connected, which is a contradiction. \smallskip

Next, we show that there exists $\Par^s_l\subset I_1'$ for some $l\in \mathbb{Z}$ and $s\in \{+,-\}$ such that 
\begin{enumerate}[(1)]
    \item 
$\Par^s_l\times I_2'$ does not intersect $P^{-1}(\widetilde{C})$, and
    \item 
$R_1$ is contained in one of the connected components of $(I_1'\times I_2')\setminus (\Par^s_l \times I_2')$, and $R_2$ is contained in the other. (See Figure \ref{r1r2}(b).)
\end{enumerate}
Denote by $\overline{\pi(R_2)}$ the closure of the projection of $R_2$ to $I_1'\subset \univcover\setminus Z$. Then its boundary $\partial \overline{\pi(R_2)}$ in $ \univcover\setminus Z$ consists of two sets $\Par^{s'}_{k'}$ and $\Par^{s''}_{k''}$ for some $s',s''\in \{+,-\}$ and $k',k''\in \mathbb{Z}$. We claim that one of these two sets satisfies Conditions (1) and (2). As shown in the previous paragraph, for every $\mathbb{B}_m\subset I_1',m\in \mathbb{Z}$, the set $\mathbb{B}_m\times I_2'$ intersects at most one of $R_1$ or $R_2$. Hence for each $(r,k)\in \big\{(s',k'),(s'',k'')\big\}$, $R_1$ is contained in one of the connected components of the complement $(I_1'\times I_2')\setminus (\Par^r_k\times I_2')$. 
 Moreover, there exists $(s,l) \in \big\{(s',k'),(s'',k'')\big\}$ such that $R_1$ and $R_2$ lie in distinct components of the complement of $\Par^s_l\times I'_2$ in $I_1'\times I_2'$. Next, assume $q\in P^{-1}(\widetilde{C})\cap (\Par^{s}_{l}\times I_2')$. As an element in $\Par^s_l\times I'_2$, $q$ lies in a common edge or a common vertex of blocks in $I_1'\times I_2'$. By Lemmas \ref{glue2} and \ref{glue3}, there exist adjacent blocks $B_1$ and $B_2$ in $\mathcal{B}$, lying in distinct components of $(I'_1\times I'_2)\setminus (\Par^s_l\times I'_2)$, 
whose common edge or vertex contains $q$. 
By Condition (a) of $I'_1\times I'_2$, one of $R_1$ and $R_2$ contains both $B_1$ and $B_2$. This contradicts that $R_1$ and $R_2$ lie in distinct components of $(I'_1\times I'_2)\setminus (\Par^s_l\times I'_2)$. Therefore, $\Par^s_l\times I'_2$ does not intersect $P^{-1}(\widetilde{C}).$\smallskip
 
 Finally, since $P^{-1}(\widetilde{C})$  intersects both components of $(I'_1\times I'_2)\setminus (\Par^s_l\times I'_2)$, $P(\{\widetilde{C}\}\times (I'_2)^{-1})$ intersects both components of $(\univcover\setminus Z)\setminus \Par_l^s$. However, $P(\{\widetilde{C}\}\times (I_2')^{-1})$ is a connected set in $\univcover\setminus Z$ that does not intersect $\Par_l^s$, which is a contradiction.
\end{proof}

\subsubsection{Proofs of Lemmas \ref{Prod_confib} and \ref{plp_HP'}} 
To prove Lemma \ref{Prod_confib}, 
we need the following Lemmas \ref{edgecon1}, \ref{edgecon2} and \ref{edgecon3}.
 The proofs of these lemmas are given in the Appendix.

 \begin{lemma}\label{edgecon1}
Let $P: \univcover\times \univcover\to \univcover$ be a product map, and let $m, n, l\in \mathbb{Z}, s\in\{+,-\}$.
Consider the following restrictions of $P$:
\begin{align*}
P_1:&\ (\Hyp_m\times \Par_0^s) \cap P^{-1}(\Hyp_l) \to \Hyp_l, \text{ and}
\\
P_2: &\ (\Ell_m \times \Par_0^s) \cap P^{-1}(\Hyp_l) \to \Hyp_l.
\end{align*}
For each $i\in \{1,2\}$, every non-empty fiber of $P_i$ is path-connected.   
\end{lemma}

Let $m\in \mathbb{Z}, s\in \{+,-\}$, and $\mathbb{B}_m\in \{\Hyp_m,\ \Ell_m\}$. The sets $E = \mathbb{B}_m\times \Par_0^s$ will be called as \emph{edges}. Let $m,n\in \mathbb{Z}, s\in \{+,-\}$ and $E_1= \mathbb{B}_m\times \Par_0^s$ and $E_2=\mathbb{B}_n\times \Par^s_0$. 
Let $\overline{\mathbb{B}_m}$ and $\overline{\mathbb{B}_n}$ be the closures of $\mathbb{B}_m$ and $\mathbb{B}_n$, respectively, in $\univcover\setminus Z$.
If $\big(\overline{\mathbb{B}_m}\times \Par_0^s\big) \cap \big(\overline{\mathbb{B}_n}\times \Par_0^s\big) = \Par_k^{s'}\times \Par_0^s$ for some $k\in \mathbb{Z}$ and $s'\in\{+,-\}$, then we say $E_1,E_2$ are \emph{adjacent at} $V=\Par_k^{s'}\times \Par_0^s$.

 \begin{lemma}\label{edgecon2}
    Let $\widetilde{C}\in \Hyp_l$ for $l\in \mathbb{Z}$. For $m,n\in \mathbb{Z}$, 
    let $\mathbb{B}_m\in \{\Hyp_m,\ \Ell_m\}$ and $\mathbb{B}_n\in \{\Hyp_n,\ \Ell_n\}$.
    For $s\in \{+,-\}$, let $E_1=\mathbb{B}_m\times \Par_0^s$ and $E_2=\mathbb{B}_n\times \Par_0^s$. Assume that $E_1$ and $E_2$ are adjacent at $V=\Par^{s'}_k\times \Par_0^s$ for some $k\in \mathbb{Z}, s'\in \{+,-\}$. Then
    \begin{enumerate}[(a)]
        \item For every $(\widetilde{A},\widetilde{B})\in V\cap P^{-1}(\widetilde{C})$ and $j\in \{1,2\}$, there is a path $\{\delta^t_j\}\interval \subset P^{-1}(\widetilde{C})$ starting at $(\widetilde{A},\widetilde{B})$ such that for all $t\in (0,1]$, $\delta^t_j\in E_j$.
        \item If $P^{-1}(\widetilde{C})$ intersects both $E_1$ and $E_2$, then $P^{-1}(\widetilde{C})$ intersects $V$.
    \end{enumerate}
 \end{lemma}
For an interval $I\subset \univcover$ and $s\in \{+,-\}$, notice that the set $I\times \Par_0^s$ is a union of edges adjacent at vertices contained in $I\times \Par_0^s$. For $\widetilde{C}\in \Hyp_l$, we denote by $\mathbf{E}$ the collection of edges $\mathbb{B}_m\times \Par_0^s\subset I\times \Par_0^s$, where $m\in \mathbb{Z}$ and $\mathbb{B}_m\in \{\Hyp_m,\ \Ell_m\}$, which intersect $P^{-1}(\widetilde{C})$.

\begin{lemma}\label{edgecon3}
    If $E,E' \in \mathbf{E}$, then there is a sequence of egdes $E_0=E, E_1,\dots ,E_k=E'$ for some $k\in \mathbb{N}$, such that for all $i\in \{0,\dots, k-1\}$, $E_i\in \mathbf{E}$ and $E_i,E_{i+1}$ are adjacent at a vertex.  
\end{lemma}
\begin{proof}[Proof of Lemma \ref{Prod_confib}]
     For (a), as a restriction of the product map to the open set $D$, $P: D\to P(D)$ is a submersion. By Lemma \ref{lem_plp},
     it suffices to show that for all $\widetilde{C}\in P(I_1\times I_2)\setminus Z$, the fiber $P^{-1}(\widetilde{C})\cap (I_1\times I_2)$ is connected.
     To this end, for any $p, q\in P^{-1}(\widetilde{C})\cap (I_1\times I_2)$, we show that $p$ and $q$ are connected by a path in $P^{-1}(\widetilde{C})\cap (I_1\times I_2)$. By Lemma \ref{glue2} and \ref{glue3}, it suffices to show this for $p$ and $q$ lying in blocks in $I_1\times I_2$. 
     Let $\mathcal{B}$ denote the collection of blocks in $I_1\times I_2$ that intersect the fiber $P^{-1}(\widetilde{C})$, and let $B_1, B_2\in \mathcal{B}$ be the blocks containing $p$ and $q$, respectively.\smallskip
     
     If $B_1 = B_2$, then by Lemmas \ref{Hypfiber}, \ref{ellfiber} and \ref{parfiblemm}, 
     the set $B_1\cap P^{-1}(\widetilde{C})$ is path-connected, and the claim directly follows.
     \smallskip
     
     If $B_1$ and $B_2$ are edge-adjacent, then by Lemma \ref{glue1}, $P^{-1}(\widetilde{C})$ intersects the common edge at a point, say $r$. Using Lemma \ref{glue2}, together with the path-connectedness of the fiber within each block $B_1$ and $B_2$, we may connect $p$ to $r$ and $r$ to $q$ by paths in $P^{-1}(\widetilde{C})\cap (I_1\times I_2)$. 
     \smallskip

     If $B_1$ and $B_2$ are vertex-adjacent, 
     then let $V$ be their common vertex. If there is a block in $\mathcal{B}$ which is edge-adjacent to both $B_1$ and $B_2$, then as in the previous paragraph, we can connect $p$ to a point $r$ in this block by a path in $P^{-1}(\widetilde{C})\cap (I_1\times I_2)$. Similarly, there is a path in the fiber connecting $r$ to $q$. 
    If no block in $\mathcal{B}$ is edge-adjacent to both $B_1$ and $B_2$, then by Lemma \ref{glue2}, $P^{-1}(\widetilde{C})$ intersects none of the four edges incident to $V$. Moreover, by Lemma \ref{glue1}, the fiber intersects $V$. Let    
    $r\in P^{-1}(\widetilde{C})\cap V$. Using Lemma \ref{glue3}, together with the path-connectedness of the fiber within each block $B_1$ and $B_2$, we may connect $p$ to $r$ and $r$ to $q$ by paths in $P^{-1}(\widetilde{C})\cap (I_1\times I_2)$. 
    
    Finally, if $B_1$ and $B_2$ are not adjacent, then by Lemma \ref{gconn}, there is a sequence of blocks in $\mathcal{B}$ from $B_1$ to $B_2$ such that consecutive blocks are adjacent. Since we can connect any two elements within the fiber lying in adjacent blocks by a path in $P^{-1}(\widetilde{C})\cap(I_1\times I_2)$, we can connect $p$ to $q$ by a path in $P^{-1}(\widetilde{C})\cap(I_1\times I_2)$. Therefore, $P^{-1}(\widetilde{C})\cap (I_1\times I_2)$ is path-connected. 
    \medskip

    For (b), by Lemma \ref{lem_evPsubm}, the map $P:D\rightarrow P(D)$ is a submersion; and by Lemma \ref{lem_plp}, it suffices to show that for all $\widetilde{C}\in \Hyp_l$, the fiber $P^{-1}(\widetilde{C})\cap \big(I\times \Par_0^s\big)$ is path-connected. 
 Let $p,q\in P^{-1}(\widetilde{C})\cap \big(I\times \Par_0^s\big)$. We show that $p,q$ are connected by a path in $P^{-1}(\widetilde{C})\cap \big(I\times \Par_0^s\big)$. By Lemma \ref{edgecon2}, 
 it suffices to show this for $p,q$ lying in edges in $I\times \Par_0^s$. Let $E_1,E_2\in \mathbf{E}$ such that $p\in E_1$ and $q\in E_2$. 
 If $E_1=E_2$, then by Lemma \ref{edgecon1}, $p$ and $q$ are connected by a path in $E_1\cap P^{-1}(\widetilde{C})$. 
 If $E_1,E_2$ are adjacent at a vertex $V$, then by Lemma \ref{edgecon2}, 
 $V$ intersects $P^{-1}(\widetilde{C})$. Let $r\in P^{-1}(\widetilde{C})\cap V$. Then, using Lemma \ref{edgecon2}, 
 there is a path in $\big(I\times \Par_0^s\big)\cap P^{-1}(\widetilde{C})$ connecting $r$ to a point $p'\in E_1\cap P^{-1}(\widetilde{C})$. Similarly, we can connect $r$ to a point $q'\in E_2\cap P^{-1}(\widetilde{C})$ by a path within $\big(I\times \Par_0^s\big)\cap P^{-1}(\widetilde{C})$. Finally, since each of $E_1\cap P^{-1}(\widetilde{C})$ and $E_2\cap P^{-1}(\widetilde{C})$ is path-connected, we may connect $p$ to $q$ by a path in $\big(I\times \Par_0^s\big)\cap P^{-1}(\widetilde{C})$ which passes through $r$. 
 If $E_1,E_2$ are not adjacent at a vertex, then by Lemma \ref{edgecon3}, there is a sequence of edges in $\mathbf{E}$ such that consecutive edges are adjacent at a vertex. Since any two elements in the fiber lying in edges adjacent at a vertex can be connected by a path in $\big(I\times \Par_0^s\big)\cap P^{-1}(\widetilde{C})$, we may connect $p$ to $q$ by a path in $\big(I\times \Par_0^s\big)\cap P^{-1}(\widetilde{C})$. This shows $P^{-1}(\widetilde{C})\cap \big(I\times \Par_0^s\big)$ is path-connected. 
\end{proof}

We now prove Lemma \ref{plp_HP'}. We first fix some notation we will use in the proof. 
Recall from Section \ref{sec_pfof4.3} the notation $\alpha_l, l\in \{1,\cdots g+m-1\}$ and $\beta_k, k\in \{1,\cdots g+m-2\}$ for the decomposition curves.
We consider the case $p=2m$ and $p=2m+1$ where $m\in \mathbb{N}\cup \{0\}$ separately. 
\smallskip

When $p=2m$, to each of these decomposition curves, we associate an interval in $\univcover$ as follows.  For $j\in \{1,\dots, g-1\}$, we let $I_{\alpha_j}:= \mathcal{I}\setminus \{\mathrm I\}$ where $\mathcal{I}$ is the image of the lifted commutator defined in Theorem \ref{thm_liftcommu}. For $i\in \{1,\dots, m\}$, we let
$$I_{\alpha_{g+i-1}} := \ev\big(\Par^{sgn(s_{2i-1})}\times \Par^{sgn(s_{2i})}\big)
\setminus \big(\Par_0\cup \{\mathrm I\}\big),$$ where $\ev$ is the lifted product map. Then for every $i\in \{1,\dots, m\}$, by Lemma \ref{parplp}, $I_{\alpha_{g+i-1}}$ is an interval. Let $I_{\beta_1} := P(I_{\alpha_1}\times I_{\alpha_2})\setminus Z$, and for each $k\in \{2,\dots, g+m-3\}$, let $I_{\beta_k}: = P(I_{\beta_{k-1}}\times I_{\alpha_{k+1}})\setminus Z$. Then for each $k\in \{1,\dots, g+m-3\}$, by Lemma \ref{Prod_open}, $I_{\beta_k}$ is an interval. 
Moreover, for every $\rho \in HP^{s,0}_n(\Sigma\setminus T)'$ and every decomposition curve $\delta$, we choose the lift $\widetilde{\rho}_\delta$ of $\rho(\delta)$ in $I_{\delta}$ as follows. For $j\in \{1,\dots,g-1\}$, we let $\widetilde{\rho}_{\alpha_j} = [\widetilde{\rho(a_j)},\widetilde{\rho(b_j)}]$; and for $i \in \{1,\dots,m\}$, we let $\widetilde{\rho}_{\alpha_{g+i-1}} =\ev\big(\rho(c_{2i-1}), \rho(c_{2i})\big)$. 
Since the restrictions $\rho|_{\pi_1(T_j)}, j\in \{1,\dots, g-1\}$ and $\rho|_{\pi_1(P_i)}, i\in \{1,\dots, m\}$ are non-abelian, we have $\widetilde{\rho}_{\alpha_l}\in I_{\alpha_l}$ for all $l\in \{1,\dots, g+m-1\}$.
Moreover, the product $\prod_{l = 1}^{g+m-1}\widetilde{\rho}_{\alpha_l} = \widetilde{e}_{s,n}(\rho)$. Let $\widetilde{\rho}_{\beta_1}: = \widetilde{\rho}_{\alpha_{1}}\widetilde{\rho}_{\alpha_{2}}$, and for each $k\in \{2,\dots, g+m-2\}$, let $\widetilde{\rho}_{\beta_k} = \widetilde{\rho}_{\beta_{k-1}}\widetilde{\rho}_{\alpha_{k+1}}$. Then for all $k\in \{1,\dots, g+m-2\}$, since $\rho(\beta_k)\neq \pm \mathrm I$, we have $\widetilde{\rho}_{\beta_k}\in I_{\beta_k}$. \smallskip

When $p=2m+1$, 
the intervals $I_{\alpha_l}, l\in \{1,\dots, g+m-1\}$ and $I_{\beta_k}, k\in \{1,\dots, g+m-3\}$ are defined as in the case $p = 2m$. Moreover, we let $I_{\beta_{g+m-2}}: = P(I_{\beta_{g+m-3}}\times I_{\alpha_{g+m-1}})\setminus Z$.
For each decomposition curve $\delta$ and $\rho\in HP^{s,0}_n(\Sigma\setminus T)'$, the association of elements $\widetilde{\rho}_{\delta}\in I_{\delta}$ follows verbatim as above, with the product $\big(\prod_{l = 1}^{g+m-1}\widetilde{\rho}_{\alpha_l}\big) \cdot \widetilde{\rho(c_p)}= \widetilde{e}_{s,n}(\rho)$ where $\widetilde{\rho(c_p)}$ is the preferred lift of $\rho(c_p)$.

\begin{proof}[Proof of Lemma \ref{plp_HP'}]
      We prove the lemma separately for the cases where $p$ is even and where $p$ is odd.
    \medskip
    
    Assume $p = 2m$ for some $m\in \mathbb{N}$. We first show that, for every path $\{\widetilde{A_t}\}\interval \subset \Hyp_n$ and every $\rho\in \widetilde{e}_{s,n}^{-1}(\widetilde{A}_0)$, there is a path $\{\rho_t
    \}\interval$ in $HP^{s,0}_n(\Sigma\setminus T)'$ starting at $\rho_0 = \rho$ such that $\widetilde{e}_{s,n}(\rho_t)= \widetilde{A_t}$ for all $t\in [0,1]$. 
    By Lemma \ref{Prod_confib}, 
    the product map $P: (I_{\beta_{g+m-3}}\times I_{\alpha_{g+m-1}})\cap P^{-1}(\Hyp_n)\to \Hyp_n$ satisfies the strong path-lifting property. Hence there is a path $$\big\{\big(\widetilde{\rho}^t_{\beta_{g+m-3}},\widetilde{\rho}^t_{\alpha_{g+m-1}}\big)\big\}\interval\subset (I_{\beta_{g+m-3}}\times I_{\alpha_{g+m-1}})\cap P^{-1}(\Hyp_n)$$ starting at 
    $\big(\widetilde{\rho}^0_{\beta_{g+m-3}},\widetilde{\rho}^0_{\alpha_{g+m-1}}\big) = \big(\widetilde{\rho}_{\beta_{g+m-3}},\widetilde{\rho}_{\alpha_{g+m-1}}\big)$, such that $P\big(\widetilde{\rho}^t_{\beta_{g+m-3}},\widetilde{\rho}^t_{\alpha_{g+m-1}}\big) = \widetilde{A_t}$ for all $t\in [0,1]$. 
    Similarly, for $k\in \{2,\dots, g+m-3\}$, Lemma \ref{Prod_confib} 
    shows that
    the path 
    $\{\widetilde{\rho}^t_{\beta_k}\}\interval$ in $I_{\beta_k}$ can be successively lifted to a path $\big\{\big(\widetilde{\rho}^t_{\beta_{k-1}},\widetilde{\rho}^t_{\alpha_{k+1}}\big)\big\}\interval\subset (I_{\beta_{k-1}}\times I_{\alpha_{k+1}})\cap P^{-1}(I_{\beta_{k}})$ starting at 
    $\big(\widetilde{\rho}_{\beta_{k-1}},\widetilde{\rho}_{\alpha_{k+1}}\big)$. 
    Finally, applying Lemma \ref{Prod_confib} 
    to the product map $P: (I_{\alpha_1}\times I_{\alpha_2})\cap P^{-1}(I_{\beta_1})\to I_{\beta_1}$, we obtain the lift $\big\{\big(\widetilde{\rho}^t_{\alpha_1}, \widetilde{\rho}^t_{\alpha_2}\big)\big\}\interval$ of $\{\widetilde{\rho}^t_{\beta_1}\}\interval$ starting at $ \big(\widetilde{\rho}_{\alpha_1}, \widetilde{\rho}_{\alpha_2}\big)$. 
    For each $j\in \{1,\dots, g-1\}$, by Proposition \ref{prop_plpliftcommu}, the path $\{\widetilde{\rho}^t_{\alpha_j}\}\interval$ can be lifted to a path $\{(\pm A_{j,t}, \pm B_{j,t})\}\interval$ of non-commuting $\psl$-pairs, starting at $(\pm A_{j,0}, \pm B_{j,0}) = (\rho(a_j), \rho(b_j))$. 
    For each $i\in \{1,\dots, m\}$, by Lemma \ref{parplp}, the path $\{\widetilde{\rho}^t_{\alpha_{g+i-1}}\}\interval$ can be lifted to a path $\{(\pm C_{2i-1,t}, \pm C_{2i,t})\}\interval$ of non-commuting pairs in $\Par^{sgn(s_{2i-1})}\times \Par^{sgn(s_{2i})}$, starting at $(\pm C_{2i-1,0}, \pm C_{2i,0}) = (\rho(c_{2i-1}), \rho(c_{2i}))$. For $t\in [0,1]$, let $\rho_t(a_j) = \pm A_{j,t}$ and $\rho_t(b_j) = \pm B_{j,t}$ for $j\in \{1,\dots, g-1\}$, and let $\rho_t(c_i) = \pm C_{i,t}$ for $i\in \{1,\dots, p\}$. This defines the desired path $\{\rho_t\}\interval$ in $HP^{s,0}_n(\Sigma\setminus T)'$. \medskip
 
    We now show the path-connectedness of the fiber. Let $\widetilde{A} \in \Hyp_n$ and $\rho,\rho'\in \widetilde{e}_{s,n}^{-1}(\widetilde{A})$. We will construct a path $\{\rho_t\}\interval$ in $\widetilde{e}_{s,n}^{-1}(\widetilde{A})$ connecting $\rho$ and $\rho'$ as follows. First, using Lemma \ref{Prod_confib}, 
    we construct paths $\{\widetilde{\rho}^t_{\alpha_l}\}\interval\subset I_{\alpha_l}$ connecting $\widetilde{\rho}_{\alpha_l}$ and $\widetilde{\rho'}_{\alpha_l}$ for all $l\in \{1,\dots, g+m-1\}$ such that for all $t\in [0,1]$, their product $\prod_{l=1}^{g+m-1}\widetilde{\rho}^t_{\alpha_l} = \widetilde{A}$. Next, using Proposition \ref{prop_plpliftcommu}, we lift the path $\{\widetilde{\rho}^t_{\alpha_l}\}\interval\subset I_{\alpha_l}$ to a path of $\psl$-pairs for each $l\in \{1,\dots, g-1\}$. Finally, using Lemma \ref{parplp}, we lift the path $\{\widetilde{\rho}^t_{\alpha_l}\}\interval\subset I_{\alpha_l}$ to a path of pairs of parabolic elements for each $l\in \{g,\dots, g+m-1\}$. These paths of $\psl$-pairs will together define the path $\{\rho_t\}\interval$. \medskip

We first construct, for each $l\in \{1,\dots, g+m-1\}$, a path $\{\widetilde{\rho}^t_{\alpha_l}\}\interval\subset I_{\alpha_l}$ from $\widetilde{\rho}^0_{\alpha_l} = \widetilde{\rho}_{\alpha_l}$ to 
 $\widetilde{\rho}^1_{\alpha_l} = \widetilde{\rho}'_{\alpha_l}$, such that for all $t\in[0,1]$, $\prod_{l=1}^{g+m-1}\widetilde{\rho}^t_{\alpha_l} = \widetilde{A}$. 
To this end, for
each $k\in\{2,\dots,g+m-2\}$, we construct paths
$
\{\widetilde\rho^t_{\beta_{k-1}}\}_{t\in[0,1]}\subset I_{\beta_{k-1}}
\text{ and }
\{\widetilde\rho^t_{\alpha_l}\}_{t\in[0,1]}\subset I_{\alpha_l},
\ l\in \{k+1,\dots,g+m-1\},
$
such that
\begin{enumerate}[(a)]
\item $\widetilde\rho^0_{\beta_{k-1}}=\widetilde\rho_{\beta_{k-1}},\ 
\widetilde\rho^1_{\beta_{k-1}}=\widetilde\rho'_{\beta_{k-1}},
$
and
$\widetilde\rho^0_{\alpha_l}=\widetilde\rho_{\alpha_l},\ 
\widetilde\rho^1_{\alpha_l}=\widetilde\rho'_{\alpha_l}\
\text{for } l\in \{k+1,\dots,g+m-1\}$, and
\item for every $t\in[0,1]$,
$\widetilde\rho^t_{\beta_{k-1}}
\prod_{l=k+1}^{g+m-1}\widetilde\rho^t_{\alpha_l}
=
\widetilde{A}.$
\end{enumerate}
\smallskip

We proceed by induction on $k$. For the base case $k = g+m-2$, note that $\widetilde{\rho}_{\beta_{g+m-3}}\widetilde{\rho}_{\alpha_{g+m-1}} = \widetilde{\rho}'_{\beta_{g+m-3}}\widetilde{\rho}'_{\alpha_{g+m-1}}=\widetilde{A}$. By Lemma \ref{Prod_confib}, 
we may choose a path $\big\{\big(\widetilde{\rho}^t_{\beta_{g+m-3}},\widetilde{\rho}^t_{\alpha_{g+m-1}}\big)\big\}\interval \subset \big(I_{\beta_{g+m-3}}\times I_{\alpha_{g+m-1}}\big)\cap P^{-1}(\widetilde{A})$ connecting $\big(\widetilde{\rho}_{\beta_{g+m-3}},\widetilde{\rho}_{\alpha_{g+m-1}}\big)$ to $\big(\widetilde{\rho}'_{\beta_{g+m-3}},\widetilde{\rho}'_{\alpha_{g+m-1}}\big)$. 
Next, assume for some $k\in\{3,\dots,g+m-2\}$, there are paths 
$ \{\widetilde\rho^t_{\beta_{k-1}}\}_{t\in[0,1]}\subset I_{\beta_{k-1}}$, 
$\{\widetilde\rho^t_{\alpha_l}\}_{t\in[0,1]}\subset I_{\alpha_l},
\ l\in \{k+1,\dots,g+m-1\}$
satisfying Conditions (a) and (b). By Lemma \ref{Prod_confib}, 
we may lift $\{\widetilde{\rho}^t_{\beta_{k-1}}\}\interval$ to a path $\big\{\big(\widetilde{\rho}^t_{\beta_{k-2}},\widetilde{\rho}^t_{\alpha_k}\big)\big\}\interval \subset I_{\beta_{k-2}}\times I_{\alpha_k}$ starting at $\big(\widetilde{\rho}_{\beta_{k-2}},\widetilde{\rho}_{\alpha_k}\big)$. Since $P(\widetilde{\rho}^1_{\beta_{k-2}},\widetilde{\rho}^1_{\alpha_k})=\widetilde{\rho}'_{\beta_{k-1}}$, by Lemma \ref{Prod_confib}, 
we may choose a path $\big\{\big(\widetilde{\rho}^t_{\beta_{k-2}},\widetilde{\rho}^t_{\alpha_k}\big)\big\}_{t\in[1,2]} \subset \big(I_{\beta_{k-2}}\times I_{\alpha_k}\big)\cap P^{-1}(\widetilde{\rho}'_{k-1})$ such that $(\widetilde{\rho}^2_{\beta_{k-2}},\widetilde{\rho}^2_{\alpha_k})=(\widetilde{\rho}'_{\beta_{k-2}},\widetilde{\rho}'_{\alpha_k})$. For every $l\in \{k+1,\dots,g+m-1\}$, we extend the paths $\{\widetilde{\rho}^t_{\alpha_l}\}\interval$ to paths $\{\widetilde{\rho}^t_{\alpha_l}\}_{t\in[0,2]}$ by defining $\widetilde{\rho}^t_{\alpha_l} = \widetilde{\rho}'_{\alpha_l}$ for all $t\in[1,2]$. 
Then, by a simultaneous re-parametrization, this defines paths
$\{\widetilde{\rho}^t_{\beta_{k-2}}\}_{t\in[0,1]}$, $\{\widetilde{\rho}_{\alpha_{k}}^t\}_{t\in[0,1]},\dots,\{\widetilde{\rho}^t_{\alpha_{g+m-1}}\}_{t\in[0,1]}$
satisfying conditions (a) and (b). This completes the construction of desired paths for all $k\in \{2,\dots g+m-2\}$ and in particular, for $k=2$, this gives paths $\{\widetilde{\rho}^t_{\beta_1}\}\interval,\{\widetilde{\rho}^t_{\alpha_3}\}\interval,\dots \{\widetilde{\rho}^t_{\alpha_{g+m-1}}\}\interval$ satisfying conditions (a) and (b). Finally, we can construct $\{\widetilde{\rho}_{\alpha_1}^t\}_{t\in[0,1]}$ and $\{\widetilde{\rho}_{\alpha_2}^t\}_{t\in[0,1]}$using Lemma \ref{Prod_confib}, and re-parametrize $\{\widetilde{\rho}^t_{\alpha_1}\}\interval,\dots, \{\widetilde{\rho}^t_{\alpha_{g+m-1}}\}\interval$ to satisfy the required conditions.
\smallskip

Next, for each $N\in \{1,\dots, g-1\}$, we construct paths 
$\{(\pm A_{j,t},\pm B_{j,t})\}\interval \subset \psl\times \psl$ for all $j\in \{1,\dots N\}$ and $\{\widetilde{\rho}^t_{\alpha_{l}}\}\interval \in I_{\alpha_l}$ for all $l\in \{N+1,\dots g+m-1\}$ such that 
\begin{enumerate}[(1)]
    \item for all $j\in \{1,\dots N\}$, 
$(\pm A_{j,0},\pm B_{j,0}) = (\rho(a_j),\rho(b_j))$ and $(\pm A_{j,1},\pm B_{j,1})=(\rho'(a_j),\rho'(b_j))$ and for all $l\in \{N+1,\dots g+m-1\}$, $\widetilde{\rho}^0_{\alpha_l}=\widetilde{\rho}_{\alpha_l}$ and $\widetilde{\rho}^1_{\alpha_l}=\widetilde{\rho}'_{\alpha_l}$, and
\item $\prod_{j=1}^N\widetilde{R}(\pm {A}_{j,t},\pm {B}_{j,t})\cdot \prod_{l=N+1}^{g+m-1}\widetilde{\rho}^t_{\alpha_l} = \widetilde{A}$ for all $t\in [0,1]$.
\end{enumerate}

We proceed by induction on $N$.
For the base case $N = 1$, by Proposition \ref{prop_plpliftcommu}, we may lift $\{\widetilde{\rho}^t_{\alpha_1}\}\interval$ to a path $\{(\pm A_{1,t},\pm B_{1,t})\}\interval$ starting at $(\rho(a_1),\rho(b_1))$. Since $\widetilde{R}(\pm A_{1,1}, \pm B_{1,1})=\widetilde{\rho}'_{\alpha_1}$, using Theorem \ref{thm_liftcommu}, we may choose a path $\{(\pm A_{1,t},\pm B_{1,t})\}_{t\in[1,2]}\subset \widetilde{R}^{-1}(\widetilde{\rho}'_{\alpha_1})$ such that $(\pm A_{1,2},\pm B_{1,2}) = \big(\rho'(a_1), \rho'(b_1)\big)$. For all $l\in\{2,\dots, g+m-1\}$, we extend the previously constructed paths $\{\widetilde{\rho}^t_{\alpha_l}\}\interval$ to $\{\widetilde{\rho}^t_{\alpha_l}\}_{t \in [0,2]}$ by defining $\widetilde{\rho}^t_{\alpha_l} = \widetilde{\rho}'_{\alpha_l}$ for all $t\in [1,2]$. Then, by a simultaneous re-parametrization, this defines the required paths for $N=1$.
Next, assume for some $N\in\{1,\dots, g-2\}$, we have paths $\{(\pm A_{1,t},\pm B_{1,t})\}\interval,\dots \{(\pm A_{N,t},\pm B_{N,t})\}\interval$ 
and paths $\{\{\widetilde{\rho}^t_{\alpha_{N+1}}\}\interval,\dots, 
\{\widetilde{\rho}^t_{\alpha_{g+m-1}}\}\interval\}$ satisfying Conditions (1) and (2).
Again by Proposition \ref{prop_plpliftcommu}, we may lift $\{\widetilde{\rho}^t_{\alpha_{N+1}}\}\interval$ to a path $\{(\pm A_{N+1,t},\pm B_{N+1,t})\}\interval$ starting at $(\rho(a_{N+1}),\rho(b_{N+1}))$. Since $\widetilde{R}(\pm A_{N+1,1},\pm B_{N+1,1})=\widetilde{\rho}'_{\alpha_{N+1}}$, using Theorem \ref{thm_liftcommu}, we may choose a path $\{(\pm A_{N+1,t},\pm B_{N+1,t})\}_{t\in[1,2]}\subset \widetilde{R}^{-1}(\widetilde{\rho}'_{\alpha_{N+1}})$ such that $(\pm A_{N+1,2},\pm B_{N+1,2})=(\rho'(a_{N+1}),\rho'(b_{N+1}))$. 
For all $l\in\{N+2,\dots, g+m-1\}$, we extend the previously constructed paths $\{\widetilde{\rho}^t_{\alpha_l}\}\interval$ to $\{\widetilde{\rho}^t_{\alpha_l}\}_{t \in [0,2]}$ by defining $\widetilde{\rho}^t_{\alpha_l} = \widetilde{\rho}'_{\alpha_l}$ for all $t\in [1,2]$; and for $j\in\{1,\dots N\}$, we extend the previously constructed paths $\{(\pm A_{j,t},\pm B_{j,t})\}\interval$ to paths on $[0,2]$ by defining $(\pm A_{j,t},\pm B_{j,t}) = (\rho'(a_j),\rho'(b_j))$ for all $t\in [1,2]$. Then, by a simultaneous re-parametrization, this defines paths $\{(\pm A_{j,t},\pm B_{j,t})\}_{t\in[0,1]}$ for $j\in \{1,\dots N+1\}$ and $\{\widetilde{\rho}^t_{\alpha_l}\}_{t\in [0,1]}$ for $l\in \{N+2,\dots, g+m-1\}$ satisfying Conditions (1) and (2). \smallskip

Finally, using Lemma \ref{parplp}, together with simultaneous re-parametrizations as above, for each $k\in \{1,\dots, m\}$, we may construct paths
\[
\big\{(\pm C_{2i-1,t},\pm C_{2i,t})\big\}_{t\in[0,1]} \subset \Par^{sgn(s_{2i-1})}\times \Par^{sgn(s_{2i})},
\quad i\in \{1,\dots,k\}
\]
connecting $\big(\rho(c_{2i-1}), \rho(c_{2i})\big)$  and $\big(\rho'(c_{2i-1}), \rho'(c_{2i})\big)$,
\[
\big\{(\pm A_{j,t},\pm B_{j,t})\big\}_{t\in[0,1]}\subset \psl \times \psl,
\quad j\in \{1,\dots,g-1\}
\]
connecting  $\big(\rho(a_j),\rho(b_j)\big)$ and $\big(\rho'(a_j),\rho'(b_j)\big)$, and
\[
\big\{\widetilde\rho^t_{\alpha_l}\big\}_{t\in[0,1]}\subset I_{\alpha_l},
\quad l\in \{g+k,\dots,g+m-1\}
\]
connecting  $\widetilde{\rho}_{\alpha_l}$ and  $\widetilde{\rho}'_{\alpha_l}$, such that for all $t\in[0,1]$,
\[
\prod_{j=1}^{g-1}
\widetilde R(\pm A_{j,t},\pm B_{j,t})\cdot
\prod_{i=1}^{k}
\operatorname{ev}(\pm C_{2i-1,t},\pm C_{2i,t})\cdot
\prod_{l=g+k}^{g+m-1}
\widetilde\rho^t_{\alpha_l}
=
\widetilde A.
\]
Then when $k=m$, by setting $\rho^t(a_j)=\pm A_{j,t}, \rho^t(b_j)=\pm B_{j,t}$ for $j\in \{1,\dots g-1\}$ and $\rho^t(c_{i})=\pm C_{i,t}$ for all $i\in \{1,\dots p\}$, we obtain a path $\{\rho^t\}\interval \subset \widetilde{e}_{s,n}^{-1}(\widetilde{A})$ connecting $\rho$ and $\rho'$.\medskip

Assume $p = 2m+1$ for some $m\in \mathbb{N}\cup \{0\}$. For every path $\{\widetilde{A_t}\}\interval \subset \Hyp_n$ and every $\rho\in \widetilde{e}_{s,n}^{-1}(\widetilde{A}_0)$, using Lemma \ref{Prod_confib}, 
we may lift $\{\widetilde{A_t}\}\interval$ to a path $\big\{\big(\widetilde{\rho}^t_{\beta_{g+m-2}}, \widetilde{\rho}^t_{p}\big)\big\}\interval\subset I_{\beta_{g+m-2}}\times \Par^{sgn(s_p)}_0$ starting at $\big(\widetilde{\rho}_{\beta_{g+m-2}}, \widetilde{\rho(c_p)}\big)$. 
Then, following the construction as in the case $p = 2m$, we may lift the path $\{\widetilde{\rho}^t_{\beta_{g+m-2}}\}\interval$ to a path $\{\rho_t\}\interval$ in $HP^{s,0}_n(\Sigma\setminus T)'$ starting at $\rho_0 = \rho$ such that $\widetilde{e}_{s,n}(\rho_t)= \widetilde{A_t}$ for all $t\in [0,1]$. Next, we show the connectedness of $\widetilde{e}_{s,n}^{-1}(\widetilde{A})$ for any $\widetilde{A} \in \Hyp_n$.  
For any $\rho,\rho'\in \widetilde{e}_{s,n}^{-1}(\widetilde{A})$, using Lemma \ref{Prod_confib}, 
we may choose a path $\big\{\big(\widetilde{\rho}^t_{\beta_{g+m-2}}, \widetilde{\rho}^t_p\big)\big\}\interval\subset \big(I_{\beta_{g+m-2}}\times \Par^{sgn(s_p)}_0\big) \cap P^{-1}(\widetilde{A})$ with endpoints $\big(\widetilde{\rho}_{\beta_{g+m-2}}, \widetilde{\rho(c_p)}\big)$ and $\big(\widetilde{\rho}'_{\beta_{g+m-2}}, \widetilde{\rho'(c_p)}\big)$. 
Then, following the case of $p=2m$, we obtain a path $\{\rho_t\}\interval$ within $\widetilde{e}_{s,n}^{-1}(\widetilde{A})$ connecting $\rho$ and $\rho'$. This completes the proof.
\end{proof}

\section{$\mathcal{EL}^s_n$ is a subset of $\mathcal{U}^s_n$}\label{sec_U}
In this section, we prove Proposition \ref{U}.
Let $\Sigma=\Sigma_{g,p}$, $g\geqslant 2,p\geqslant 1$, and let $n\in \mathbb{Z}$, $s\in \{\pm 1\}^p$. 
Recall from Section \ref{sec_outline} the notation used to define $\mathcal{U}^s_n$. Assume $[\rho] \in \mathcal{EL}^s_n$, i.e., there exists $\gamma \in \mathcal{S}^{ns}$ such that $\rho(\gamma)$ is elliptic of infinite order. Denote by $D_{[\rho]}$ the subspace of $T^*_{[\rho]}\mathcal{M}^s_n$ spanned by $dF_{\alpha}$ for simple closed curves $\alpha$ with $|Tr(\rho(\alpha))|<2$. Clearly $dF_{\gamma}\in D_{[\rho]}$. We claim $D_{[\rho]}=T^*_{[\rho]}\mathcal{M}^s_n$.
\smallskip

To prove the claim, we need the following lemmas.

\begin{lemma}\label{int_1_drho}\cite[Lemma 6.12]{Marche-Wolff}
Let $\delta \in \mathcal{S}$ such that $i(\delta,\gamma)=1$. Then $dF_{\delta}\in D_{[\rho]}$.  
\end{lemma}
\begin{lemma}\label{int0}
    Let $\delta$ be a simple closed curve such that $i(\delta,\gamma)=0$. Then $dF_{\delta}\in D_{[\rho]}$.
\end{lemma}
\begin{proof}
    If $\delta$ is non-separating, the proof can be adapted verbatim from \cite[Lemma 6.13]{Marche-Wolff}.
    We prove the case when $\delta$ is separating. Let $\Sigma =\Sigma_1\cup_{\delta}\Sigma_2$. If  $\Sigma_2$ has positive genus, then the proof follows verbatim from \cite[Lemma 6.13]{Marche-Wolff}. It remains to show the case where $\Sigma_2$ is a surface homeomorphic to a punctured sphere $\Sigma_{0,p_2}$ and $\Sigma_1$ is a surface with positive genus. We proceed by induction on the number of punctures $p_2$ of $\Sigma_2$. 
    For the base case $p_2 = 2$, $\delta$ is peripheral and $F_\delta$ is constant, hence $dF_{\delta}=0$, which is clearly in $D_{[\rho]}$. Now assume, for every separating curve $\eta$ disjoint from $\gamma$ such that $\Sigma = \Sigma'_1\cup_\eta \Sigma'_2$ where $\Sigma'_2$ is homeomorphic to $\Sigma_{0,p_2-1}$,
    $dF_\eta$ belongs to $D_{[\rho]}$. Choose a curve $\zeta$ such that $i(\zeta,\gamma)=1$, $i(\zeta,\delta)=2$, and $\zeta = \zeta_1*\zeta_2$ where $\zeta_1$ is non-separating and $\zeta_2$ is a peripheral curve. See Figure \ref{peripheral}. By Lemma \ref{BAn}, up to replacing $\zeta$ by a power of its Dehn twist about $\gamma$, we may assume $|Tr(\rho(\zeta))|\neq 0$. 
    \begin{figure}[h!]
            
            \centering
            \begin{overpic}[width=0.8\textwidth, trim = 0 400 0 150]{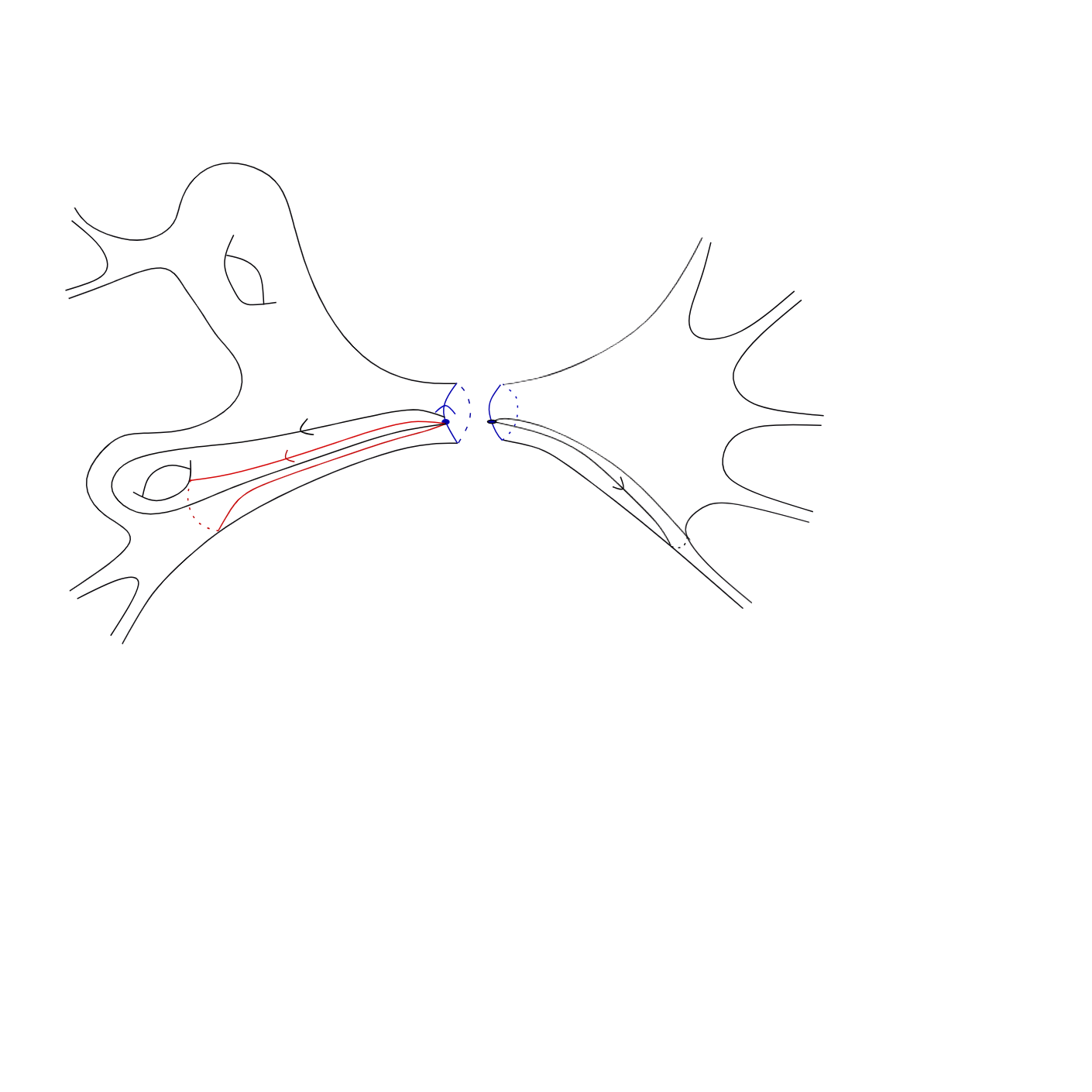}
            \put(30,34){\scalebox{0.7}{$\zeta_1$}}
            \put(20,29){\scalebox{0.7}{\textcolor{red}{$\gamma$}}}
            \put(39,34.5){\scalebox{0.7}{\textcolor{blue}{$\delta$}}}
            \put(49,33){\scalebox{0.7}{$\zeta_2$}}
                
            \end{overpic}
            \vspace{-1cm}
             \caption{The choice of $\zeta = \zeta_1*\zeta_2$}
            \label{peripheral}
          
        \end{figure}
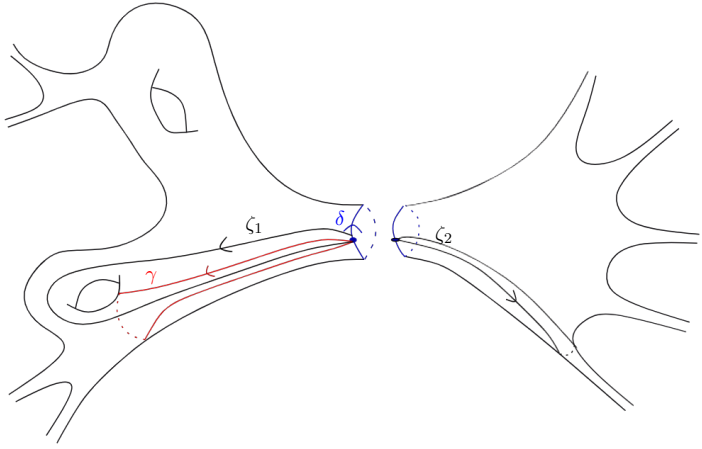
    
    Recall from Section 3 that there exist neighbourhoods $\overline{W}^{[\overline{\rho}]}_{s,n}\subset \overline{\mathcal{M}}_{s,n}$ of $[\overline{\rho}]$
     and $W^{[\rho]}_{s,n}\subset \mathcal{M}^s_n$ of $[\rho]$ diffeomorphic by $\Pi$. For each $[\phi]\in W^{[\rho]}_{s,n}$, let $[\overline{\phi}] = \Pi^{-1}([\phi])$, and let $B_1(\phi)= \overline{\phi}(\zeta_1), B_2(\phi)= \overline{\phi}(\zeta_2)$ and $C(\phi)=\overline{\phi}(\delta)$. Now consider the trace identity,
   \[
    \begin{aligned}
    Tr(B)Tr(C) &= Tr(B_1B_2)Tr(C) \\
               &= Tr(B_1B_2C)+Tr(B_1B_2C^{-1}) \\
               &= Tr(B_1)Tr(B_2C)-Tr(B_1C^{-1}B_2^{-1})+Tr(B_2)Tr(B_1C^{-1})-Tr(B_1B_2^{-1}C^{-1}).
\end{aligned}
\]
This translates to the equation 
\[F_{\zeta_1\zeta_2}F_{\delta}=F_{\zeta_1}F_{\zeta_2\delta}-F_{\zeta_1\delta^{-1}\zeta_2^{-1}}+F_{\zeta_2}F_{\zeta_1\delta^{-1}}-F_{\zeta_1\zeta_2^{-1}\delta^{-1}}.\]
 By Lemma \ref{int_1_drho}, for any simple curve $\alpha$ containing $\zeta_1$, $dF_{\alpha}$ belong to $D_{[\rho]}$; and since $\zeta_2$ is peripheral, $dF_{\zeta_2}=0$. Moreover, since the curve $\zeta_2\delta$ separates $\Sigma$ into a surface with positive genus and a punctured sphere with $p_2-1$ punctures, by the induction hypothesis, $dF_{\zeta_2\delta} \in D_{[\rho]}$. Differentiating the equation above, since all terms except possibly $F_\zeta(\rho)\,dF_\delta$ are in $D_{[\rho]}$, $F_\zeta(\rho)dF_\delta$ is also in $D_{[\rho]}$ where $F_\zeta(\rho)\neq 0$. This completes the inductive step.
\end{proof}

For every simple closed curve $\gamma$, define $f_{\gamma}:\mathcal{M}^s_n\rightarrow \mathbb{R}$ as $f_{\gamma}([\phi])=Tr(\phi(\gamma))^2$ for every $[\phi]\in \mathcal{M}^s_n$.
The following Lemma \ref{poisson} can be adapted verbatim to our case of relative character varieties. 
\begin{lemma}\label{poisson}\cite[Lemma 6.14]{Marche-Wolff}\label{Poisson}
There exists $\delta$ such that $i(\gamma,\delta)=1$ and $\{f_\gamma,f_\delta\}\neq 0$. In particular, $df_\gamma\neq 0$.
\end{lemma}
 Choose $[\overline{\rho}]\in \overline{\mathcal{M}}_{s,n}$ such that $[\pi(\overline{\rho})]=[\rho]$ and $Tr(\overline{\rho}(\gamma))=2\cos(\theta)$. Choose diffeomorphic neighbourhoods $W^{[\rho]}_{s,n}\subset \mathcal{M}^s_n$ of $[\rho]$ and $\overline{W}^{[\overline{\rho}]}_{s,n}\subset \overline{\mathcal{M}}_{s,n}$ of $[\overline{\rho}]$ as in section \ref{sec_outline}. Since $df_{\gamma}$ is proportional to $dF_{\gamma}$, we also have $dF_{\gamma}\neq 0$. In particular, $W^{[\rho]}_{\theta}=F_{\gamma}^{-1}(2\cos(\theta))$ is a proper closed submanifold of $W^{[\rho]}_{s,n}$. 
 
 
 
\medskip

We now complete the proof of the claim $D_{[\rho]}=T^*_{[\rho]}\mathcal{M}^s_n$ by adapting the final step of the proof of \cite[Proposition 6.4]{Marche-Wolff}.

\begin{proof}[Proof of Proposition \ref{U}]
 Let $v \in T_{[\rho]}\mathcal{M}^s_n$ orthogonal to $D_{[\rho]}$. It suffices to prove $v=0$. Since $dF_{\gamma}\in D_{[\rho]}$, $dF_{\gamma}(v)=0$ and hence $v\in T_{[\rho]}W^{[\rho]}_{\theta}$.
 Let $W_{\theta}^{[\rho]}(\Sigma \setminus \gamma)$ be the set of restrictions $[\phi|_{\pi_1(\Sigma\setminus \gamma)}]$ of $[\phi]\in W_{\theta}^{[\rho]}$, and let
 $r: W_{\theta}^{[\rho]}\rightarrow W_{\theta}^{[\rho]}(\Sigma \setminus \gamma)$ be the restriction map.  
 The restriction $r([\rho])$ is a smooth point. Indeed as $\gamma$ is non-separating, we can choose a simple closed curve $\beta$ such that $i(\gamma,\beta) =1$. Then the commutator $[\gamma,\beta]$ is sent to a hyperbolic element which implies $r[\rho]$ is non-elementary. The differentials $dF_\delta$ for non-peripheral simple closed curves $\delta$ disjoint from $\gamma$ span $T^*_{{r}[{\rho}]}{W}^{[{\rho}]}_{\theta}(\Sigma\setminus \gamma)$. (See Section 2 and Lemma 3.1 in \cite{goldman_SU2} for more details.) 
Moreover, by Lemma \ref{int0}, $r^*(dF_{\delta})(v)=0$ for every simple curve $\delta$ disjoint from $\gamma$. Thus $Dr(v)=0$. The kernel of $Dr$ is precisely the span of the symplectic gradient of $f_\gamma$ at $[\rho]$ that we denote by $(X_{\gamma})_{[\rho]}$. Thus, $v = \lambda (X_{\gamma})_{[\rho]}$ for some $\lambda \in \mathbb{R}$. 
By Lemma \ref{Poisson}, there is an $\alpha \in \mathcal{S}$ such that $i(\alpha,\gamma)=1$ and $\{f_\alpha,f_\gamma\}\neq 0$. 
By Lemma \ref{int_1_drho}, $dF_\alpha \in D_{[\rho]}$, and hence $df_{\alpha}\in D_{[\rho]}$. Thus, $(X_\alpha,v)=df_\alpha(v) = 0$ where $(.,.)$ is the Goldman symplectic form. Therefore, $((X_\alpha)_{[\rho]}, \lambda (X_\gamma)_{[\rho]})=\lambda\{f_\alpha,f_\gamma\}=0$, which implies $\lambda =0$. We conclude that $v = 0$. 
\end{proof}

\section{$\mathcal{EL}^s_n$ has full measure in $\mathcal{NH}^s_n$}\label{sec_EfullNH}

In this section, we prove Proposition \ref{EfullNH}. 
By abuse of notation, we use the term ``simple closed curve" to refer both to the actual curve on the surface $\Sigma = \Sigma_{g,p}$ and to its homotopy class in $\pi_1(\Sigma)$.
Recall that $\mathcal{NH}^s_n$ is a subset of $\Mns$ consisting of the classes of all non-elementary type-preserving representations that map at least one non-peripheral simple closed curve to a non-hyperbolic element.
\smallskip

In order to prove Proposition \ref{EfullNH}, we define a full measure subset $X^s_n$ of $\Mns$ as follows.
Recall from Section \ref{sec_outline} that $\N$ 
is the subspace of $\Mns$ consisting of $[\rho]$ mapping
no non-peripheral
simple closed curve to the identity, a parabolic element, or to an elliptic element of finite order. 
By Lemma \ref{Nfull}, $\N$ has full measure in $\mathcal{M}^s_n$. Let $\nonabel$ be the subspace of $\mathcal{M}^s_n$ consisting of $[\rho]$ such that, for every embedded subsurface $P$ of $\Sigma$ homeomorphic to a pair of pants, the restriction $\rho|_{\pi_1(P)}$ is non-abelian. We claim $\nonabel$ has full measure in $\mathcal{M}^s_n$. Let $P$ be a pair of pants in $\Sigma$ such that $\pi_1(P)=\langle a,b\rangle$ where $a$ and $b$ are represented by two distinct boundary components of $P$. Let $R_P:\mathcal{M}^s_n\to \mathbb{R}$ be a map defined by $R_P\big([\rho]\big)=\tr\big([A_\rho, B_\rho]\big)$, where $A_\rho$ and $B_\rho$ are the lifts of $\rho(a)$ and $\rho(b)$ in $\SL$, respectively. Notice that the commutator $[A_\rho, B_\rho]$ is independent of the choice of the lifts $A_\rho$ and $B_\rho$, thus its trace is well-defined. 
Then $R_P^{-1}(2)$ is a closed 
proper real-analytic subset of  $\mathcal{M}^s_n$, hence is of measure zero \cite{Mityagin}. Since the set of pairs of pants in $\Sigma$ is countable up to isotopy, the union $\bigcup_{P}R_P^{-1}(2)$ for every pair of pants $P \subset \Sigma$ has measure zero in $\mathcal{M}^s_n$. Since $\mathcal{M}^s_n\setminus \nonabel$ is contained in $\bigcup_{P}R_P^{-1}(2)$, it has measure zero as claimed. Let $X^s_n= \N\cap \nonabel$. Then $X^s_n$ has full measure in $\mathcal{M}^s_n$, hence it suffices to show that $\ELns\cap X^s_n = \NHns\cap X^s_n$. Therefore, Proposition \ref{EfullNH} follows directly from the proposition below.

\begin{proposition}\label{prop_ELnNHn}
    Let $\Sigma = \Sigma_{g,p}$ with $g\geqslant 2$ and $p\geqslant 1$, and let $s\in \{\pm 1\}^p$ and $n\in \{\chi(\Sigma) + p_+(s) + 1,\dots, -\chi(\Sigma) - p_-(s)- 1\}$. 
    Then $\ELns\cap X^s_n = \NHns\cap X^s_n$. Consequently, $\mathcal{E}^s_n$ has full measure in $\NHns$.
\end{proposition}

    We begin with several lemmas that will be used to prove Proposition~\ref{prop_ELnNHn}.    
    \begin{lemma}
    \label{lem_MW_posgenus}
       Let $\Sigma = \Sigma_{g,p}$ with $g\geqslant 2$ and $p\geqslant 1$. Let $\gamma$ be a simple closed curve that separates $\Sigma$ into two subsurfaces $\Sigma_1$ and $\Sigma_2$. Let $s\in \{\pm 1\}^p$ and $n\in \big\{\chi(\Sigma) + p_+(s) ,\dots, -\chi(\Sigma) - p_-(s) \big\}$, and let 
        $[\rho]\in \NHns\cap X^s_n$ such that $\rho(\gamma)$ is elliptic of infinite order.
       Suppose that one of the following holds:
       \begin{enumerate}[(1)]
           \item Both $\Sigma_1$ and $\Sigma_2$ have positive genus.

           \item $\Sigma_2$ has genus 0, and there is a simple closed curve in $\Sigma_2$ that maps to a hyperbolic element.
       \end{enumerate}
        Then there exists a non-separating simple closed curve in $\Sigma$ that maps to an elliptic element.
    \end{lemma}
    \begin{proof}
        Let $x$ be a point on $\gamma$.
        We first consider Case (1), where both $\Sigma_1$ and $\Sigma_2$ have positive genus.
        For the case where either $\Sigma_1$ or $\Sigma_2$ has genus greater than 1, see the proof of \cite[Lemma 6.17]{Marche-Wolff}. If both $\Sigma_1$ and $\Sigma_2$ have genus 1, then we may assume without loss of generality  that $\pi_1(\Sigma_1\textcolor{black}{, x})$ contains at least one peripheral element of $\pi_1(\Sigma\textcolor{black}{, x})$, say $c$. 
         Let $a\in \pi_1(\Sigma_1\textcolor{black}{, x})$ be a non-separating simple closed curve  such that 
         $ac\in \pi_1(\Sigma_1\textcolor{black}{, x})$ is simple.
         If either $\rho(a)$ or $\rho(ac)$ is elliptic, then the statement follows. If both $\rho(a)$ and $\rho(ac)$ are hyperbolic, then they have different axes in $\mathbb{H}^2$, as otherwise $\rho(c)$ cannot be parabolic. 
         Consider an image $(ac)^N a (ac)^{-N}$ of $a$ under the $N$th power of the Dehn twist along $ac$. As $N$ goes to $+\infty$ or $-\infty$, the distance between the axis of $\rho\big( (ac)^N a (ac)^{-N}\big)$ and the fixed point of $\rho(\gamma)$ tending to $+\infty$. The rest of the proof follows verbatim as in \cite[Lemma 6.17]{Marche-Wolff}. 
         \smallskip
         
         For Case (2) where \(\Sigma_2\) has genus \(0\), as $\Sigma$ has genus $g\geqslant 2$,
        \(\Sigma_1\) has genus greater than \(1\). Let
        \(b\) be a simple closed curve in
        \(\pi_1(\Sigma_2,x)\) such that \(\rho(b)\) is hyperbolic. The proof is the same as
        the proof of \cite[Lemma 6.17]{Marche-Wolff} for the case where exactly one of the separated subsurfaces has genus 1, with the
        curve \(b\) playing the role of the curve chosen on such subsurface.
        Although \(b\) is separating in the punctured sphere \(\Sigma_2\), the
        argument in the proof only uses that \(b\) is simple and that \(\rho(b)\) is hyperbolic. Hence the rest of the proof follows verbatim.
    \end{proof} 

    For a hyperbolic element $\pm A$, let $l_A$ denote its axis in $\mathbb{H}^2$, oriented in the direction in which $\pm A$ translates points along $l_A$. Let $x_0$ be a point in $\mathbb{H}^2$. 
    We say $l_A$ \emph{turns positively around $x_0$} if, when traveling along $l_A$ toward its attracting fixed point, the point $x_0$ lies on the left-hand side of $l_A$.
    
    \begin{lemma}\label{lem_ellnonsep}
        Let $\Sigma = \Sigma_{g,p}$ with $g\geqslant 2$ and $p\geqslant 2$. Let $\gamma$ be a separating simple closed curve in $\Sigma$ that separates $\Sigma$ into $\Sigma_1$ and $\Sigma_2$, where $\Sigma_2$ has genus 0. Let 
        $s\in \{\pm 1\}^p$ and $n\in \big\{\chi(\Sigma) + p_+(s) , \dots, -\chi(\Sigma) - p_-(s)\big\}$, and 
        let $[\rho]\in \NHns\cap X^s_n$ such that $\rho(\gamma)$ is elliptic of infinite order with fixed point $x_0\in \mathbb{H}^2$. 
        \textcolor{black}{Let $x$} be a point on $\gamma$, and let $c_1$ be a preferred peripheral element of $\pi_1(\Sigma\textcolor{black}{, x})$ contained in $\pi_1(\Sigma_2\textcolor{black}{, x})$ that maps to a negative parabolic element.
        Suppose that there is a sequence $\{\delta_k\}_{k = 1}^\infty \subset \pi_1(\Sigma, x)$ of simple closed curves in $\pi_1(\Sigma_1,x)$ such that for each $k\in \mathbb{Z}_+$,
        \begin{enumerate}[(a)]
            \item $\delta_kc_1$ is simple, 
            \item $\rho(\delta_k)$ is hyperbolic, and
            \item the (oriented) axis $l_k$ of $\rho(\delta_k)$ turns positively around $x_0$.
        \end{enumerate}
        Assume further that 
        either every $\delta_k$ is non-separating, or every $\delta_k$ is separating and each connected component of $\Sigma\setminus [\delta_k]$
        has positive genus.
        In addition, suppose that the distance $d(l_k, x_0)\to \infty$ as $k\to \infty$.
        Then $\rho$ also sends a non-separating simple closed curve to an elliptic element.
    \end{lemma}

    \begin{proof}
       We consider the two cases separately: First, that all $\delta_k$ are non-separating; and second, that all $\delta_k$ are separating and each connected component of $\Sigma\setminus [\delta_k]$ has positive genus.
        \medskip
        
        We first assume that all $\delta_k$ are non-separating.
        Up to $\psl$-conjugation of $\rho$, we assume that $\rho(c_1) = \pm \parmp$. 
        Notice that for each $m\in \mathbb{Z}$ and $k\in \mathbb{Z}_+$, $\gamma^m\delta_k \gamma^{-m} c_1$ is represented by a non-separating simple closed curve. We will show that, for some $k_0\in \mathbb{Z}_+$ and $M\in \mathbb{Z}$, $\rho(\gamma^M \delta_{k_0} \gamma^{-M}c_1)$ is elliptic.
        \smallskip
        
        First, we choose $k_0$ as follows.
        For each $k\in \mathbb{Z}_+$, the hyperbolic element $\rho(\delta_k)$ fixes a geodesic axis $l_k$ with displacement $\lambda_k>0$. 
        Notice that $\rho(\gamma^m\delta_k \gamma^{-m})$ is hyperbolic with the same displacement $\lambda_k$.
        The axis of $\rho(\gamma^m\delta_k \gamma^{-m})$ is the rotation $\rho(\gamma)^m (l_k)$ of $l_k$ around the point $x_0$, which still turns positively around $x_0$ with the same distance $d(\rho(\gamma)^m(l_k), x_0) = d(l_k, x_0)$. 
        Since $\rho(\gamma)$ is elliptic of infinite order,
        by replacing $\delta_k$ by $\gamma^m\delta_k \gamma^{-m}$ for suitable $m\in \mathbb{Z}$, we can assume that each $l_k$ is a semicircle of radius less than $\mathrm{Im}(x_0)$, where the point $\mathrm{Re}(x_0)$ in $\mathbb{R} \subset \partial \mathbb{H}^2$ 
        lies between two endpoints of $l_k$ in $\mathbb{R}$. Note that in this case, as the semicircle $l_k$ turns positively around $x_0$, it is oriented from left to right in the upper half plane. See Figure \ref{fig: lemma_6.4.jpg}. 
        Moreover, $\displaystyle\lim_{k\to\infty} d(l_k, x_0)\to \infty$ implies that the Euclidean radius $r_k$ of $l_k$ converges to $0$. In particular, there exists $k_0\in \mathbb{Z}$ such that $r_{k_0}<\frac{1}{2}$.
            \begin{figure}[h!]
                \centering
                \begin{overpic}[width=0.4\textwidth, trim = 0 0 0 0]{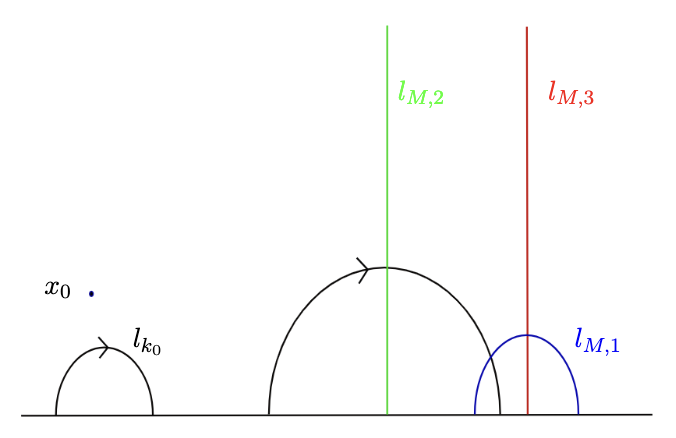}
                \end{overpic}
                \caption{Configuration of $l_{k_0}$ in $\mathbb{H}^2$}
                {\label{fig: lemma_6.4.jpg}}
                
            \end{figure}
        
        Next, we choose $M$ as follows. 
        Recall that for each $m\in \mathbb{Z}$, the hyperbolic element $\rho(\gamma^m\delta_{k_0}\gamma^{-m})$ is represented by a composition $s_{m,1}s_{m,2}$  of two reflections $s_{m,1}$ and $s_{m,2}$ across the axes $l_{m,1}$ and $l_{m,2}$ that are perpendicular to the axis of $\rho(\gamma^m\delta_{k_0}\gamma^{-m})$.
        For $m\in\mathbb{Z}$ where the axis of $\rho(\gamma^m\delta_{k_0}\gamma^{-m})$ is a semicircle with center $x_m\in \mathbb{R}$, we let $l_{m,2}$ be a vertical line $x = x_m$. 
        Denote by $J_+$ the subinterval of $(-\pi, \pi)$ consisting of rotation angles $\theta$ such that the $\theta$-rotation of $l_{k_0}$ around $x_0$ is a semicircle in $\mathbb{H}^2$ oriented from left to right. Let \(\theta_m\in (-\pi, \pi)\) be the rotation parameter of the oriented geodesic
        \(\rho(\gamma)^m(l_{k_0})\) about \(x_0\), with $\theta_0 = 0$.
        Then \(l_{m,1}\) lies in
        the right half-plane with respect to \(l_{m,2}\) with distance $d(l_{m,1},l_{m,2})=\frac{\lambda_{k_0}}{2}$ if and only if \(\theta_m\in J_+\). As we define $l_{m,3}$ to be a vertical line $x = x_m + \frac{1}{2}$ which also lies in the right half-plane with respect to $l_{m,2}$, and define $s_{m,3}$ be a reflection across $l_{m,3}$, the composition $s_{m,2} s_{m,3}$ equals $\rho(c_1)= \pm \parmp$. See Figure \ref{fig: lemma_6.4.jpg}. Therefore, we have $\rho(\gamma^m\delta_{k_0}\gamma^{-m}c_1) = (s_{m,1}s_{m,2}) (s_{m,2}s_{m,3}) = s_{m,1}s_{m,3}$.
        Since $\rho(\gamma)$ is an irrational rotation, 
        the orbit
        \(\{\theta_m\}_{m\in\mathbb Z}\) is dense in \((-\pi, \pi)\); and since $\theta_0\in J_+$ with the radius of $l_{k_0}$ equals to $r_{k_0}$,
        the radii of
        \(\rho(\gamma)^m(l_{k_0})\) with \(\theta_m\in J_+\) are dense in
        \([r_{k_0},\infty)\). Since \(r_{k_0}<1/2\), we may choose \(M\) with
        \(\theta_M\in J_+\) such that the radius of \(\rho(\gamma)^M(l_{k_0})\) is
        arbitrarily close to \(1/2\), in which case the axes $l_{M,1}$ and $l_{M,3}$ intersect; and  $\rho(\gamma^M\delta_{k_0}\gamma^{-M}c_1) = s_{M,1}s_{M,3}$ is elliptic.
        \medskip
    
        We now assume that all $\delta_k$ are separating and each component of $\Sigma\setminus [\delta_k]$ has positive genus. 
        In this case, for each $m\in \mathbb{Z}$ and $k\in \mathbb{Z}_+$, $\gamma^m\delta_k \gamma^{-m} c_1$ is represented by a separating simple closed curve whose complement has components of positive genus. 
        As in the previous case, we can choose $k_0\in \mathbb{Z}_+$ and $M\in \mathbb{Z}$ such that $\rho(\gamma^M \delta_{k_0} \gamma^{-M}c_1)$ is elliptic, which is of infinite order since $[\rho]\in \N$. Then Lemma \ref{lem_MW_posgenus} completes the proof.
    \end{proof}
     
    \begin{lemma}\label{lem_fourpuncsphere}
    Let $\Sigma_{0,4}$ be a four-holed sphere, and let $c_1,c_2,c_3,c_4$ denote the preferred peripheral elements of its fundamental group $\pi_1(\Sigma_{0,4})$. Let $\gamma = c_1c_2$ and $s = (0,0,-1,-1)$. Let $\rho\in \HP^s_{-1}(\Sigma_{0,4})$ that maps $\gamma$ to an elliptic element of infinite order. Then there exists $n\in \mathbb{Z}$ such that $\rho$ maps the simple closed curve $c_1\gamma^nc_3\gamma^{-n}$ to an elliptic element.  
\end{lemma}
\begin{proof}
Up to a $\psl$-conjugation of $\rho$, we can assume that $\rho(\gamma) = 
\pm\begin{bmatrix}
    \cos(\theta) & \sin(\theta)\\
    -\sin(\theta) & \cos(\theta)
\end{bmatrix}$ for some $\theta \in (0,\pi)$. Since $\rho(c_3)\in \Par^-$, by further conjugation of $\rho$ by an element of the centralizer of $\rho(\gamma)$, we may assume $\rho(c_3)=\pm\begin{bmatrix}
    1 & t\\
    0 & 1
\end{bmatrix}$ for some $t<0$.
Since $e(\rho)=-1$, we may choose a lift $\bar{\rho}:\pi_1(\Sigma_{0,4})\rightarrow \SL$ of $\rho$ such that $\tr\big(\bar{\rho}(c_1)\big)>2, \tr\big(\bar{\rho}(c_2)\big)>2, \tr\big(\bar{\rho}(c_3)\big)=2$ and $\tr\big(\bar{\rho}(c_4)\big)=-2$. Thus, $\bar{\rho}(c_3)=\begin{bmatrix}
    1 & t\\
    0 & 1
\end{bmatrix}$. For $i\in \{1,2,3,4\}$, let $\widetilde{\rho(c_i)}$ be the lift of $\rho(c_i)$ in $\overline{\Hyp_0}$. Since $\widetilde{\rho(c_3)}^{-1}$ and $\widetilde{\rho(c_4)}^{-1}$ are positive parabolic and $\widetilde{\rho(c_1)}\widetilde{\rho(c_2)}=z^{-1}\widetilde{\rho(c_4)}^{-1}\widetilde{\rho(c_3)}^{-1}$ is elliptic,  Lemma \ref{parplp} implies 
$\widetilde{\rho(c_1)}\widetilde{\rho(c_2)}\in \Ell_{-1}$. Then by Lemma \ref{lem_offdiag_Ell}, we have $\bar{\rho}(\gamma) = \begin{bmatrix}
    \cos(\phi) & \sin(\phi)\\
    -\sin(\phi) & \cos(\phi) 
\end{bmatrix},\ \phi = \theta - \pi\in (-\pi,0).$ Let $\bar{\rho}(c_1) = \begin{bmatrix}
    p & q\\
    r & s
\end{bmatrix}$. Since $\bar{\rho}(\gamma c_3c_4) = \bar{\rho}(c_1c_2c_3c_4)=\mathrm{I}$, the trace $\tr\big(\bar{\rho}(\gamma c_3)\big)=\tr\big(\bar{\rho}(c_4)\big) = -2$; and solving the equation $\tr\big(\bar{\rho}(\gamma c_3)\big)= 2cos(\phi)-tsin(\phi) = -2$ for $t$ gives $t = 2cot(\phi/2).$
Next, since $\tr\big(\bar{\rho}(c_2)\big) = \tr\big(\bar{\rho}(c_1^{-1}\gamma)\big)>2$, we obtain the inequality $$\tr\big(\bar{\rho}(c_1)^{-1}\bar{\rho}(\gamma)\big) = (s+p)cos(\phi) + (q-r)sin(\phi)>2,$$ where we substitute $t = 2cot(\phi/2)$ and rewrite as
\begin{equation}\label{trb}
    \frac{t^2-4}{t^2+4}(p+s) + \frac{4t}{t^2+4}(q-r)>2.
\end{equation}
Further, a computation shows $$\tr\big(\bar{\rho}(c_1\gamma^nc_3\gamma^{-n})\big)=(p+s)+\frac{t(r-q)}{2}+\frac{t(p-s)}{2}sin(2n\phi) + \frac{t(r+q)}{2}cos(2n\phi).$$
Since $\phi$ is an irrational multiple of $\pi$, as $n$ varies in $\mathbb{Z}$, $\tr\big(\bar{\rho}(c_1\gamma^nc_3\gamma^{-n})\big)$ is dense in the interval $J=\big[(p+s)+\frac{t(r-q)}{2}-K, (p+s)+\frac{t(r-q)}{2}+K\big]$ where $K= \big(\sqrt{t^2(p-s)^2+t^2(r+q)^2}\big)/2$. 
Hence, to show $\tr\big(\bar{\rho}(c_1\gamma^nc_3\gamma^{-n})\big)\in (-2,2)$ for some $n\in \mathbb{Z}$, it suffices to show $0 \in J$. To this end, we show $K>|(p+s)+\frac{t(r-q)}{2}|$. Denote by 
$$\alpha :=K^2-\bigg[(p+s)+\frac{t(r-q)}{2}\bigg]^2 = \frac{t^2}{4}\big((p-s)^2+(r+q)^2-(r-q)^2\big)-(p+s)^2-t(p+s)(r-q).$$ 
We claim $\alpha>0$.
After rearranging and using $ps-rq=1$, we obtain $$\alpha=[(t^2-4)(p+s)^2-4t^2-4t(p+s)(r-q)]/4 = \frac{(p+s)}{4}[(t^2-4)(p+s)+4t(q-r)]-t^2.$$ By the inequality \ref{trb} and $p+s>2$ we have $$\alpha = \frac{(p+s)}{4}[(t^2-4)(p+s)+4t(q-r)]-t^2 > (p+s)(t^2+4)/2-t^2>0,$$ as desired.
\end{proof}

    Let $\gamma$ be a simple closed curve in $\Sigma$ that separates $\Sigma$ into a surface $\Sigma_1$ with genus at least two and a punctured sphere $\Sigma_2$. 
    We consider decompositions $\mathcal{P}$ of $\Sigma$ whose complement $\Sigma\setminus \mathcal{P}$
    consists of $g$ one-holed tori and $g + p - 2$ pairs of pants,
    such that $\gamma\in \mathcal{P}$, and every curve in $\mathcal{P}\setminus\{\gamma\}$ contained in $\Sigma_1$ separates $\Sigma$ into two subsurfaces of positive genus. See Figure \ref{fig: decomp_for_part_3_2.jpeg}. We call such decompositions \emph{positive-genus decompositions of $\Sigma$ with respect to $\gamma$}, and denote by $\pgd$ the collection of all such decompositions. 
    
    Let $[\rho]\in \NHns\cap X^s_n$ sending $\gamma$ to an elliptic element of infinite order.
    By the definition of $X^s_n$,  $\rho$ maps each decomposition curve to either a hyperbolic element or an elliptic element of infinite order. Then Lemma \ref{lem_MW_posgenus} implies $[\rho]\in \ELns$ unless every decomposition curve contained in $\Sigma_1$ is mapped to a hyperbolic element.

            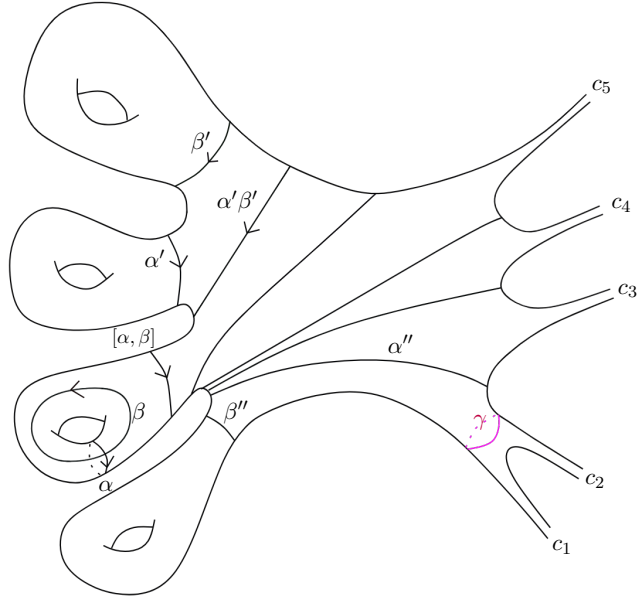
\begin{figure}[h!]
            \centering
            \begin{overpic}[width=0.5\textwidth]{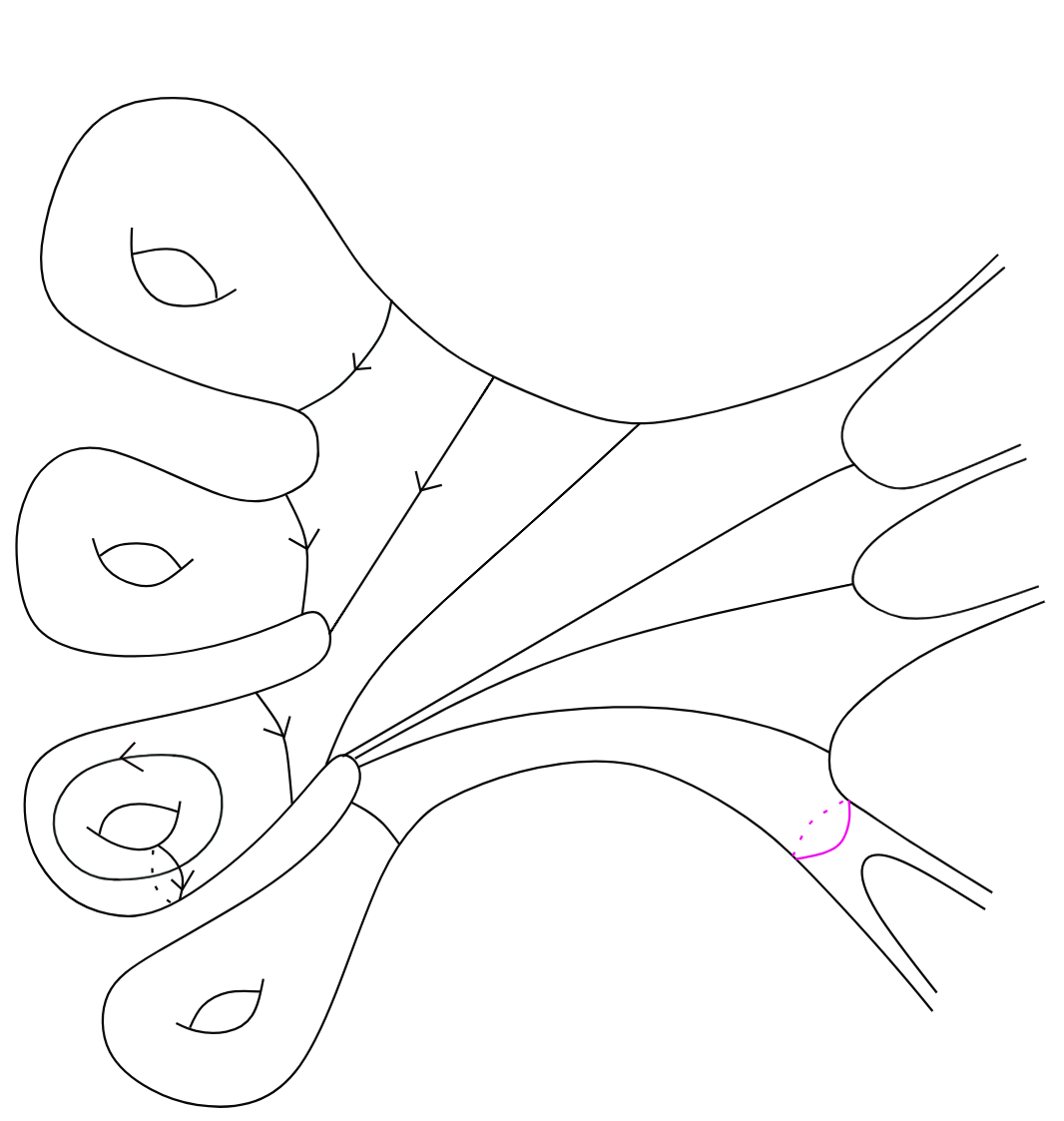}
            \put(90,78){\scalebox{0.8}{$c_5$}}   
            \put(93,60){\scalebox{0.8}{$c_4$}}
            \put(94,47){\scalebox{0.8}{$c_3$}}
            \put(84,8){\scalebox{0.8}{$c_1$}}
            \put(89,18){\scalebox{0.8}{$c_2$}}
            \put(72,27){\scalebox{0.8}{\textcolor{purple}{$\gamma$}}}
            \put(15,17){\scalebox{0.8}{$\alpha$}}
            \put(20,28){\scalebox{0.8}{$\beta$}}
            \put(17,39.5){\scalebox{0.67}{$[\alpha,\beta]$}}
            \put(22,51){\scalebox{0.8}{$\alpha'$}}
            \put(29,69){\scalebox{0.8}{$\beta'$}}
            \put(33,60){\scalebox{0.8}{$\alpha'\beta'$}}
            \put(34,28){\scalebox{0.8}{$\beta''$}}
            \put(59,38){\scalebox{0.8}{$\alpha''$}}
            \end{overpic}
            \caption{
            Positive-genus decomposition of a surface with respect to $\gamma$}\label{fig: decomp_for_part_3_2.jpeg}
            
        \end{figure}

    \begin{lemma}\label{lem_part3_pants}
        Let $\Sigma = \Sigma_{g,p}$ with $g\geqslant 2$ and $p\geqslant 1$. 
        Let 
        $s\in \{\pm 1\}^p$ and $n\in \big\{\chi(\Sigma) + p_+(s) ,\dots, -\chi(\Sigma) - p_-(s)\big\}$, and let $[\rho]\in X^s_n$. 
        Assume that there is a separating simple closed curve $\gamma$ 
        that satisfies the following conditions:
        \begin{enumerate}[(1)]

            \item $\rho(\gamma)$ is elliptic of infinite order.

            \item $\gamma$ separates $\Sigma$ into a surface $\Sigma_1$ with positive genus and a pair of pants $\Sigma_2$, where $\pi_1(\Sigma_2)$ contains peripheral elements $c_1$ and $c_2$ of $\pi_1(\Sigma)$ that map to negative parabolic elements.

        \end{enumerate}
        Assume further that there is a decomposition $\mathcal{P}\in \pgd$ whose complement $\Sigma\setminus \mathcal{P}$ contains 
        \begin{enumerate}[(a)]
            \item 
            a one-holed torus $T$ in $\Sigma_1$ where $e\big(\rho|_{\pi_1(T)}\big)\in \{-1,0\}$,

            \item a pair of pants $P$ in $\Sigma_1$ where $e\big(\rho|_{\pi_1(P)}\big) = -1$ and $s\big(\rho|_{\pi_1(P)}\big) = (0,0,0)$ or $(0,-1,0)$,

            \item a pair of pants $P$ in $\Sigma_1$ where $e\big(\rho|_{\pi_1(P)}\big) = 0$, $s\big(\rho|_{\pi_1(P)}\big) = (0,0,0)$ or $(0,+1,0)$, and $\rho|_{\pi_1(P)}$ is non-abelian, or

            \item a pair of pants $P$ in $\Sigma_1$ where $\gamma$ is represented by one of its boundary components, and $e\big(\rho|_{\pi_1(P\cup_\gamma \Sigma_2)}\big) = -1$.
        \end{enumerate}
        Then there is a non-separating simple closed curve that maps to an elliptic element.
    \end{lemma}
    \begin{proof}
    First, we consider a one-holed torus $T\subset \Sigma_1$ whose boundary component has a hyperbolic image. 
    Let $\alpha, \beta\in \pi_1(\Sigma)$ be two generators of $\pi_1(T)$ such that $\alpha c_1$ and $\beta^{-1} c_1$ are simple.  Note that their commutator $\gamma_0:= \alpha\beta \alpha^{-1}\beta^{-1}$ is represented by the boundary component of $T$; and as we let $x\in \gamma$ as a basepoint, $\gamma_0$ is a simple closed curve in $\pi_1(\Sigma_1, x)$. 
    As $[\rho]\in \N$, $\rho(\alpha)$ and $\rho(\beta^{-1})$ are either elliptic or hyperbolic; and we assume that they are both hyperbolic, as otherwise the lemma follows. We will show that there exists $\delta$ in $\pi_1(T)$ such that $\delta c_1$ is simple in $\pi_1(\Sigma,x)$ and the axis of $\rho(\gamma_0)$ lies on the left of the oriented axis of $\rho(\delta)$. That is, as we conjugate $\rho$ by $\psl$ so that $\rho(\delta) = 
    \pm \begin{bmatrix}
        e^{\lambda} & 0 \\
        0 & e^{-\lambda}
    \end{bmatrix}$ for some $\lambda > 0$, the axis of $\rho(\gamma_0)$ lies entirely on the left half-plane with respect to the $y$-axis.
    Then there is an $M\in \mathbb{Z}$ such that, for all $k\geqslant M$, the axis of $\rho(\gamma_0)^k\rho(\delta)\rho(\gamma_0)^{-k}$ turns positively around the fixed point $x_0$ of $\rho(\gamma)$. 
    Letting $\delta_k:= \gamma_0^k \delta \gamma_0^{-k}$, we find a sequence $\{\delta_k\}_{k\geqslant M}$ of non-separating simple closed curves such that, for each $k$,  
    $\delta_kc_1$ is simple, 
    $\rho(\delta_k)$  is hyperbolic whose axis turns positively around $x_0$, and the distance $d(\delta_k, x_0)\to \infty$ as $k\to \infty$. Then the rest of the proof follows by Lemma \ref{lem_ellnonsep}. \medskip
    
    We first consider the case where $e\big(\rho|_{\pi_1(T)}\big) = -1$. Up to $\psl$-conjugation, we assume that $\rho(\alpha) = 
    \pm \begin{bmatrix}
        e^{\lambda} & 0 \\
        0 & e^{-\lambda}
    \end{bmatrix}$ for some $\lambda > 0$. 
    Letting $\rho(\beta) = \pm \begin{bmatrix}
        a & b \\
        c & d
    \end{bmatrix}$, and letting $A, B\in \mathrm{SL}(2,\mathbb{R})$ respectively be the lifts of $\rho(\alpha)$ and $\rho(\beta)$ in $\mathrm{SL}(2,\mathbb{R})$ of positive traces, we have 
    $$ABA^{-1}B^{-1} 
    = \begin{bmatrix}
        ad - bce^{2\lambda} & ab(e^{2\lambda}-1) \\
        cd(e^{-2\lambda}-1) & ad - bce^{-2\lambda}
    \end{bmatrix}.$$
    Let $x_1, x_2\in \partial \mathbb{H}^2$ be the two endpoints of the axis of $\rho(\gamma_0)$.
    As $x_1$ and $x_2$ are fixed by 
    $\rho(\gamma_0)$, we have the quadratic equation 
    $$\rho(\gamma_0)(x_i)= \dfrac{(ad - bce^{2\lambda})x_i + ab(e^{2\lambda}-1)}{cd(e^{-2\lambda}-1)x_i + ad - bce^{-2\lambda}} = x_i,\ i\in \{1,2\}.$$
    
    Since $e\big(\rho|_{\pi_1(T)}\big) = -1$, it follows from Proposition \ref{prop_holonomy} that
    $\rho|_{\pi_1(T)}$ is a holonomy representation of $\pi_1(T)$ corresponding to a hyperbolic structure on $T$ with geodesic boundary.
    Since $\alpha$ and $\gamma_0$ are disjoint, their geodesic representatives in this structure are at positive distance, and thus the axes of $\rho(\alpha)$ and $\rho(\gamma_0)$ are ultraparallel in $\mathbb{H}^2$.
    Hence, the endpoints $x_1$ and $x_2$ of the axis of $\rho(\gamma_0)$ are real numbers of the same sign. To show that the axis of $\rho(\gamma_0)$ lies on the left half-plane with respect to the $y$-axis, it suffices to show that the sum $x_1 + x_2$ is negative, which will imply that both $x_1$ and $x_2$ are negative.

    From the equation above, we obtain the sign of the sum
    $$sgn(x_1 + x_2) = sgn\bigg(\dfrac{bc(e^{-2\lambda} - e^{2\lambda})}{cd(e^{-2\lambda}-1)}\bigg)= sgn(bd).$$
    Let $\widetilde{A}$ and $\widetilde{B}$ be the lifts of $\rho(\alpha)$ and $\rho(\beta)$, respectively, in $\Hyp_0$. 
    Since $e(\rho|_{\pi_1(T)}) = -1$, we have $\widetilde{A}\widetilde{B}\widetilde{A}^{-1}\widetilde{B}^{-1}\in \Hyp_{-1}$ and $\tr\big(ABA^{-1}B^{-1}\big) = 2ad - bc\big(e^{2\lambda} + e^{-2\lambda}\big) < -2$. Comparing with $2ad - 2bc = 2$, we obtain $$bc>0.$$
    Let
    $A_t:=
    \begin{bmatrix}
        e^{\lambda t} & 0 \\
        0 & e^{-\lambda t}
    \end{bmatrix}$ for $t\in [0,1]$, 
    and let  
    $\{\widetilde{A}_t\}_{t\in [0,1]}$ be the lift of $A_t$ in $\univcover$ connecting $\rm I$ and $\widetilde{A}$. Then the path 
    $\{\widetilde{A}_t\widetilde{B}\widetilde{A}_t^{-1}\widetilde{B}^{-1}\}_{t\in [0,1]}$ connects $\rm I$ and $\widetilde{A}\widetilde{B}\widetilde{A}^{-1}\widetilde{B}^{-1}\in \Hyp_{-1}$, hence must pass $\Ell_{-1}$ at some time $t\in (0,1)$. For such $t$, Lemma \ref{lem_offdiag_Ell} implies that
    $$ab(e^{2\lambda t}- 1)<0\mbox{ and }cd(e^{-2\lambda t}-1)>0,\text{ i.e., }ab<0\mbox{ and }cd<0.$$
    Since $bd = \dfrac{(bc)(cd)}{c^2}<0$, the sum $x_1 + x_2$ is negative, and the axis of $\rho(\gamma_0)$ lies entirely on the left half-plane with respect to the $y$-axis. Therefore, setting $\delta=\alpha$ gives the desired result.
    \smallskip
    
    We now consider the case where $e\big(\rho|_{\pi_1(T)}\big) = 0$. In this case,
    either $\rho|_{\pi_1(T)}$ sends some non-separating simple closed curve to an elliptic element, or the convex core \textbf{core}$\big(\mathbb{H}^2/\rho(\pi_1(T))\big)$ is isometric to a pair of pants with three geodesic boundaries \cite[Theorem 5.2.1]{goldman_torus}. Moreover, there exist simple closed curves $\alpha',\beta',\alpha'\beta'$ in $\pi_1(T)$ such that the axes of $\rho(\alpha'),\rho(\beta')$ and $\rho(\alpha'\beta')$ are boundary components of a convex domain $\mathcal{D}$ corresponding to the pair of pants. Since the lemma directly follows in the former case, we assume the latter. Then either $\alpha'c_1$ is simple or $(\alpha')^{-1}c_1$ is simple, which leads to four cases depending on whether $\beta'c_1$ is simple or $(\beta')^{-1}c_1$ is simple. Denote by $l_{A'},l_{B'},l_{A'B'}$ and $l_{B'A'}$ the oriented axes of $\rho(\alpha'), \rho(\beta'),\rho(\alpha'\beta')$ and $\rho(\beta'\alpha')$, respectively.  The axis of $\rho(\gamma_0)$ is contained entirely in the interior of $\mathcal{D}$, as otherwise the domain cannot be invariant under $\rho(\gamma_0)$-action. See Figure \ref{fig: pants_holonomy}. First, if both $\alpha'c_1$ and $\beta'c_1$ are simple, then their Dehn twists along $\beta'$ and $\alpha'$, respectively, are simple, which implies either $\alpha'\beta'c_1$ or $\beta'\alpha'c_1$ is simple. Without loss of generality, we assume $\alpha'\beta'c_1$ is simple. Since the axis of $\rho(\gamma_0)$ lies on the left of either $l_{A'}$ or $l_{A'B'}$, setting $\delta$ as $\alpha'$ or $\alpha'\beta'$ accordingly completes the proof. Secondly, if $\alpha'c_1$ and $(\beta')^{-1}c_1$ are simple, then axis of $\rho(\gamma_0)$ lies on the left of $l_{A'}$ or right of $l_{B'}$. Then setting $\delta$ as $\alpha'$ or $(\beta')^{-1}$ accordingly completes the proof. The cases where $(\alpha')^{-1}c_1$ and $\beta'c_1$ are simple and where $(\alpha')^{-1}c_1$ and $(\beta')^{-1}c_1$ are simple follow verbatim. 

    \begin{figure}[h!]
    \centering
    \begin{overpic}[width=0.4\textwidth]{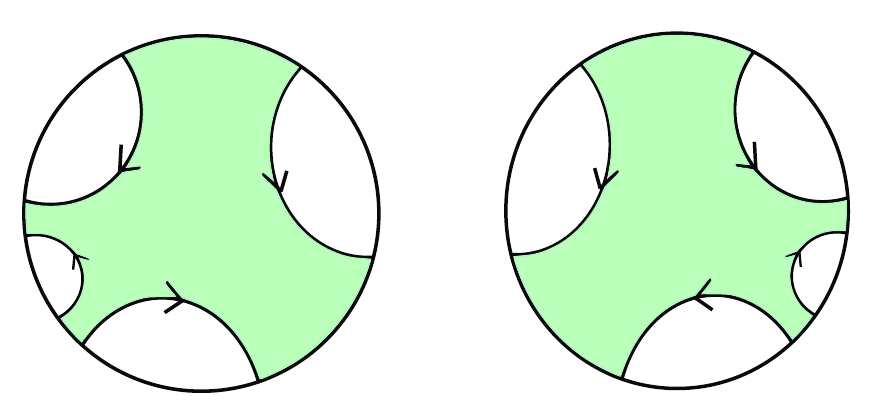}
    \put(18,7){\scalebox{0.8}{$l_{A'}$}}
    \put(8,28){\scalebox{0.8}{$l_{B'}$}}
    \put(33,26){\scalebox{0.8}{$l_{A'B'}$}}
    \put(11,17){\scalebox{0.8}{$l_{B'A'}$}}

    \put(78,7.5){\scalebox{0.8}{$l_{A'}$}}
    \put(86,29){\scalebox{0.8}{$l_{B'}$}}
    \put(59,26){\scalebox{0.8}{$l_{A'B'}$}}
    \put(80,17){\scalebox{0.8}{$l_{B'A'}$}}
    \end{overpic}
    \caption{\label{fig: pants_holonomy} Two configurations of $\mathcal{D}$}
    \end{figure}
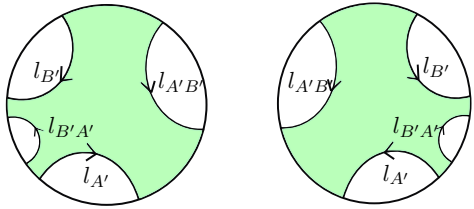
   
    \medskip

    Next, we consider a pair of pants $P\subset \Sigma_1$.
    Let $\alpha', \beta'\in \fund$ be two generators of $\pi_1(P)$, represented by two boundary components of $P$, 
    such that $\alpha'\beta'$, $\alpha' c_1$ and $\beta' c_1$ are all simple. See Figure \ref{fig: decomp_for_part_3_2.jpeg}. We further assume that all $\rho(\alpha')$, $\rho(\beta')$ and $\rho(\alpha'\beta')$ are hyperbolic, unless $\beta'$ is a peripheral element of $\pi_1(\Sigma)$---in which case $\rho(\beta')$ is parabolic.
    We will show that, letting $\gamma_0: = \alpha'\beta'$ and after possibly exchanging $\alpha'$ and $\beta'$, the axis of $\rho(\gamma_0)$ lies on the left-hand side of the axis of $\rho(\alpha')$. 
    Then, as discussed in the beginning of the proof, letting $\alpha_k:= \gamma_0^k \alpha' \gamma_0^{-k}$, we find a sequence $\{\alpha_k\}_{k\geqslant 1}$ of separating simple closed curves 
    whose complement has components of positive genus.
    Moreover, for each $k$,  
    $\alpha_kc_1$ is simple, 
    $\rho(\alpha_k)$  is hyperbolic whose axis turns positively around $x_0$, and the distance $d(\alpha_k, x_0)\to \infty$ as $k\to \infty$. Then the rest of the proof follows by Lemma \ref{lem_ellnonsep}.
    \smallskip
    
    We first consider the case where
    $e\big(\rho|_{\pi_1(P)}\big) = -1$ and $\rho(\beta')$ is either hyperbolic or negative parabolic. Up to $\psl$-conjugation, we assume that $\rho(\alpha') = 
    \pm \begin{bmatrix}
        e^{\lambda} & 0 \\
        0 & e^{-\lambda}
    \end{bmatrix}$ for some $\lambda > 0$. Letting $\rho(\beta') = \pm \begin{bmatrix}
        a & b \\
        c & d
    \end{bmatrix}$
    and letting $A, B\in \mathrm{SL}(2,\mathbb{R})$ respectively be the lifts of $\rho(\alpha')$ and $\rho(\beta')$ in $\mathrm{SL}(2,\mathbb{R})$ of positive traces, we have 
    $$AB 
    = \begin{bmatrix}
        ae^\lambda & be^\lambda \\
        ce^{-\lambda} & de^{-\lambda}
    \end{bmatrix}.$$
    Then we have the quadratic equation 
    $$\rho(\gamma_0)(x_i)= \dfrac{ae^\lambda x_i + be^\lambda}{ce^{-\lambda} x_i + de^{-\lambda}} = x_i,\ i\in \{1,2\}.$$
    
    Since $e\big(\rho|_{\pi_1(P)}\big) = -1$, by Proposition \ref{prop_holonomy},
    $\rho|_{\pi_1(P)}$ is a holonomy representation of $\pi_1(P)$. 
    As $\alpha'$ and $\gamma_0$ are represented by two distinct boundary components of $P$, the axes of $\rho(\alpha')$ and $\rho(\gamma_0)$ are ultraparallel in $\mathbb{H}^2$.
    This implies that $x_1$ and $x_2$ are real numbers of the same sign. To show that $\rho(\gamma_0)$ lies on the left half-plane with respect to the $y$-axis, we show that both $x_1$ and $x_2$ are negative.
    From the equation above, we obtain the sum
    $$x_1 + x_2 = \dfrac{ae^{\lambda} - de^{-\lambda}}{ce^{-\lambda}}.$$
    Let $\widetilde{A}$ and $\widetilde{B}$ be the lifts of $\rho(\alpha')$ and $\rho(\beta')$, respectively, in $\Hyp_0\cup \Parm_0$.
    Since $e(\rho|_{\pi_1(P)}) = -1$, we have $\widetilde{A}\widetilde{B}\in \Hyp_{-1}$ and $\tr\big(AB\big) = ae^{\lambda} + de^{-\lambda} < -2$. If $ae^{\lambda} - de^{-\lambda} > 0$, then subtracting from
    $ae^{\lambda} + de^{-\lambda}$, we obtain $2de^{-\lambda} < 0\text{ i.e., } d < 0.$ Moreover, $a + d = \tr(B) \geqslant 2$ implies $a \geqslant |d| + 2$.
    Then we get 
    $$2 > ae^{\lambda} + de^{-\lambda} \geqslant (|d| + 2)e^{\lambda} + de^{-\lambda}
    = 2e^{\lambda} + |d|(e^{\lambda} -  e^{-\lambda}) > 2,$$
    which is a contradiction. Therefore, we obtain $$ae^{\lambda} - de^{-\lambda} < 0.$$
    Let
    $A_t : = \begin{bmatrix}
        e^{\lambda t} & 0 \\
        0 & e^{-\lambda t}
    \end{bmatrix}$ for $t\in [0,1]$, and let  
    $\{\widetilde{A}_t\}_{t\in [0,1]}$ be the lift of $A_t$ in $\univcover$ connecting $\rm I$ and $\widetilde{A}$.
    Then the path $\big\{\widetilde{A}_t\widetilde{B}\big\}_{t\in [0,1]} 
    $ connects $\widetilde{B}\in \Hyp_0\cup \Parm_0$ and 
    $\widetilde{A}\widetilde{B}\in \Hyp_{-1}$, hence passes $\Ell_{-1}$ at some time $t\in (0,1)$. For such $t$, Lemma \ref{lem_offdiag_Ell} implies that $$ be^{\lambda t} < 0 \text{ and } ce^{-\lambda t} > 0 \text{, i.e., }b < 0\text{ and }c > 0.$$ Therefore, $ae^{\lambda} - de^{-\lambda} < 0$ and $c>0$ together imply $x_1 + x_2 = \dfrac{ae^{\lambda} - de^{-\lambda}}{ce^{-\lambda}} < 0$, and hence both $x_1$ and $x_2$ are negative, as desired.
    \smallskip

    We now consider the case where
    $e\big(\rho|_{\pi_1(P)}\big) = 0$ and $\rho(\beta')$ is either hyperbolic or positive parabolic.  We first assume that $\rho(\beta')$ is hyperbolic, and $\rho|_{\pi_1(P)}$ is non-abelian.
    If the axis of $\rho(\alpha')$ or $\rho(\beta')$, say $\rho(\alpha')$, intersects the axis of $\rho(\gamma_0)$, then we can find an $M\in \mathbb{Z}$ such that, for each $k\geqslant M$,  the axis of $\rho(\gamma_0)^k\rho(\alpha')\rho(\gamma_0)^{-k}$ turns positively around the fixed point $x_0$ of $\rho(\gamma)$, with its distance from $x_0$ tending to $+\infty$ as $k\to \infty$.
    Then the rest of the proof follows by Lemma \ref{lem_ellnonsep}.
    Assume neither the axis of $\rho(\alpha')$ nor $\rho(\beta')$ intersects the axis of $\rho(\gamma_0)$. Since $\tr[\rho(\alpha'),\rho(\beta')]=\tr[\rho(\alpha'),\rho(\alpha'\beta')]=\tr[\rho(\beta'),\rho(\alpha'\beta')]$, by Lemma \ref{lem_torus}, the axis of $\rho(\alpha')$ and $\rho(\beta')$ do not intersect.
    By Proposition \ref{prop_holonomy}, $\rho|_{\pi_1(P)}$ is not a holonomy representation of $\pi_1(P)$, i.e.,
    the oriented axes $l_A,\ l_B$ and $l_{AB}$ of $\rho(\alpha'),\ \rho(\beta'),\ $and $\rho(\gamma_0)$, respectively, do not occur as the three boundary axes of the convex core of a hyperbolic pair of pants as in Figure \ref{fig: pants_holonomy}. Then up to a conjugation of $\rho$ by an orientation-reversing isometry on $\mathbb{H}^2$,
    the axes $l_A,\ l_B$ and $l_{AB}$ lie in $\mathbb{H}^2$ as in Figure \ref{fig: pants_hyp_euler_0.png}. In any of the cases, $l_{AB}$ lies on the left-hand side of either 
    $l_A$ or $l_B$, as desired.  
        \begin{figure}[H]
        \centering
        \begin{overpic}[width = 0.8\textwidth, trim = 0 180 0 30]{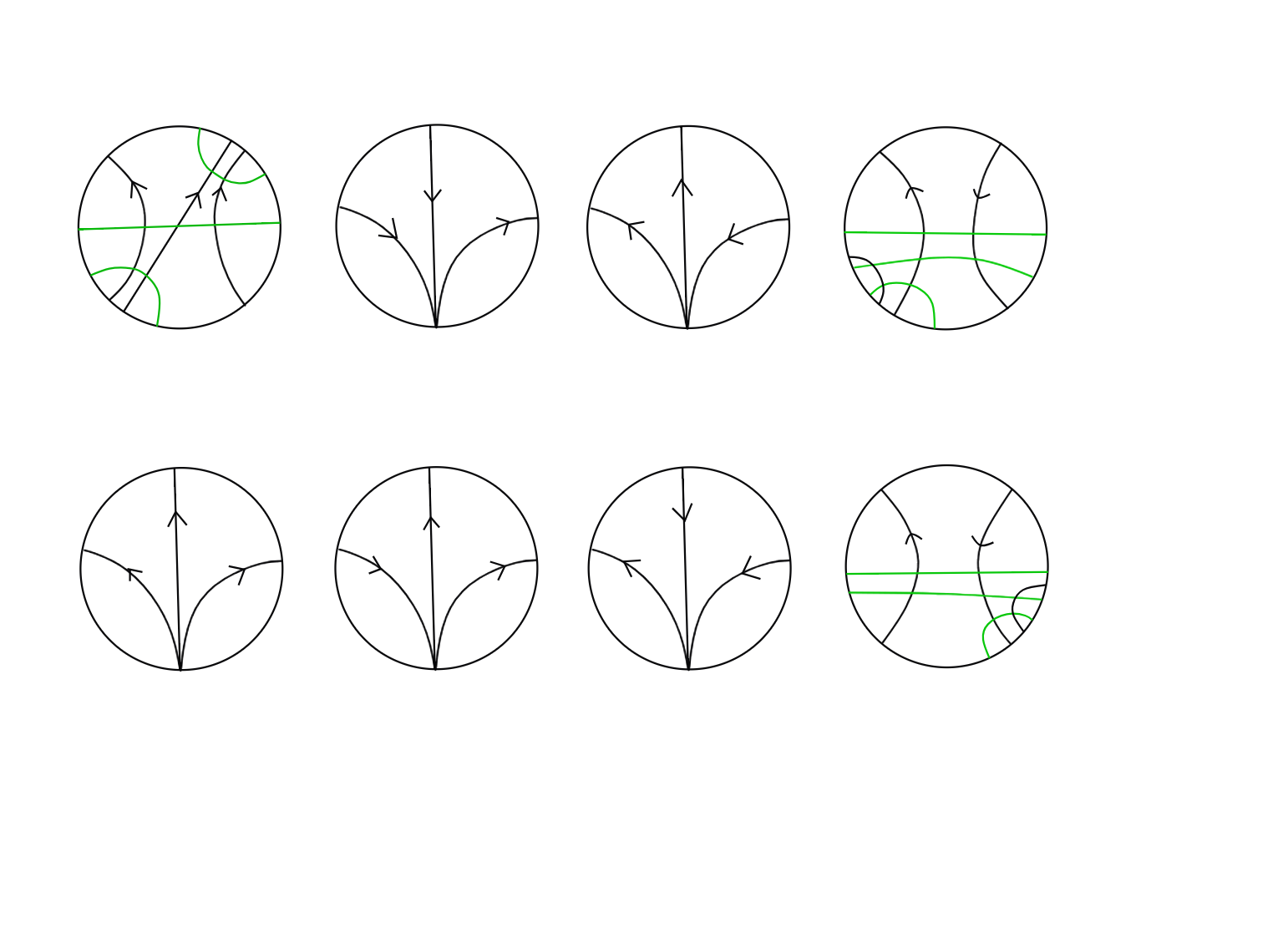}
        \put(9,44.5){\scalebox{0.6}{$l_B$}}
        \put(17.5,44.5){\scalebox{0.6}{$l_A$}}
        \put(4.5,43){\scalebox{0.6}{\textcolor{green!45!black}{$s$}}}
        \put(4,39.5){\scalebox{0.6}{\textcolor{green!45!black}{$s_B$}}}
        \put(12.5,47){\scalebox{0.6}{$l_{AB}$}}
        \put(15,52){\scalebox{0.6}{\textcolor{green!45!black}{$s_A$}}}
        \put(37,41){\scalebox{0.6}{$l_A$}}
        \put(34,47){\scalebox{0.6}{$l_B$}}
        \put(28,41){\scalebox{0.6}{$l_{AB}$}}
        \put(54,47){\scalebox{0.6}{$l_B$}}
        \put(57,40.5){\scalebox{0.6}{$l_A$}}
        \put(47,41){\scalebox{0.6}{$l_{AB}$}}
        \put(8,15){\scalebox{0.6}{$l_B$}}
        \put(10.5,18){\scalebox{0.6}{$l_{AB}$}}
        \put(18,14){\scalebox{0.6}{$l_A$}}
        \put(28,15){\scalebox{0.6}{$l_B$}}
        \put(34,18){\scalebox{0.6}{$l_A$}}
        \put(37,14){\scalebox{0.6}{$l_{AB}$}}
        \put(49,14){\scalebox{0.6}{$l_B$}}
        \put(54,17){\scalebox{0.6}{$l_A$}}
        \put(57,14){\scalebox{0.6}{$l_{AB}$}}
        \put(68.5,45){\scalebox{0.6}{$l_B$}}
        \put(77,45){\scalebox{0.6}{$l_A$}}
        \put(82,42.5){\scalebox{0.6}{\textcolor{green!45!black}{$s$}}}
        \put(81,38){\scalebox{0.6}{\textcolor{green!45!black}{$s_A$}}}
        \put(71,34){\scalebox{0.6}{\textcolor{green!45!black}{$s_B$}}}
        \put(62.5,40){\scalebox{0.6}{$l_{AB}$}}
        \put(77.5,19){\scalebox{0.6}{$l_A$}}
        \put(68,19){\scalebox{0.6}{$l_B$}}
        \put(64,17){\scalebox{0.6}{\textcolor{green!45!black}{$s$}}}
        \put(63.5,14){\scalebox{0.6}{\textcolor{green!45!black}{$s_B$}}}
        \put(76,8){\scalebox{0.6}{\textcolor{green!45!black}{$s_A$}}}
        \put(81.5,15){\scalebox{0.6}{$l_{AB}$}}
        \end{overpic}
        \caption{Configurations of $l_A,\ l_B$ and $l_{AB}$.
        The geodesics $s_A,\ s_B, s$ are the axes of reflections $R_{s_A},\ R_{s_B}, R_{s}$, respectively, where $\rho(\alpha') = R_{s_A}R_{s}$ and $\rho(\beta') = R_{s}R_{s_B}$.}
        \label{fig: pants_hyp_euler_0.png}
    \end{figure}

              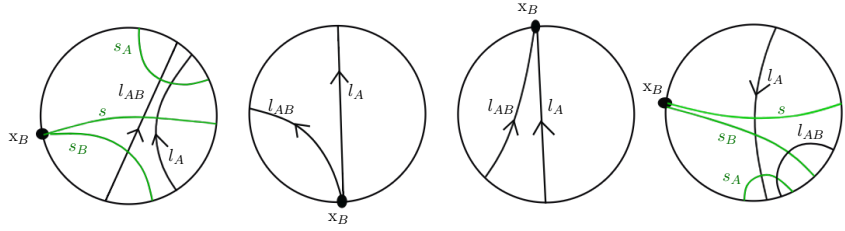
\begin{figure}[H]
        \centering
        \begin{overpic}[width=0.8\textwidth, trim = 70 0 0 0]{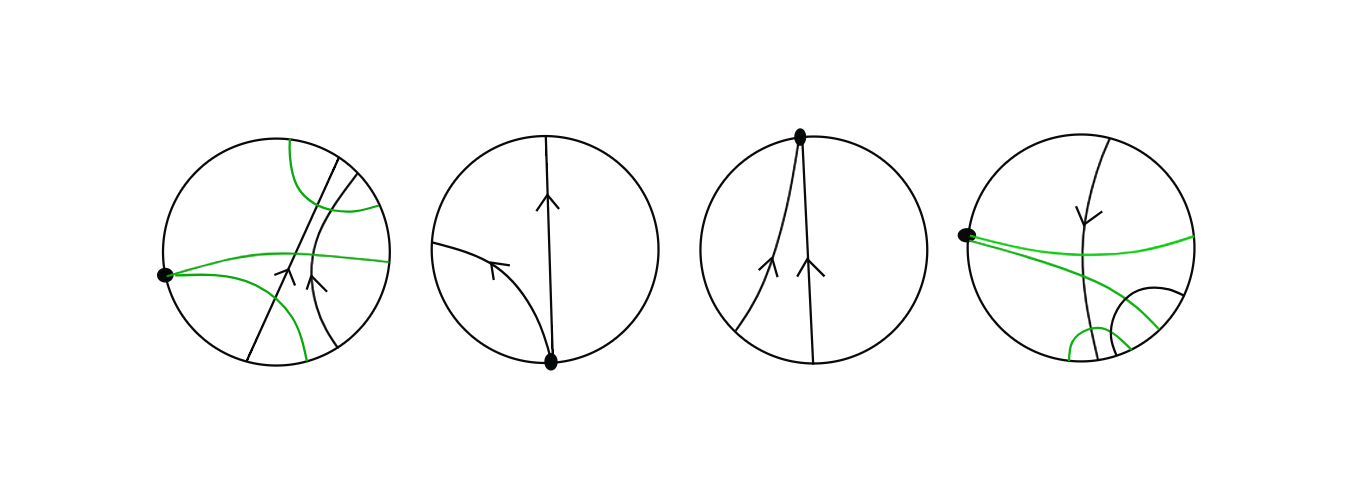}
        \put(4,17){\scalebox{0.6}{$\text{x}_B$}}
        \put(20,15){\scalebox{0.6}{$l_A$}}
        \put(13,19.5){\scalebox{0.6}{\textcolor{green!45!black}{$s$}}}
        \put(10,16){\scalebox{0.6}{\textcolor{green!45!black}{$s_B$}}}
        \put(14.5,26){\scalebox{0.6}{\textcolor{green!45!black}{$s_A$}}}
        \put(15,21.5){\scalebox{0.6}{$l_{AB}$}}
        \put(36,9){\scalebox{0.6}{$\text{x}_B$}}
        \put(38,21.5){\scalebox{0.6}{$l_A$}}
        \put(30,20){\scalebox{0.6}{$l_{AB}$}}
        \put(55,30){\scalebox{0.6}{$\text{x}_B$}}
        \put(58,20){\scalebox{0.6}{$l_A$}}
        \put(52,20){\scalebox{0.6}{$l_{AB}$}}
        \put(80,23){\scalebox{0.6}{$l_A$}}
        \put(81,19.7){\scalebox{0.6}{\textcolor{green!45!black}{$s$}}}
        \put(75,17){\scalebox{0.6}{\textcolor{green!45!black}{$s_B$}}}
        \put(75.5,13){\scalebox{0.6}{\textcolor{green!45!black}{$s_A$}}}
        \put(83,17.5){\scalebox{0.6}{$l_{AB}$}}
        \put(67.5,22){\scalebox{0.6}{$\text{x}_B$}}
        \end{overpic}
        \caption{Configurations of $l_A,\ l_{AB}$ and $\text{x}_B$.}
        \label{fig: pants_par_euler_0.png}
    \end{figure}

     If $\rho(\beta')$ is positive parabolic, then again by Proposition \ref{prop_holonomy}, the axes $l_A$ and $l_{AB}$ of $\rho(\alpha')$ and $\rho(\gamma_0)$, together with the fixed point $\text{x}_B$ of $\rho(\beta')$, fail to form the degenerate right-angled hexagon configuration associated to a one-cusped hyperbolic pair of pants.
     In this case, the axes $l_A$ and $l_{AB}$ lie in $\mathbb{H}^2$ as in Figure \ref{fig: pants_par_euler_0.png}, where $l_{AB}$ lies on the left-hand side of $l_A$, as desired.
     \smallskip

     Finally, we address the case where one of the boundary components of
     $P$ represents $\gamma$. 
    Let $\alpha'', \beta''\in \fund$ be two generators of $\pi_1(P)$, represented by two boundary components of $P$, 
    such that $\alpha''\beta'' = \gamma$, and $\alpha'' c_1$ and $\beta'' c_1$ are all simple. See Figure \ref{fig: decomp_for_part_3_2.jpeg}.
    In this case, we assume that both $\rho(\alpha'')$ and $\rho(\beta'')$ are hyperbolic, and $e\big(\rho|_{\pi_1(P\cup_\gamma \Sigma_2)}\big) = -1$. 
    Then by Lemma \ref{lem_fourpuncsphere}, for some $k\in \mathbb{Z}$, the simple closed curve $\alpha''\gamma^k c_1\gamma^{-k}$ maps to an elliptic element. Notice that the curve $\alpha''\gamma^k c_1\gamma^{-k}$ separates the four-holed sphere $P\cup_\gamma \Sigma_2$ into two pairs of pants, where one has $c_1$ and $\alpha''$ as its boundary components, and another has $c_2$ and $\beta''$ as its boundary components. Since $\alpha''$ and $\beta''$ are the decomposition curves in the positive-genus decomposition $\mathcal{P}$, the curve $\alpha''\gamma^k c_1\gamma^{-k}$ also separates $\Sigma$ into two subsurfaces of positive genus. Then the proof follows by Lemma \ref{lem_MW_posgenus}.
    This completes the proof of Lemma \ref{lem_part3_pants}.
\end{proof}

     Finally, we prove Proposition \ref{prop_ELnNHn}.

       \begin{proof}[Proof of Proposition \ref{prop_ELnNHn}]

        Since $\ELns\cap X^s_n \subset \NHns\cap X^s_n$ by definition, we show that $\NHns\cap X^s_n\subset \ELns\cap X^s_n$.
        Let $[\rho]\in \NHns\cap X^s_n$, assuming that $\rho$ maps a separating simple closed curve $\gamma$ to an elliptic element of infinite order. Let $\Sigma_1$ and $\Sigma_2$ be two subsurfaces of $\Sigma$ separated by $\gamma$. Assuming each condition in the theorem, we will show that $\rho$ also maps a non-separating simple closed curve to an elliptic element, concluding that $[\rho]\in \ELns\cap X^s_n$.
        \smallskip

        If $p = 1$, then both $\Sigma_1$ and $\Sigma_2$ have positive genus, hence the result follows from Lemma \ref{lem_MW_posgenus}.
        Similarly, if $p = 2$ and $s = (+1, -1)$ or $(-1, +1)$, then both $\Sigma_1$ and $\Sigma_2$ have positive genus. Indeed, if $\Sigma_2$ were a pair of pants whose fundamental group $\pi_1(\Sigma_2)$ contains both peripheral elements $c_1$ and $c_2$ of $\pi_1(\Sigma)$, then Lemma \ref{parplp} would imply that $\rho(\gamma)=\rho(c_1c_2)$ is not an elliptic element. Therefore, Lemma \ref{lem_MW_posgenus} again completes the proof. Otherwise, 
        we can assume that one of $\Sigma_1$ and $\Sigma_2$, say $\Sigma_2$, has genus zero, as otherwise the theorem directly follows from Lemma \ref{lem_MW_posgenus}. 
        Moreover, if there is a simple closed curve in $\Sigma_2$ whose image is hyperbolic, then the result directly follows from Lemma \ref{lem_MW_posgenus}. Hence we assume that $\rho$ maps all non-peripheral simple closed curves in $\Sigma_2$ to elliptic elements. 
        Then by replacing $\gamma$ with another simple closed curve in $\Sigma_2$, we can assume that $\Sigma_2$ is a pair of pants where $\pi_1(\Sigma_2)$ contains two preferred peripheral elements $c_1$ and $c_2$ of $\pi_1(\Sigma)$, and $\gamma = c_1c_2$. Since
        $\rho(\gamma)$ is elliptic, by Lemma \ref{parplp}, both peripheral elements map to negative parabolic elements, or both map to positive parabolic elements. 
        It suffices to show for the case where both $\rho(c_1)$ and $\rho(c_2)$ are negative parabolic. Indeed, if both $\rho(c_1)$ and $\rho(c_2)$ are positive  parabolic, the conjugation $g\rho g^{-1}$ by an orientation-reversing isometry $g\in \pgl\setminus \psl$ of $\mathbb{H}^2$ maps both peripheral elements to negative parabolic elements. By Proposition \ref{prop_pglpsl}, $g \rho g^{-1}$ has sign $-s$ and relative Euler class $-n\in \{\chi(\Sigma) +p_+(-s) + 1,\dots, -\chi(\Sigma) - p_-(-s) - 1\}$. 
        Since $[\rho]\in \NHns\cap X^s_n$ and $\rho(\gamma)$ is elliptic, we have $[g\rho g^{-1}]\in \mathcal{NH}^{-s}_{-n}\cap X^{-s}_{-n}$ and $g\rho(\gamma)g^{-1}$ is elliptic. Hence $[g\rho g^{-1}]\in \mathcal{EL}^{-s}_{-n}\cap X^{-s}_{-n}$ follows from the case where $c_1$ and $c_2$ map to negative parabolic elements, which also implies $[\rho]\in \mathcal{EL}^{s}_{n}\cap X^{s}_{n}$.

        Consider a pants decomposition $\mathcal{P}\in \mathrm{PGD}_\gamma$ of $\Sigma$ whose complement $\Sigma\setminus \mathcal{P}$ consists of $g$ number of one-holed tori and $g + p - 2$ number of pairs of pants, as described in 
        Figure \ref{fig: decomp_for_part_3_2.jpeg}. 
        Recall that $[\rho]\in \N$ maps each non-peripheral simple closed curve to either a hyperbolic element or an elliptic element of infinite order.
        By Lemma \ref{lem_MW_posgenus}, we can assume that every curve in $\mathcal{P}\setminus\{\gamma\}$ contained in $\Sigma_1$ has hyperbolic image. Hence the relative Euler class is well-defined for the restriction of $\rho$ to 
        $\pi_1(\Sigma')$, where $\Sigma'\subset \Sigma$ is
        a union of two pairs of pants
        glued along $\gamma$, as well as for the restriction of $\rho$ to
        each torus or pair of pants contained in $\Sigma\setminus \Sigma'$ that is a connected component of $\Sigma_1\setminus \mathcal{P}$. Since $\rho(\gamma)$ is elliptic, 
        $\rho|_{\pi_1(\Sigma')}$ is not a holonomy representation, and Proposition \ref{prop_holonomy} implies $e\big(\rho|_{\pi_1(\Sigma')}\big)\geqslant -1$. 
        Since $[\rho]\in X^s_n$, for each pair of pants $P$ in $\Sigma\setminus \mathcal{P}$, $\rho|_{\pi_1(P)}$ is non-abelian. Thus,
        $\mathcal{P}$ satisfies Condition \textit{(a)}, \textit{(b)}, \textit{(c)} or \textit{(d)} in Lemma \ref{lem_part3_pants},
        as otherwise Proposition \ref{prop_additivity} implies that $n\geqslant  -\chi(\Sigma) - p_-(s)$. Finally, Lemma \ref{lem_part3_pants} completes the proof.
     \end{proof}

\appendix
\section{An alternate proof of Ryu's result}
\begin{theorem}\label{thm_tothyp}
Let $\Sigma = \Sigma_{g,p}$ be a punctured surface with $g\geqslant 2,p\geqslant 1$. Let $s\in\{\pm 1\}^p$ and $n\in \mathbb{Z}$ satisfy:
\begin{enumerate}[(1)]
    \item $p_+(s)=1$ and $n=\chi(\Sigma)+1$, or
    \item $p_-(s)=1$ and $n=-\chi(\Sigma)-1$
\end{enumerate}
Then $\mathcal{NH}^s_n$ has measure zero.
\end{theorem}
\begin{proof}
    By Proposition \ref{prop_pglpsl}, it suffices 
    to show for the case where $p_+(s)=1$ and $n=\chi(\Sigma)+1$. 
     Let $[\rho]\in X^s_n$. 
     By the definition of $X^s_n$, either $[\rho]\in \mathcal{EL}^s_n$, or 
     $\rho$ sends every non-separating simple closed curve to a hyperbolic element.
     Then by Proposition \ref{connect_el}, $\rho$ sends every non-separating simple closed curve to a hyperbolic element. 
     
     Assume $\rho$ sends a separating simple closed curve $\alpha$ to an elliptic element. Since $[\rho]\in X^s_n$, $\rho(\alpha)$ is elliptic of infinite order. We claim that $\rho$ also sends some non-separating simple closed curve to an elliptic element, which gives a contradiction. 
     Let $\Sigma_1$ and $\Sigma_2$ be subsurfaces separated by $\alpha$. 
     If both $\Sigma_1$ and $\Sigma_2$ have positive genus, then the claim directly follows from Lemma \ref{lem_MW_posgenus}. If $\Sigma_2$ is a punctured sphere, we consider a positive-genus decomposition $\mathcal{P}$ of $\Sigma$ with respect to $\alpha$. Let $T\subset \Sigma_1$ be a one-holed torus, whose boundary component $\gamma$ is contained in $\mathcal{P}\setminus \{\alpha\}$. If $\rho(\gamma)$ is elliptic, then again Lemma \ref{lem_MW_posgenus} proves the claim. If $\rho(\gamma)$ is hyperbolic, then the relative Euler class of the restriction $\rho|_{\pi_1(T)}$ is well-defined. By Proposition \ref{prop_additivity} and Theorem \ref{ryuhp}, we have $e\big(\rho|_{\pi_1(T)}\big) = -1$. Then the claim follows by Lemma \ref{lem_part3_pants}. This completes the proof.
\end{proof}
\section{Proofs of Lemmas \ref{parplp}, \ref{edgecon1}, \ref{edgecon2} and \ref{edgecon3}}
\begin{proof}[Proof of Lemma \ref{parplp}]
    For the images of the lifted product map, see the proof of \cite[Proposition 3.4]{ryu}. In the rest of the proof, for $\pm P\in \Par$, we denote by $P$ the lift of $\pm P$ in $\SL$ with trace 2; and denote by $\widetilde{P}$ the lift of $\pm P$ in $\Par_0$.\smallskip
    
    We first show that, for $\pm P_1,\pm P_2\in \Par$, 
    $\ev(\pm{P_1},\pm{P_2})\in \Par_0\cup \{\mathrm{I}\}$ if and only if 
    $\pm{P_1}$ and $\pm{P_2}$ commute.
    Let $\ev(\pm P_1,\pm P_2)\in \Par_0\cup \{\mathrm{I}\}$.
    Then we have $\kappa(\chi(P_1,P_2))= \kappa(2,2,2)=2$. By Proposition \ref{prop_char}, $P_1$ and $P_2$ preserve the same eigenvector in $\mathbb{R}^2$. Hence, up to an $\SL$-conjugation, $P_1$ and $P_2$ are of the form $P_1 = \begin{bmatrix}
        1 & t_1\\
        0 & 1
    \end{bmatrix}, P_2 = \begin{bmatrix}
        1 & t_2\\
        0 & 1
    \end{bmatrix}, t_1, t_2\in \mathbb{R}\setminus \{0\},$ that clearly commute. Conversely, if $(\pm P_1, \pm P_2)$ commute, they lie in the common one-parameter subgroup of $\psl$ generated by $\pm P_1$, 
    which lifts to a one-parameter subgroup of $\univcover$ containing $\ev(\pm P_1, \pm P_2)$.
    Thus $\ev(\pm P_1,\pm P_2)\in \Par_0\cup \mathrm{I}$.\smallskip
    
    We now show the path-lifting property of the lifted product map on the set $S_{s_1,s_2} = \{(\pm P_1,\pm P_2) \in \Par^{s_1}\times \Par^{s_2}| \pm P_1P_2 \neq \pm P_2P_1\}$. We show this for the case when $s_1 = s_2 = -$, as the proof follows similarly for the other cases. Consider $\ev: S_{-,-} \rightarrow \Ell_{-1}\cup \Par_{-1}^+ \cup \Hyp_{-1}$. 
    By Lemma \ref{lem_plp}  and Lemma \ref{lem_evPsubm}, it suffices to show that for every element $\widetilde{C}\in \Ell_{-1}\cup \Par^+_{-1}\cup \Hyp_{-1}$, the fiber of $\widetilde{C}$ is path-connected. Let $(\pm P_1,\pm P_2),(\pm P_1',\pm P_2')\in \ev^{-1}(\widetilde{C})$.
    Then $\chi(P_1,P_2) = \chi(P_1',P_2') = (2,2,z)$, where $z = Tr(\widetilde{C})$; and since $z\neq 2$, we have $\kappa(2,2,z) \neq 2$. Thus,  by Proposition \ref{prop_char}, $(P_1,P_2)=(gP_1'g^{-1},gP_2'g^{-1})$ for some $g \in \GL$. 
    If $g\in \GL\setminus \SL$, then 
    letting $P_{i,j}$ denote the $(i,j)$-entry of $P$, $(P_1')_{12}-(P_1')_{21}$ and $ (P_1)_{12} - (P_1)_{21}$ have opposite signs. However, since $\widetilde{P_1},\widetilde{P_1'}\in \Par_0^-$,
    Lemma \ref{lem_offdiag} implies $(P_1)_{12}-(P_1)_{21} < 0$ and $(P_1')_{12} - (P_1')_{21}<0$. Thus, $g \in \SL$. Since 
    $(\pm P_1,\pm P_2) = (\pm g P_1'g^{-1},\pm gP_2'g^{-1})$, we have $ \pm C = \pm P_1  P_2 =\pm gP_1'P_2'g^{-1} = \pm gCg^{-1}$, i.e., $\pm g$ commutes with $\pm C$. 
    Let $\{\pm g_t\}\interval$ denote a path in $\psl$ from $\pm \mathrm{I}$ to $\pm g$ which is in the one-parameter subgroup generated by $\pm g$. Then for all $t\in [0,1]$, $\pm g_t$ commutes with $\pm C$. 
     Let $\{\widetilde{g_t}\}\interval \subset \univcover$ denote the lift of $\{g_t\}\interval$ starting at $\mathrm{I}$. For all $t\in [0,1]$, $\ev(\pm g_tP_1'g_t^{-1},\pm g_tP_2'g_t^{-1})=\widetilde{g_t}\widetilde{P_1'}\widetilde{P_2'}\widetilde{g_t}^{-1}$. Since $\pm g_tP_1'P_2'g_t^{-1} = \pm C$ for all $t\in [0,1]$, $\{\widetilde{g_t}\widetilde{P_1'}\widetilde{P_2'}\widetilde{g_t}^{-1}\}\interval$ is a lift of the constant path $\pm C$ starting at $\widetilde{C}$, i.e., $\{\widetilde{g_t}\widetilde{P_1'}\widetilde{P_2'}\widetilde{g_t}^{-1}\}\interval = \{\widetilde{C}\}$. Therefore, the path $\big\{(\pm g_tP_1'g_t^{-1},\pm g_tP_2'g_t^{-1})\big\}\interval$ connects $(P_1',P_2')$ to $(P_1,P_2)$ within $\ev^
     {-1}(\widetilde{C})$, as desired. 
    \end{proof}

\begin{proof}[Proof of Lemma \ref{edgecon1}]
To prove the lemma for $P_1$, we first show that 
any point $(\widetilde{A}, \widetilde{B})\in P_1^{-1}(\widetilde{C})$ with $\kappa\big(\chi(\widetilde{A}, \widetilde{B})\big) = 2$ can be connected to a point $(\widetilde{A}', \widetilde{B}')\in P_1^{-1}(\widetilde{C})$ with $\kappa\big(\chi(\widetilde{A}', \widetilde{B}')\big) \neq 2$ by a path in $P_1^{-1}(\widetilde{C})$.
Let $\chi(\widetilde{A},\widetilde{B})=(x,2,z)$ where $|x|>2,|z|>2$. Since $\kappa(x,2,z) = 2$, we have $x=z$. Choose $\epsilon>0$ such that the path $\{z+\epsilon t\}\interval$ lies in either $(-\infty,-2)$ or $(2,\infty)$. 
Let $\big\{\chi_t\big\}\interval:=\big\{(z+\epsilon t,2,z)\big\}\interval$. Let 
$(A,B)$ be the projection of $(\widetilde{A},\widetilde{B})$ to $\SL\times \SL$. 
Since $|z|>2$ and
$[A,B]\neq \mathrm{I}$, by Lemma \ref{lem_plpchar}, we may lift $\{\chi_t\}\interval$ to a path $\big\{(A^t,B^t)\big\}\interval$ of non-commuting $\SL$-pairs starting at $(A,B)$. Let $\big\{(\widetilde{A^t},\widetilde{B^t})\big\}\interval$ be the lift of $\big\{(A^t,B^t)\big\}\interval$ in $\univcover \times \univcover$ starting at $(\widetilde{A},\widetilde{B})$. 
The path $\{\widetilde{A^t}\}\interval$ lies in $\Hyp_m$, as it is a path of hyperbolic elements starting at $\widetilde{A^0} = \widetilde{A}\in \Hyp_m$; and since $Tr(\widetilde{B^t})=2$ and $[A^t,B^t]\neq \mathrm{I}$ for all $t\in[0,1]$, $\{\widetilde{B^t}\}\interval$ is a path of parabolic elements starting at  $\widetilde{B^0}= \widetilde{B}\in \Par_0^s$, hence lies in $ \Par_0^s$. Moreover, the path $\{\widetilde{A^t}\widetilde{B^t}\}\interval$ lies in a conjugacy class in $\Hyp_l$,
as it is a path of hyperbolic elements of trace $z$ starting at $\widetilde{A}\widetilde{B} = \widetilde{C}\in \Hyp_l$. 
Then by Lemma \ref{lem_conjugacypath_const}, there is a path $\{\pm g_t\}\interval \subset \psl$ starting at $\pm \mathrm{I}$ such that $\pm C = \pm g_tA^tB^t g_t^{-1}$ for all $t\in [0,1]$. 
Let $\{\widetilde{g_t}\}\interval$ be the lift of $\{\pm g_t\}\interval$ in $\univcover$ starting at $\mathrm{I}$. Then for all $t\in [0,1]$, we have
$\widetilde{g_t}(\widetilde{A_t}\widetilde{B_t})\widetilde{g_t}^{-1} = \widetilde{C}$. 
Thus, the path $\big\{(\widetilde{g_t}\widetilde{A^t}\widetilde{g_t}^{-1},\widetilde{g_t}\widetilde{B^t}\widetilde{g_t}^{-1})\big\}\interval$ lies in $P_1^{-1}(\widetilde{C})$, connecting $(\widetilde{A},\widetilde{B})$ to a point $(\widetilde{A}', \widetilde{B'}) := (\widetilde{g_1}\widetilde{A^1}\widetilde{g_1}^{-1},\widetilde{g_1}\widetilde{B^1}\widetilde{g_1}^{-1})$ with $\kappa(\chi(\widetilde{A}', \widetilde{B'}))\neq 2$. Therefore, for any two points $(\widetilde{A_1}, \widetilde{B_1})$ and $(\widetilde{A_2}, \widetilde{B_2})$ in $P^{-1}(\widetilde{C})$, we can assume without loss of generality that $\kappa\big(\chi(\widetilde{A_1}, \widetilde{B_1})\big)\neq 2$ and $\kappa\big(\chi(\widetilde{A_2}, \widetilde{B_2})\big)\neq 2$.
Then, as in the proof of Lemma \ref{Hypfiber}, we can lift the straight-line path connecting $\chi(\widetilde{A_1}, \widetilde{B_1})$ and $\chi(\widetilde{A_2}, \widetilde{B_2})$ using Lemma \ref{lem_plpchar}, and the rest of the proof follows verbatim from Lemma \ref{Hypfiber}.\smallskip

 For $P_2$, for every $(\widetilde{A}, \widetilde{B})\in P_2^{-1}(\widetilde{C})$, $|Tr(\widetilde{A})|<2$ and $|Tr(\widetilde{C})|>2$ together imply $\kappa\big(\chi(\widetilde{A}, \widetilde{B})\big)\neq 2$. Therefore, the proof of the lemma follows verbatim from Lemma \ref{Hypfiber}.
\end{proof}
\begin{proof}[Proof of Lemma \ref{edgecon2}]
For (a), we show for the case where $s'=+$, $k =  0$ and $j = 1$, as the proof follows similarly for all the other cases. Let $E_1 = \Hyp_0\times \Par_0^s$, $E_2= \Ell_1\times \Par_0^s$ and $V = \Par_0^+\times \Par_0^s$. Let $(\widetilde{A},\widetilde{B})\in V\cap P^{-1}(\widetilde{C})$. Then $\chi(\widetilde{A},\widetilde{B})=(2,2,z)$ where $|z| = |Tr(\widetilde{C})|>2$. Let $\{\chi_t\}\interval = \{(2+ t,2,z)\}\interval$. As in the proof of Lemma \ref{edgecon1}, we may use Lemma \ref{lem_plpchar} to lift $\{\chi_t\}\interval$ to a path in $\big\{(\widetilde{A^t},\widetilde{B^t})\big\}\interval$ starting at $(\widetilde{A},\widetilde{B})$ and lying in $E_1$ for all $t\in (0,1]$, and use Lemma \ref{lem_conjugacypath_const} to find $\{\widetilde{g_t}\}\interval$ in $\univcover$ with $\widetilde{g_0} = \mathrm{I}$ such that for all $t\in [0,1]$, $\widetilde{g_t}(\widetilde{A_t}\widetilde{B_t})\widetilde{g_t}^{-1} = \widetilde{C}$. Finally, letting $\delta_1^t: = (\widetilde{g_t}\widetilde{A^t}\widetilde{g_t}^{-1},\widetilde{g_t}\widetilde{B^t}\widetilde{g_t}^{-1})$ for $t\in [0,1]$ gives the desired path $\{\delta_1^t\}\interval$ in $P^{-1}(\widetilde{C})$.\medskip

For (b), assume $P^{-1}(\widetilde{C})\cap V=\emptyset$. Since $P^{-1}(\widetilde{C})$ intersects both $E_1$ and $E_2$, the image $P\big(\{\widetilde{C}\}\times \Par_0^{-s}\big)$ intersects both $\mathbb{B}_m$ and $\mathbb{B}_n$, hence intersects the two connected components of $(\univcover\setminus Z)\setminus \Par_k^{s'}$. Since $P(\widetilde{C}\times \Par_0^{-s})\subset \univcover\setminus Z$ is a connected set intersecting both components of complement of $\Par_k^{s'}$ without intersecting $\Par_k^{s'}$, we obtain a contradiction.
\end{proof}

\begin{proof}[Proof of Lemma \ref{edgecon3}]
    Let $E=\mathbb{B}_m\times \Par_0^s$ and $E'=\mathbb{B}_n\times \Par_0^s$ for some $m,n\in \mathbb{Z}$ and $s\in \{\pm\}$. Assume such a sequence does not exist. Then there is an edge $E'' = \mathbb{B}_r\times \Par_0^s \subset I\times \Par_0^s$ for some $r\in \mathbb{Z}$, such that $E''\cap P^{-1}(\widetilde{C})=\emptyset$, and $\mathbb{B}_m$ and $\mathbb{B}_n$ are contained in the distinct components of the complement of $\mathbb{B}_r$ in $\univcover\setminus Z$. 
    Since $P^{-1}(\widetilde{C})$ intersects both $E$ and $E'$, $P\big(\widetilde{C}\times \Par_0^{-s}\big)$ intersects both $\mathbb{B}_m$ and $\mathbb{B}_n$, hence intersects the two connected components of $(\univcover\setminus Z)\setminus \mathbb{B}_r$. However, $P(\widetilde{C}\times \Par_0^{-s})$ is path-connected in $\univcover\setminus Z$ and does not intersect $\mathbb{B}_r$, which leads to a contradiction.
\end{proof}

\noindent
Viraj Joshi\\
Department of Mathematics\\  Texas A\&M University\\
College Station, TX 77843, USA\\
(virajajoshi@tamu.edu)
\\
\\
\noindent
Inyoung Ryu\\
Department of Mathematics\\  Texas A\&M University\\
College Station, TX 77843, USA\\
(riy520@tamu.edu)

\end{document}